\documentclass[11pt, a4paper, reqno]{amsart}
\usepackage{amsmath, amsthm, amscd, amsfonts, amssymb, graphicx, color}
\usepackage{float}

\usepackage{tikz}

\usepackage{multicol}
\usepackage{nicematrix}
\allowdisplaybreaks 
\usepackage{mathtools}
\usepackage[bookmarksnumbered, colorlinks, plainpages]{hyperref}
\hypersetup{colorlinks=true,linkcolor=blue, anchorcolor=green, citecolor=cyan, urlcolor=red, filecolor=magenta, pdftoolbar=true}
\usepackage{tikz-cd}
\usepackage[T1]{fontenc}%
\newtheorem{definition}{Definition}[section]
\newtheorem{theorem}[definition]{Theorem}
\newtheorem{lemma}[definition]{Lemma}

\newtheorem{proposition}[definition]{Proposition}
\theoremstyle{definition}
\newtheorem{remark}[definition]{Remark}
\newtheorem{note}[definition]{Note}
\newtheorem{example}[definition]{Example}

\newcommand{\diX}{\int^\oplus_{X}}
\newcommand{\dilX}{\int^{\oplus_\text{loc}}_{X}}
\newcommand{\dmu}{\mathrm{d} \mu (p)}
\newcommand{\la}{\left\langle}
\newcommand{\ra}{\right\rangle}
\newcommand\up[1]{\mbox{\raisebox{1pt}{\ensuremath{#1}}}} 
\usepackage[charter]{mathdesign}
\usepackage{comment}
\title[Direct integral and decompositions of locally Hilbert spaces]{Direct integral and decompositions of locally Hilbert spaces}
\author[Kulkarni]{Chaitanya J. Kulkarni}
\address{Chaitanya J. Kulkarni, Indian Institute of Science  Education and Research (IISER) Mohali, Knowledge City, S.A.S Nagar, Punjab 140306, India.}
\email{chaitanyakulkarni58@gmail.com}

\author[Pamula]{Santhosh Kumar Pamula}
\address{Santhosh Kumar Pamula, Indian Institute of Science  Education and Research (IISER) Mohali, Knowledge City, S.A.S Nagar, Punjab 140306, India.}
\email{santhoshkp@iisermohali.ac.in}

\date{\today}
\begin{document}

\maketitle

\begin{abstract}
   In this work, we introduce the concept of direct integral of locally Hilbert spaces by using the notion of locally standard measure space (analogous to standard measure space defined in the classical setup), which we obtain by considering a strictly inductive system of measurable spaces along with a projective system of finite measures. Next, we define a locally Hilbert space given by the direct integral of a family of locally Hilbert spaces. Following that we introduce decomposable locally bounded and diagonalizable  locally bounded operators. Further, we show that the class of diagonalizable locally bounded operators is an abelian locally von Neumann algebra, and this can be seen as the commutant of decomposable locally bounded operators. Finally, we discuss the following converse question: 
   
   For a locally Hilbert space $\mathcal{D}$ and an abelian locally von Neumann algebra $\mathcal{M}$, does there exist a locally standard measure space and a family of locally Hilbert spaces such that
   \begin{enumerate}
       \item the locally Hilbert space $\mathcal{D}$ is identified with the direct integral of family of locally Hilbert spaces;
       \item the abelian locally von Neumann algebra $\mathcal{M}$ is identified with the abelian locally von Neumann algebra of all diagonalizable locally bounded operators?
   \end{enumerate} 
   We answer this question affirmatively for a certain class of abelian locally von Neumann algebras.

\end{abstract}


\tableofcontents
\section{Introduction and Preliminaries} \label{sec;Introduction and Preliminaries}

The concept of the direct sum of Hilbert spaces can be generalized to the notion of the direct integral of Hilbert spaces. In the context of direct integrals, the discrete index set used in direct sum is replaced by a suitable measure space known as a standard measure space (see Note \ref{note; standard measure space}). The notion of the direct integral of Hilbert spaces is associated with an abelian von Neumann algebra, referred to as the algebra of diagonalizable operators. Conversely, given a separable Hilbert space $\mathcal{H}$ and an abelian von Neumann algebra in $\mathcal{B}(\mathcal{H})$, there exists a family of Hilbert spaces and a standard measure space such that the given Hilbert space can be identified with the direct integral of Hilbert spaces, whereas the abelian von Neumann algebra can be identified with the algebra of diagonalizable operators (see Theorem \ref{thm;dihs}). This process is referred to as the ``disintegration" of the Hilbert space. The concept of disintegration of Hilbert spaces (see \cite{DJ, Wils1, Wils2})  is crucial in the decomposition of a representation of a separable $C^\ast$-algebra into irreducible representations, or in the disintegration of a von Neumann algebra into factors (see \cite{OB1, DixC, DixV, KR2, Tak1} for further details on this topic). We strongly recommend articles \cite{Arveson} and \cite{Arveson2} for useful insight related to the theory of $C^\ast$-algebras. W. 
 B. Arveson proved in \cite{Arveson3} a long--standing open problem of whether all opertaor systems have sufficianetly many boundary representations. In particular, the author shows that the answer is affirmative if the operator system is separable. The proofs of several results in \cite{Arveson3} are based on the theory of disintegration of Hilbert spaces. 
 
In this article, we extend the concept of direct integral and disintegration to the context of locally Hilbert spaces in order to study non-commutative Choquet boundary problems in case of local operator systems. We begin by proposing a definition of the direct integral of locally Hilbert spaces, which involves the introduction of a generalization of the standard measure space. We call it as a locally standard measure space  (see Definition \ref{def; lsms}). In this framework, we also introduce the classes of decomposable locally bounded operators and diagonalizable locally bounded operators. We show that these classes individually form a locally von Neumann algebra. Moreover, the locally von Neumann algebra of all diagonalizable locally bounded operators is abelian. In the general case, it can be observed that the locally von Neumann algebra of all decomposable locally bounded operators is contained in the commutant of the locally von Neumann algebra of all diagonalizable locally bounded operators. However, from Part (a) of Theorem \ref{thm;DEC and DIAG LvNA} and Theorem \ref{thm; relation between M DEC and M DIAG}, both of these locally von Neumann algebras can be viewed as the projective limit of the same projective system of von Neumann algebras (also see Remark \ref{rem; inverse limits of decomposable vNA}). The set equality between these two locally von Neumann algebras is not completely known yet. We establish the set equality in Theorem \ref{thm; M DEC = M DIAG Commutant} for certain cases. Towards the end of the article, we address the question of disintegratiion of a locally Hilbert space. That is, given a locally Hilbert space $\mathcal{D}$ and an abelian locally von Neumann algebra $\mathcal{M}$ contained in the locally $C^\ast$-algebra of all locally bounded operators on $\mathcal{D}$, does there exist a locally standard measure space and a family of locally Hilbert spaces such that the locally Hilbert space $\mathcal{D}$ is identified with the direct integral of the family of locally Hilbert spaces, and the abelian locally von Neumann algebra $\mathcal{M}$ is identified with the abelian locally von Neumann algebra of diagonalizable locally bounded operators? (see Theorem \ref{thm;dlhs}). Throughout this article, we employ the concepts of inductive limits and projective limits of locally convex spaces along with the framework of direct integrals. For this, we use the results presented in \cite{BGP, AD, AG, MJ3} , as well as the results on the theory of locally von Neumann algebras discussed in \cite{MF, MJ1, MJ2}.

This article is organized into three sections. In Section  \ref{sec;Introduction and Preliminaries}, we review key definitions and results from the theory of direct integrals of Hilbert spaces, as well as from the theory of locally $C^\ast$-algebra and locally von Neumann algebra. In Section \ref{sec; Direct integrals}, we begin by introducing the notion of a locally standard measure space. We then present the definition of the direct integral of locally Hilbert spaces and prove that the direct integral of locally Hilbert spaces is again a locally Hilbert space. We also provide some examples of this concept. Subsequently, we introduce the notions of decomposable locally bounded operators and diagonalizable locally bounded operators, showing that each individually forms a locally von Neumann algebra. Further, in the later part of this section we explore the relation between the locally von Neumann algebra of decomposable locally bounded operators and diagonalizable locally bounded operators. In Section \ref{sec; Disintegration}, we examine the converse question of disintegrating a locally Hilbert space. To address this, we consider a particular class of abelian locally von Neumann algebras and with respect to such an algebra, we disintegrate a given locally Hilbert space. The details of this are presented in Theorem \ref{thm;dlhs}. 

\subsection{Direct integral of Hilbert spaces}
We recall a few definitions and results from the theory of direct integral of Hilbert spaces. The reader is directed to \cite{OB1, DixC, DixV, KR2, Tak1} for a comprehensive reading of this topic. Throughout this article, for terminology and notations related to the theory of direct integral of Hilbert spaces, we refer to Chapter 14 of \cite{KR2}. 

\begin{definition}\cite[Definition 14.1.1]{KR2} \label{def;dihs}
If $X$ is a $\sigma$-compact locally compact (Borel measure) space, $\mu$ is the completion of a Borel measure on $X$, and $\{ \mathcal{H}_p \}_{p \in X}$ is a family of separable Hilbert spaces indexed by points $p$ in $X$, we say that a separable Hilbert space $\mathcal{H}$ is the direct integral of $\{ \mathcal{H}_p \}$ over $(X, \mu)$ \Big(we write: $\mathcal{H} = \diX \mathcal{H}_p \, \dmu$ \Big) when, to each $x \in \mathcal{H}$, there corresponds a function $ p \mapsto x(p)$ on $X$ such that $x(p) \in \mathcal{H}_p$ for each $p$ and 
\begin{enumerate}
\item $p \mapsto \la x(p), y(p) \ra$ is $\mu$-integrable, when $x, y \in \mathcal{H}$ and
\begin{align*}
\la x, y \ra = \int_X \la x(p), y(p) \ra \, \dmu
\end{align*}
\item if $x_p \in \mathcal{H}_p$ for all $p$ in $X$ and $p \mapsto \la x_p, y(p) \ra$ is integrable for each $y \in \mathcal{H}$, then there is a $x \in \mathcal{H}$ such that 
$x(p) = x_p$ for almost every $p$. We say that $\diX \mathcal{H}_p \, \dmu$ and $p \mapsto x(p)$ are the (direct integral) decompositions of $\mathcal{H}$ and $x$ respectively.
\end{enumerate}
\end{definition}

\begin{note} \label{note; standard measure space}
    In view of Theorem \ref{thm;dihs}, our discussion gives a special attention to standard measure space $(X, \mu)$. That is, $X$ is a complete, separable, metric space and $\mu$ is a finite, positive measure on $X$.
\end{note} 

Next, we recall the definition of a decomposable and diagonalizable bounded operators on the direct integral of Hilbert spaces. 
\begin{definition}\cite[Definition 14.1.6]{KR2} \label{def;Debo}
Let $\mathcal{H}$ be the direct integral of Hilbert spaces $\{ \mathcal{H}_p \}$ over the standard measure space $(X, \mu)$. Then an operator $T$ in $\mathcal{B}(\mathcal{H})$ is said to be:
\begin{enumerate}
    \item \label{def;Decbo} {\it decomposable}, if there is a family  $\{ T_p \in \mathcal{B}(\mathcal{H}_p) \}_{p \in X}$ such that for each $x \in \mathcal{H}$, we have 
\begin{align*}
Tx(p) = T_p x(p)
\end{align*}
for almost every $p \in X$. Subsequently, $T$ is denoted by $\int^\oplus_X T_p \, \dmu$. Moreover, the norm of $\| T \|$ is defined by
\begin{equation} \label{eq;norm of T}
\| T \| := \text{ess sup} \big \{ \| T_p \| \; : \; p \in X \big \}.
\end{equation}
\item \label{def;Diagbo}{\it diagonalizable}, if $T$ is decomposable and there exists a function $f \in \text{L}^\infty(X, \mu)$ such that for each $x \in \mathcal{H}$, we have 
\begin{align*}
Tx(p) = f(p) x(p)
\end{align*}
for almost every $p \in X$.
\end{enumerate}  
\end{definition}

The following theorem provides a structure on the set of all decomposable and the set of all diagonalizable operators, and describes the relationship between them. 

\begin{theorem}\cite[Theorem 14.1.10]{KR2} \label{thm;DeDibo vNA}
Let $\mathcal{H} = \diX \mathcal{H}_p \, \dmu$. Then the set of all decomposable operators is a von Neumann algebra with the abelian commutant coinciding with the set of all diagonalizable operators.  
\end{theorem}



The disintegration (or decomposition) of a given Hilbert space with respect to an abelian von Neumann algebra is given below. 

\begin{theorem}\cite[Part II, Chapter 6, Theorem 2]{DixV} \label{thm;dihs}
Let $\mathcal{H}$ be a separable Hilbert space, and $\mathcal{M}$ be an abelian von Neumann algebra in $\mathcal{B}(\mathcal{H})$. Then there exists a standard measure space $(X, \mu)$, a family of separable Hilbert spaces $\{ \mathcal{H}_p \}_{p \in X}$, and an isomorphism of $\mathcal{H}$ onto $\diX \mathcal{H}_p \, \dmu$ which transforms $\mathcal{M}$ into the algebra of diagonalizable operators.
\end{theorem}


Our aim is to establish suitable notions like direct integrals, decomposable operators, diagonaliazable operators, etc (that are described earlier) in locally Hilbert space setting. Before that, we turn our attention towards basic notations, terminologies and necessary concepts from the theory of locally C*-algebra. The following definitions are mainly drawn from \cite{AG}, and for a detailed discussion on various topics of locally C*-algebra one can see \cite{BGP, AD, AG2, MJ1, MJ2, MJ3}.

\subsection{Locally Hilbert space} 
A locally Hilbert space is the inductive limit of a strictly inductive system (or an upward filtered family) of Hilbert spaces. The formal definition is given below:

\begin{definition}\cite[Subsection 1.3]{AG} \label{def;sis}
Let $( \mathcal{H}_\alpha,\; \la \cdot, \cdot \ra_{\mathcal{H}_\alpha} )_{\alpha \in \Lambda}$ be a net of Hilbert spaces. Then $\mathcal{E} = \{\mathcal{H}_\alpha \}_{\alpha \in \Lambda}$ is said to be a strictly inductive system (or an upward filtered family) of Hilbert spaces if:
\begin{enumerate}
\item $(\Lambda, \leq)$ is a directed partially ordered set;
\item for each $\alpha, \beta \in \Lambda$ with $\alpha \leq \beta \in \Lambda$ we have $\mathcal{H}_\alpha \subseteq \mathcal{H}_\beta$;
\item for each $\alpha, \beta \in \Lambda$ with $\alpha \leq \beta \in \Lambda$ the inclusion map $J_{\beta, \alpha} : \mathcal{H}_\alpha \rightarrow \mathcal{H}_\beta$ is isometric, that is,
$\la u, v \ra_{\mathcal{H}_\alpha} = \la u, v \ra_{\mathcal{H}_\beta}$ for all $u, v \in \mathcal{H}_\alpha$.
\end{enumerate}
\end{definition}

It is evident from equation 1.13 of \cite{AG} that for a strictly inductive system $\mathcal{E} = \{\mathcal{H}_\alpha \}_{\alpha \in \Lambda}$ of Hilbert spaces, the inductive limit denoted by $\varinjlim\limits_{\alpha \in \Lambda} \mathcal{H}_\alpha$, and it is given by
\begin{equation} \label{eq;lhs}
\varinjlim_{\alpha \in \Lambda} \mathcal{H}_\alpha = \bigcup_{\alpha \in \Lambda} \mathcal{H}_\alpha.
\end{equation}

\begin{definition}\cite[Subsection 1.3]{AG} \label{def;lhs}
A locally Hilbert space $\mathcal{D}$, is defined as the inductive limit of some strictly inductive system $\mathcal{E} = \{\mathcal{H}_\alpha \}_{\alpha \in \Lambda}$ of Hilbert spaces. 
\end{definition}

It is worth to point out that in the work of \cite{BGP, AD}, the authors used the term ``quantized domain" in place of locally Hilbert space. In particular, $\mathcal{D}$ is the quantized domain given by the family $\mathcal{E} = \{\mathcal{H}_\alpha \}_{\alpha \in \Lambda}$. Moreover, if $\Lambda = \mathbb{N}$, then $\mathcal{E}$ is countable and $\mathcal{D}$ is called a \textit{quantized Frechet domain}. For details, see Definition 2.3 of \cite{BGP}.


\begin{example} \label{ex;lhs}
Let $\{e_{n}: n \in \mathbb{N}\}$ be a Hilbert basis of $\ell^2(\mathbb{N}).$ For each $k \in \mathbb{N}$, define the closed (in fact, finite dimensional) subspace  $\mathcal{H}_k := \text{span} \{ e_1, e_2,...,e_k \}$. It follows that $\mathcal{H}_{m} \subseteq \mathcal{H}_{n}$ for $m \leq n$. In other words, the family $\mathcal{E} = \{\mathcal{H}_n\}_{n \in \mathbb{N}}$ forms a strictly inductive system of Hilbert spaces. The inductive limit of the strictly inductive system $\mathcal{E} = \{\mathcal{H}_n\}_{n \in \mathbb{N}}$ is given by
\begin{align*}
\mathcal{D} = \varinjlim_{n \in \mathbb{N}} \mathcal{H}_n = \bigcup_{n \in \mathbb{N}} \mathcal{H}_n.
\end{align*}
Hence $\mathcal{D}$ is the locally Hilbert space (or the quantized domain) given by $\mathcal{E}$. Moreover, we obtain $\overline{\mathcal{D}} = \ell^2(\mathbb{N})$. For a more general construction, one can see Example 2.9 of \cite{BGP}. 
\end{example}

From now onwards, the notation $\mathcal{D} = \varinjlim\limits_{\alpha \in \Lambda} \mathcal{H}_\alpha$ will indicate that $\mathcal{D}$ is the locally Hilbert space given by a strictly inductive system (or an upward filtered family) $ \mathcal{E} = \{\mathcal{H}_\alpha\}_{\alpha \in \Lambda}$ of Hilbert spaces. The reader may note that in article \cite{BGP}, the authors used the notation $\{ \mathcal{H}; \mathcal{E}; \mathcal{D} \}$, whereas in the work of \cite{MJ3}, the author used the notation $\mathcal{D}_\mathcal{E}$. However, we stick to the symbol $\mathcal{D}$ while denoting locally Hilbert space given by some strictly inductive system of Hilbert spaces. Next, we recall the notion of a locally bounded operator between locally Hilbert spaces.

\begin{definition}\cite[Section 2]{BGP} \label{def;lbo}
Let $\mathcal{D} = \varinjlim\limits_{\alpha \in \Lambda} \mathcal{H}_\alpha = \bigcup\limits_{\alpha \in \Lambda} \mathcal{H}_\alpha$ and $\mathcal{O} = \varinjlim\limits_{\alpha \in \Lambda} \mathcal{K}_\alpha = \bigcup\limits_{\alpha \in \Lambda} \mathcal{K}_\alpha$ be two locally Hilbert spaces. A linear map $T : \mathcal{D} \rightarrow \mathcal{O}$ is said to be a locally bounded operator if 
\begin{equation*}
    T(\mathcal{H}_{\alpha}) \subseteq \mathcal{H}_{\alpha},\; \; \; T(\mathcal{H}_{\alpha}^{\bot}\cap \mathcal{D}) \subseteq \mathcal{H}_{\alpha}^{\bot}\cap \mathcal{D} \; \; \; \text{and}\; T\big|_{\mathcal{H}_{\alpha}} \in \mathcal{B}(\mathcal{H}_{\alpha})\; 
\end{equation*}
\text{for each}\; $\alpha \in \Lambda.$
\end{definition}
Let $\mathcal{D}$ and $\mathcal{O}$ be the locally Hilbert spaces given by the strictly inductive systems of Hilbert spaces $\mathcal{E} = \{ \mathcal{H}_\alpha \}_{\alpha \in \Lambda}$ and $\mathcal{F} = \{ \mathcal{H}_\alpha \}_{\alpha \in \Lambda}$ respectively. We denote the collection of all locally bounded operators from $\mathcal{D}$ to $\mathcal{O}$ by the notation $C^*_{\mathcal{E}, \mathcal{F}}(\mathcal{D}, \mathcal{O})$. In particular, $C^*_{\mathcal{E}, \mathcal{E}}(\mathcal{D}, \mathcal{D}) = C^*_\mathcal{E}(\mathcal{D})$.
\begin{example} \label{ex;lbo}
Let $\mathcal{D}$ be the locally Hilbert space given by the strictly inductive system $\mathcal{E}$ as in Example \ref{ex;lhs}. Define a operator $T : \mathcal{D} \rightarrow \mathcal{D}$ by using the following matrix
\begin{align*}
T= \begin{pmatrix}
1      & 0      & 0      & \cdots & 0      & \cdots \\
0      & 2      & 0      & \cdots & 0      & \cdots \\
0      & 0      & 3      & \cdots & 0      & \cdots \\
\vdots & \vdots & \vdots & \ddots & \vdots & \vdots\\
0      & 0      & 0      & \cdots & m      &  \cdots\\
0      & 0      & 0      & \cdots & 0      & \ddots
\end{pmatrix}.
\end{align*}
This gives $Te_n = n e_n$ for each $n \in \mathbb{N}$. Clearly, $T$ is a densely defined closed operator. For each $n \in \mathbb{N}$, $\mathcal{H}_{n}= span\{e_{1}, e_{2}, \cdots, e_{n}\}$ is a reducing subspace of $T$ and the restriction $T\big|_{\mathcal{H}_{n}} \in \mathcal{B}(\mathcal{H}_{n})$ is given by  
\begin{align*}
T\big|_{\mathcal{H}_{n}} = 
\begin{pmatrix}
1      & 0      & 0      & \cdots & 0  \\
0      & 2      & 0      & \cdots & 0  \\
0      & 0      & 3      & \cdots & 0  \\  
\vdots & \vdots & \vdots & \ddots & \vdots \\
0      & 0      & 0      & \cdots & n      
\end{pmatrix}.
\end{align*}
Therefore, $T$ is a locally bounded operator, but not bounded.
\end{example}

\begin{remark}\cite[Subection 1.4]{AG} \label{rem;lbo}
Let $T \in C^*_{\mathcal{E}, \mathcal{F}}(\mathcal{D}, \mathcal{O})$, and for each $\alpha \in \Lambda$ take $T_\alpha = T\big|_{\mathcal{H}_{\alpha}}$. For a fixed $\alpha \in \Lambda$, we denote the inclusion maps by the notations $J_{\mathcal{D}, \alpha} : \mathcal{H}_\alpha \rightarrow \mathcal{D}$ and $J_{\mathcal{O}, \alpha} : \mathcal{K}_\alpha \rightarrow \mathcal{O}$. Then the collection $\{ T_\alpha \}_{\alpha \in \Lambda}$ satisfies the following properties:
\begin{enumerate}
\item for each $\alpha \in \Lambda$, $T_\alpha \in \mathcal{B}(\mathcal{H}_\alpha, \mathcal{K}_\alpha)$ and $T J_{\mathcal{D}, \alpha} = J_{\mathcal{O}, \alpha} T_\alpha$;
\item for $\alpha, \beta \in \Lambda$ with $\alpha \leq \beta$, we get $T^*_\beta\big|_{\mathcal{K}_{\alpha}} = T^*_\alpha$.
\end{enumerate}
\end{remark}
In view of Remark \ref{rem;lbo} and following the notations of \cite{AG}, we say that every $T \in C^{\ast}_{\mathcal{E}}(\mathcal{D})$ can be seen as a projective (or inverse) limit of the net $\{T_{\alpha}\}_{\alpha \in \Lambda}$ of bounded operators. That is, 
\begin{equation} \label{eq; inverese limit of bounded operators}
    T = \varprojlim\limits_{\alpha \in \Lambda} T_{\alpha}.
\end{equation}
Now we are in a position to discuss the notion of locally $C^{\ast}$-algebra. For a comprehensive study of such algebras and local completely positive maps, one can see \cite{BGP, AG, AD, MJ3, AI} and references therein. 
\subsection{Locally $C^{\ast}$-algebra} Let $\mathcal{A}$ be a unital $\ast$-algebra. A seminorm $p$ on $\mathcal{A}$ is said to be a $C^{\ast}$-seminorm, if 
\begin{multicols}{2}
    \begin{enumerate}
        \item $p(1_\mathcal{A}) = 1$\\
        \item $p(ab) \leq p(a) p(b)$
        \item $p(a^\ast) = p(a)$\\
        \item $p(a^*a) = p(a)^2,$
    \end{enumerate}
\end{multicols}
\noindent for all $a,b \in \mathcal{A}$. Let $(\Lambda, \leq)$ be a directed poset and $\mathcal{P} := \{ p_\alpha \;  : \;   \alpha \in \Lambda \}$ be a family of $C^{\ast}$-seminorms defined on a the $\ast$-algebra $\mathcal{A}$. Then $\mathcal{P}$ is called a upward filtered family if for each $a \in \mathcal{A}$, we have  $p_\alpha(a) \leq p_\beta(a)$, whenever $\alpha \leq \beta$.

\begin{definition} \cite[Definition 2.1]{BGP} \label{def;lca}
A unital $*$-algebra $\mathcal{A}$ which is complete with respect to the locally convex topology generated by an upward filtered family $\{ p_\alpha \;  : \;  \alpha \in \Lambda \}$ of C*-seminorms is called a locally $C^{\ast}$-algebra.
\end{definition}

It is well known that every locally $C^{\ast}$-algebra can be realized as the projective limit of a projective system of $C^{\ast}$-algebras. The construction of such projective system is given in \cite{BGP, AG}. However, we recall a few points here: Let $\mathcal{A}$ be a locally $C^{\ast}$-algebra. Then for each $\alpha \in \Lambda$, take $\mathcal{I}_\alpha := \{ a \in \mathcal{A} \; : \; p_\alpha(a) = 0 \}$, which is a two-sided closed ideal in $\mathcal{A}$, then $\mathcal{A}_\alpha : = {\mathcal{A}}\!/\!{\mathcal{I}_\alpha}$ is a $C^{\ast}$-algebra with respect to the $C^{\ast}$-norm induced by $p_\alpha$. Whenever $\alpha \leq \beta$, since $P_\alpha(a) \leq p_\beta(a)$ for each $a \in \mathcal{A}$, there is a $C^\ast$-homomorphism (surjective) $\pi_{\alpha, \beta} : \mathcal{A}_\beta \rightarrow \mathcal{A}_\alpha$ given by $\pi_{\alpha, \beta}(a + \mathcal{I}_\beta) = a + \mathcal{I}_\alpha$. This shows that $\left ( \{ \mathcal{A}_\alpha \}_{\alpha \in \Lambda},  \{ \pi_{\alpha, \beta} \}_{\alpha \leq \beta}\right )$ forms a projective system of $C^{\ast}$-algebras. There is a canonical projection map $\pi_\alpha : \mathcal{A} \rightarrow \mathcal{A}_\alpha$ for each $\alpha \in \Lambda$ satisfying 
\begin{align}\label{eq;ps}
\pi_{\alpha, \beta} \circ \pi_\beta = \pi_\alpha,\; \text{whenever} \; \alpha \leq \beta.
\end{align} 
Equivalently, the pair $(\mathcal{A}, \{ \pi_\alpha \}_{\alpha \in \Lambda})$ is compatible with the projective system  $\left ( \{ \mathcal{A}_\alpha \}_{\alpha \in \Lambda},  \{ \pi_{\alpha, \beta} \}_{\alpha \leq \beta} \right )$ \big (see Subsection 1.1 of \cite{AG} \big ). Further, we consider the projective limit of the projective system $\left ( \{ \mathcal{A}_\alpha \}_{\alpha \in \Lambda},  \{ \pi_{\alpha, \beta} \}_{\alpha \leq \beta} \right )$, which is given by \big (refer Subsection 1.1 of \cite{AG} \big)
\begin{align}\label{eq;pl}
\varprojlim\limits_{\alpha \in \Lambda} \mathcal{A}_\alpha := \left \{ \{x_\alpha\}_{\alpha \in \Lambda} \in \prod_{\alpha \in \Lambda} \mathcal{A}_\alpha ~~ : ~~ \pi_{\alpha, \beta}(x_\beta) = x_\alpha, ~~ \text{whenever} ~~ \alpha \leq \beta \right \}.
\end{align}
Here the topology on $\varprojlim\limits_{\alpha \in \Lambda} \mathcal{A}_\alpha$ is the weakest locally convex topology that makes each linear map $\phi_\alpha : \varprojlim\limits_{\alpha \in \Lambda} \mathcal{A}_\alpha \rightarrow \mathcal{A}_\alpha$ defined by $\phi_\alpha(\{x_\alpha\}_{\alpha \in \Lambda}) := x_\alpha$ is continuous. It is known as the projective limit topology. Since $\mathcal{A}_\alpha$ is complete for each $\alpha \in \Lambda$, the projective limit $\varprojlim\limits_{\alpha \in \Lambda} \mathcal{A}_\alpha$ is complete with respect to the projective limit topology \big (see Subsection 1.1 of \cite{AG} \big ). Moreover, the pair 
$\big(\varprojlim\limits_{\alpha \in \Lambda} \mathcal{A}_\alpha, \{ \phi_\alpha \}_{\alpha \in \Lambda} \big)$ is compatible with the projective system $\left ( \{ \mathcal{A}_\alpha \}_{\alpha \in \Lambda},  \{ \pi_{\alpha, \beta} \}_{\alpha \leq \beta} \right )$ in the sense of Equation \eqref{eq;ps}. It follows that the map $\phi : \mathcal{A} \rightarrow \varprojlim\limits_{\alpha \in \Lambda} \mathcal{A}_\alpha$ given by $a \mapsto \{a + \mathcal{I}_\alpha\}_{\alpha \in \Lambda}$ is a continuous linear map satisfying,
\begin{equation}\label{eq;c}
    \pi_\alpha = \phi_\alpha \circ \phi,\; \text{for every}\; \alpha \in \Lambda.
\end{equation}
If there is any other continuous linear map $\phi^\prime : \mathcal{A} \rightarrow \varprojlim\limits_{\alpha \in \Lambda} \mathcal{A}_\alpha$ satisfying, $\pi_\alpha = \phi_\alpha \circ \phi^\prime$ for every $\alpha \in \Lambda$, then for $a \in \mathcal{A}$, we have $\phi^\prime(a) \in \varprojlim\limits_{\alpha \in \Lambda} \mathcal{A}_\alpha$ and $\phi_\alpha \big (  \phi^\prime(a) \big ) = \pi_\alpha(a) = \phi_\alpha \big (  \phi(a) \big )$ for every $\alpha \in \Lambda$. Since each $\pi_\alpha$ is a canonical projection, we get $\phi^\prime(a) = \phi(a)$ for every $a \in \mathcal{A}$. In conclusion, there is a unique continuous linear map $\phi : \mathcal{A} \rightarrow \varprojlim\limits_{\alpha \in \Lambda} \mathcal{A}_\alpha$ satisfying Equation \eqref{eq;c}.
 
From here onwards, we write $\mathcal{A} = \varprojlim\limits_{\alpha \in \Lambda} \mathcal{A}_\alpha$ with the understanding that there exists such a unique continuous linear map $\phi$ satisfying Equation \eqref{eq;c}. 


\begin{remark} \label{rem;pl,sh}
In this remark, we show that the projective limit of a projective system of unital $C^\ast$-algebras is a unital locally $C^\ast$-algebra. 
\begin{enumerate}
\item  If we start with a projective system $\left ( \{ \mathcal{B}_\alpha \}_{\alpha \in \Lambda},  \{ \psi_{\alpha, \beta} \}_{\alpha \leq \beta} \right )$ of unital $C^{\ast}$-algebras, where $\psi_{\alpha, \beta} : \mathcal{B}_\beta \rightarrow \mathcal{B}_\alpha$ is a unital surjective $C^\ast$-homomorphism, whenever $\alpha \leq \beta$, then we get the inverse limit $\varprojlim\limits_{\alpha \in \Lambda} \mathcal{B}_\alpha$ of the projective system $\left ( \{ \mathcal{B}_\alpha \}_{\alpha \in \Lambda},  \{ \psi_{\alpha, \beta} \}_{\alpha \leq \beta} \right )$ defined as in Equation \eqref{eq;pl}. Further, for any two elements $\{x_\alpha\}_{\alpha \in \Lambda}$ and $\{y_\alpha\}_{\alpha \in \Lambda}$ in $\varprojlim\limits_{\alpha \in \Lambda} \mathcal{B}_\alpha$ , we define 
\begin{align*}
\{x_\alpha\}_{\alpha \in \Lambda} \cdot \{y_\alpha\}_{\alpha \in \Lambda} := \{x_\alpha y_\alpha\}_{\alpha \in \Lambda} \; \; \text{and} \; \; \{x_\alpha\}_{\alpha \in \Lambda}^\ast := \{x_\alpha^\ast\}_{\alpha \in \Lambda}.
\end{align*}
We see that if $\{x_\alpha\}_{\alpha \in \Lambda},\; \{y_\alpha\}_{\alpha \in \Lambda} \in \varprojlim\limits_{\alpha \in \Lambda} \mathcal{B}_{\alpha} $ then $\{x_\alpha y_\alpha\}_{\alpha \in \Lambda} \in \varprojlim\limits_{\alpha \in \Lambda} \mathcal{B}_{\alpha}$ and $\{x^{\ast}_{\alpha}\} \in \varprojlim\limits_{\alpha \in \Lambda}\mathcal{B}_{\alpha}$ since
\begin{align*}
    \psi_{\alpha, \beta} (x_\beta y_\beta) = \psi_{\alpha, \beta} (x_\beta) \psi_{\alpha, \beta} (y_\beta) =  x_\alpha y_\alpha
\end{align*}
and 
\begin{equation*}
    \psi_{\alpha, \beta} (x_\beta^\ast) = (\psi_{\alpha, \beta} (x_\beta))^\ast = x_\alpha^\ast, \; \; \text{whenever} \; \alpha \leq \beta.
\end{equation*}
Next, for each $\beta \in \Lambda$, define a $C^\ast$-seminorm on $\varprojlim\limits_{\alpha \in \Lambda} \mathcal{B}_\alpha$ by 
\begin{equation} \label{Eq: seminorm}
    p_\beta \big (\{x_\alpha\}_{\alpha \in \Lambda} \big ) := \| x_\beta \|_{\mathcal{B}_\beta},\; \text{for all}\; \{x_\alpha\}_{\alpha \in \Lambda} \in  \varprojlim\limits_{\alpha \in \Lambda} \mathcal{B}_\alpha.
\end{equation}
Since $\psi_{\alpha, \beta}$ is a contraction (whenever $\alpha \leq \beta$), for each $\{x_\alpha\}_{\alpha \in \Lambda} \in \varprojlim\limits_{\alpha \in \Lambda} \mathcal{B}_\alpha$, we get 
\begin{align*}
p_\alpha \big  (\{x_\alpha\}_{\alpha \in \Lambda} \big  ) = \big  \| x_\alpha \big  \|_{\mathcal{B}_\alpha} = \big  \| \psi_{\alpha, \beta} (x_\beta) \big  \|_{\mathcal{B}_\alpha} \leq \big  \| x_\beta \big  \|_{\mathcal{B}_\beta} = p_\beta \big  (\{x_\alpha\}_{\alpha \in \Lambda} \big ).
\end{align*} 
That is, $\{ p_\alpha \}_{\alpha \in \Lambda}$ is an upward filtered family of $C^\ast$-seminorms on $\varprojlim\limits_{\alpha \in \Lambda} \mathcal{B}_\alpha$. We know that the projective limit topology on $\varprojlim\limits_{\alpha \in \Lambda} \mathcal{B}_\alpha$ is the weakest locally convex topology that makes the $\ast$-homomorphism $\psi_\beta : \varprojlim\limits_{\alpha \in \Lambda} \mathcal{B}_\alpha \rightarrow \mathcal{B}_\beta$ given by $\psi_\beta \big  (\{ x_\alpha \}_{\alpha \in \Lambda} \big  ) = x_\beta$ continuous for every $\beta \in \Lambda$. Then from Equation \eqref{Eq: seminorm}, one can see that the locally convex topology induced by the family $\{p_\alpha\}_{\alpha \in \Lambda}$ coincides with the projective limit topology on $\varprojlim\limits_{\alpha \in \Lambda} \mathcal{B}_\alpha$. Since each $\mathcal{B}_\alpha$ is a complete space, we get that the inverse limit $\varprojlim\limits_{\alpha \in \Lambda} \mathcal{B}_\alpha$ is also a complete space \big  (see Subsection 1.1 of \cite{AG} \big  ). Hence, in view of Definition \ref{def;lca}, we obtain that the locally convex space $\varprojlim\limits_{\alpha \in \Lambda} \mathcal{B}_\alpha$ is a unital locally $C^\ast$-algebra.

\item From the above discussion, we get that $\varprojlim\limits_{\alpha \in \Lambda} \mathcal{A}_\alpha$ defined in Equation \eqref{eq;pl} is a locally $C^\ast$-algebra and the map $\phi$ appeared in Equation \eqref{eq;c} is a $\ast$-homomorphism.
\end{enumerate}
\end{remark}
\begin{example} \label{ex;lca}
Let $\mathcal{D}$ be the locally Hilbert space given by the strictly inductive system $\mathcal{E}$ as in Example \ref{ex;lhs}. Here $\Lambda = \mathbb{N}$ and consider the family $\big  \{ \mathcal{B}(\mathcal{H}_n) \big  \}_{n \in \mathbb{N}}$ of $C^{\ast}$-algebras. For $m \leq n$, define a map $\phi_{m, n} : \mathcal{B}(\mathcal{H}_n) \rightarrow \mathcal{B}(\mathcal{H}_m)$ by
\begin{align*}
 \phi_{m, n}(S) = J_{n,m}^* S J_{n,m},
\end{align*}
where $J_{n,m} : \mathcal{H}_m \rightarrow \mathcal{H}_n$ is the inclusion map. Then the family $\left ( \{ \mathcal{B}(\mathcal{H}_n) \}_{n \in \mathbb{N}},  \{\phi_{m,n} \}_{m \leq n} \right )$ forms a projective system of $C^{\ast}$-algebras. Now consider the projective limit $\varprojlim\limits_{n \in \mathbb{N}} \mathcal{B}(\mathcal{H}_n)$ that is defined as in Equation \eqref{eq;pl} along with the family $\{ \phi_n \}_{n \in \mathbb{N}}$, where $\phi_n : \varprojlim\limits_{n \in \mathbb{N}} \mathcal{B}(\mathcal{H}_n) \rightarrow \mathcal{B}(\mathcal{H}_n)$ given by $\{T_n\}_{n \in \mathbb{N}} \mapsto T_n$ is a surjective $\ast$-homomorphism satisfying $\phi_{m,n} \circ \phi_n = \phi_m$, whenever $m \leq n$. This means that the pair $\left(  \varprojlim\limits_{n \in \mathbb{N}} \mathcal{B}(\mathcal{H}_n), \{ \phi_n \}_{n \in \mathbb{N}} \right)$ is compatible with the projective system $\left ( \{ \mathcal{B}(\mathcal{H}_n) \}_{n \in \mathbb{N}},  \{\phi_{m,n} \}_{m \leq n} \right )$. On the other hand, for every $n \in \mathbb{N}$, there is a surjective $\ast$-homomorphism $\psi_n : C^*_\mathcal{E}(\mathcal{D}) \rightarrow \mathcal{B}(\mathcal{H}_n)$ given by $T \mapsto T\big|_{\mathcal{H}_n}$ satisfying, $\phi_{m,n} \circ \psi_n = \psi_m$, whenever $m \leq n$. Equivalently, the pair $\left(C^*_\mathcal{E}(\mathcal{D}) , \{ \psi_n \}_{n \in \mathbb{N}} \right )$ is also compatible with the projective system $\left ( \{ \mathcal{B}(\mathcal{H}_n) \}_{n \in \mathbb{N}},  \{\phi_{m,n} \}_{m \leq n} \right )$.  Thus there exists a unique unital $\ast$-homomorphism \big (refer Subsection 1.1 of \cite{AG} \big ) $\psi : C^*_\mathcal{E}(\mathcal{D}) \rightarrow \varprojlim\limits_{n \in \mathbb{N}} \mathcal{B}(\mathcal{H}_n)$ given by $\psi(T) = \big \{ T\big|_{\mathcal{H}_n} \big \}_{n \in \mathbb{N}}$ such that $\psi_n = \phi_n \circ \psi$. In other words $C^*_\mathcal{E}(\mathcal{D}) = \varprojlim\limits_{n \in \mathbb{N}} \mathcal{B}(\mathcal{H}_n)$. Hence, $C^*_\mathcal{E}(\mathcal{D})$ is a unital locally $C^{\ast}$-algebra. The following commuting diagram will summarize Example \ref{ex;lca}.
\begin{center}
\begin{tikzcd}[sep=huge]
& C^*_\mathcal{E}(\mathcal{D}) 
\arrow[dl, "\psi_m"']  
\arrow[rr, "\psi"] 
\arrow[dr, "\psi_n"] & & \varprojlim\limits_{n \in \mathbb{N}} \mathcal{B}(\mathcal{H}_n) 
\arrow[dl, "\phi_n"'] \arrow[dr, "\phi_m"] \\
\mathcal{B}(\mathcal{H}_m)  & &  \mathcal{B}(\mathcal{H}_n) \arrow[to path={node[midway,scale=2.5] {$\circlearrowright$}}] \arrow[ll, "\phi_{m,n}"] \arrow[rr, "\phi_{m,n}"'] & &\mathcal{B}(\mathcal{H}_m)
\end{tikzcd}
\end{center}
\end{example}



\subsection{Locally von Neumann algebra} In this subsection, we recall some results from the notion of locally von Neumann algebra. Let $\mathcal{D}$ be the locally Hilbert space given by a strictly inductive system $\mathcal{E} = \{\mathcal{H}_{\alpha}\}_{\alpha \in \Lambda}$ of Hilbert spaces. If $u, v \in \mathcal{D}$, then $u, v \in \mathcal{H}_\gamma$ for some $\gamma \in \Lambda$ and we define
\begin{equation*}
q_{u}(T) := \| Tu \|_{\mathcal{H}_\gamma} \; \; \text{and} \; \;q_{u,v}(T) := \left | \langle u, Tv \rangle_{\mathcal{H}_\gamma} \right | \; \; \text{for all} \; T \in C^*_\mathcal{E}(\mathcal{D}).
\end{equation*}
Thus $q_u$ and $q_{u, v}$ are $C^\ast$-seminorms. Then
\begin{enumerate}
\item[(a)] (SOT) {\it strong operator topology} on $C^*_\mathcal{E}(\mathcal{D})$ is the locally convex topology generated by the family $\{q_u ~~ : ~~ u \in \mathcal{D} \}$ of $C^\ast$-seminorms;
\item[(b)] (WOT) {\it weak operator topology} on $C^*_\mathcal{E}(\mathcal{D})$ is the locally convex topology generated by the family $\{q_{u,v} ~~ : ~~ u, v \in \mathcal{D} \}$ of $C^\ast$-seminorms.
\end{enumerate}
For a detailed introduction to locally von Neumann algebras, a reader is directed to \cite{MF, MJ1, MJ2}.

\begin{definition}\cite[Definition 3.7]{MJ1}
Let $\mathcal{D}$ be a locally Hilbert space given by a strictly inductive system $\mathcal{E}$ of locally Hilbert spaces. Then a locally von Neumann algebra is a strongly closed unital locally $C^{\ast}$-algebra contained in $C^*_\mathcal{E}(\mathcal{D})$.
\end{definition}

Let $\mathcal{M} \subseteq C^*_\mathcal{E}(\mathcal{D})$. Consider the set $\mathcal{M}^\prime := \{ T \in C^*_\mathcal{E}(\mathcal{D}) ~~:~~ TS = ST ~~ \text{for all} ~~ S \in \mathcal{M} \}$ which is called as the commutant of $\mathcal{M}$. We denote $(\mathcal{M}^\prime)^\prime$ by the notation $\mathcal{M}^{\prime \prime}$. The following theorem proved in \cite{MJ1} is the double commutant theorem in the setting of locally von Neumann algebra. 

\begin{theorem}\cite[Theorem 3.6]{MJ1}
Let $\mathcal{M} \subseteq C^*_\mathcal{E}(\mathcal{D})$ be a locally $C^{\ast}$-algebra containing the identity operator on $\mathcal{D}$. Then the following statements are equivalent:
\begin{enumerate}
\item $ \mathcal{M} = \mathcal{M}^{\prime \prime}$;
\item $\mathcal{M}$ is weakly closed;
\item $\mathcal{M}$ is strongly closed.
\end{enumerate}
\end{theorem}


\section{Direct integral of locally Hilbert spaces} \label{sec; Direct integrals}
Motivated from the theory of direct integral of Hilbert spaces, we propose an approach to define the notion of direct integral in locally Hilbert space setting. Before we define this formally, it is required to understand the terminology analogous to the standard measure space. In view of this, we introduce the concept of ``locally standard measure space" below.

\subsection{Locally standard measure space}
First, we recall the definition of the standard measure space from Section 4.4.1 of \cite{OB1}. A Borel measure space $X$ with a finite positive measure is called a standard measure space, if the Borel structure on $X$ is defined by a complete, separable, metric space (also see Chapter 4, Section 8 of \cite{Tak1}). For example, consider the Borel measure space $X = [0,1]$ with the Lebesgue measure. The following definition is inspired by the concept of a strictly inductive system of measurable spaces discussed in \cite{AG2}.

\begin{definition} \label{def;sisms}
Let $\Lambda$ be a directed POSET. We say that a family $\{(X_\alpha, \Sigma_\alpha)\}_{\alpha \in \Lambda}$ forms a strictly inductive system of measurable spaces, if 
\begin{enumerate}
\item $X_\alpha \subseteq X_\beta$;
\item $\Sigma_\alpha = \{ E \cap X_\alpha ~~ : ~~ E \in \Sigma_\beta \}$ (this implies $\Sigma_\alpha \subseteq \Sigma_\beta$),
\end{enumerate}
whenever $\alpha \leq \beta$.
\end{definition}


Next, we give the construction of inductive limit in this context. Suppose $\{(X_\alpha, \Sigma_\alpha)\}_{\alpha \in \Lambda}$ is a strictly inductive system of measurable spaces, then define 
\begin{equation} \label{eq;sigma algebra}
X = \bigcup\limits_{\alpha \in \Lambda} X_\alpha \; \; \; \; \; \text{and} \; \; \; \; \; \Sigma := \{ E \subseteq X ~~ : ~~ E \cap X_\alpha \in \Sigma_\alpha, ~~ \text{for all} ~~ \alpha \in \Lambda \}.
\end{equation}
Now, we prove that the collection $\Sigma$ is a $\sigma$-algebra.

\begin{proposition} \label{prop;sa}
The set $\Sigma$ defined above is a $\sigma$-algebra. 
\end{proposition}
\begin{proof}
Let $E \in \Sigma$. This means $E \subseteq X$, and $E \cap X_\alpha \in \Sigma_\alpha$ for all $\alpha \in \Lambda$. Since $\Sigma_\alpha$ is a $\sigma$-algebra, we have $(E \cap X_\alpha)^c \cap X_\alpha \in \Sigma_\alpha$ for all $\alpha \in \Lambda$. However, 
\begin{equation*}
(E \cap X_\alpha)^c \cap X_\alpha = E^c \cap X_\alpha.
\end{equation*}
Therefore, $E^c \cap X_\alpha \in \Sigma_\alpha$ for all $\alpha \in \Lambda$. Hence, $E^c \in \Sigma$. This shows that the set $\Sigma$ is closed under the set complement. Next, we show that the set $\Sigma$ is closed under countable union. Let $\{ E_n ~~ | ~~ E_n \in \Sigma , n \in \mathbb{N} \}$ be a collection of subsets of $X$ in $\Sigma$. This means that for each fixed $n \in \mathbb{N}$, we have $E_n \subseteq X$ and $E_n \cap X_\alpha \in \Sigma_\alpha$ for all $\alpha \in \Lambda$. If $\alpha \in \Lambda$ is fixed, then 
\begin{align*}
\big (\bigcup\limits_{n \in \mathbb{N}} E_n \big ) \bigcap X_\alpha = \bigcup\limits_{n \in \mathbb{N}} \big (E_n \bigcap X_\alpha \big ) \in \Sigma_\alpha
\end{align*}
Since $\alpha \in \Lambda$ was chosen arbitrarily, we get $\bigcup\limits_{n \in \mathbb{N}} E_n \in \Sigma$. This proves that the set $\Sigma$ is a $\sigma$-algebra.
\end{proof}

\subsection{Observations I} \label{obs;inductive system of MS} Here we list out a few observations related to the notion of strictly inductive system $\{(X_\alpha, \Sigma_\alpha)\}_{\alpha \in \Lambda}$ of measurable spaces.
\begin{enumerate}
\item \label{pt;csims1} It follows from Proposition \ref{prop;sa} that $(X, \Sigma)$ is a measurable space and for each $\alpha \in \Lambda$, the inclusion map $J_\alpha : X_\alpha \rightarrow X$ is measurable. Whenever $\alpha \leq \beta$, we have $\Sigma_\alpha \subseteq \Sigma_\beta$ and the inclusion map $J_{\beta, \alpha} : X_\alpha \rightarrow X_\beta$ is measurable such that 
\begin{equation*} 
J_\beta \circ J_{\beta, \alpha} = J_\alpha.
\end{equation*}
In this situation, we say that $\big ((X, \Sigma), \{J_\alpha\}_{\alpha \in \Lambda} \big)$ is compatible with $\big(\{(X_\alpha, \Sigma_\alpha)\}_{\alpha \in \Lambda}, \{J_{\beta, \alpha}\}_{\alpha \leq \beta} \big)$.

\item \label{pt;csims2} Let $(Y, \Omega)$ be a measurable space and $f_\alpha : X_\alpha \rightarrow Y$ be a measurable map for each $\alpha \in \Lambda$. Suppose $\big ((Y, \Omega), \{f_\alpha\}_{\alpha \in \Lambda} \big)$ is also compatible with $\big(\{(X_\alpha, \Sigma_\alpha)\}_{\alpha \in \Lambda}, \{J_{\beta, \alpha}\}_{\alpha \leq \beta} \big)$, that is, 
\begin{equation} \label{eq;csims1}
f_\beta \circ J_{\beta, \alpha} = f_\alpha, \; \; \text{whenever} \; \; \alpha \leq \beta,
\end{equation}
then we define a map $\Phi : X \rightarrow Y$ by $\Phi(x) := f_\alpha(x)$ for some $\alpha \in \Lambda$ such that $x \in X_\alpha$. From Equation \eqref{eq;csims1}, we get that the map $\Phi$ is well defined. Now we show that the map $\Phi$ is measurable. Let $E \in \Omega$, for any $\beta \in \Lambda$, we obtain  
\begin{equation*}
X_\beta \cap \Phi^{-1}(E) = \{ x \in X_\beta ~~ : ~~  \Phi(x) = f_\beta(x) \in E \} = f^{-1}_\beta(E) \in \Sigma_\beta.
\end{equation*}
By using the definition of the $\sigma$-algebra $\Sigma$ and since $\beta \in \Lambda$ was arbitrarily chosen, we get that the map $\Phi$ is measurable. Moreover, the map $\Phi$ satisfies the following condition
\begin{equation} \label{eq;csims2}
\Phi \circ J_\alpha = f_\alpha \; \; \text{for each} \; \; \alpha \in \Lambda.
\end{equation}

\item \label{pt;csims3} We show that the such a measurable map from $X$ to $Y$ satisfying Equation \eqref{eq;csims2} is unique. Suppose there is another measurable map $\Psi : X \rightarrow Y$ satisfying Equation \eqref{eq;csims2}, that is, $\Psi \circ J_\alpha = f_\alpha$ for each $\alpha \in \Lambda$. Then for an arbitrary $x \in X$ such that $x \in X_\alpha$ for some $\alpha \in \Lambda$, we get 
\begin{equation*}
\Psi(x) = \Psi \circ J_\alpha(x) = f_\alpha(x) \; \; \text{and} \; \; \Phi(x) = \Phi \circ J_\alpha(x) = f_\alpha(x).
\end{equation*}
This shows that $\Psi(x) = \Phi(x)$. Since $x 
\in X$ was arbitrarily chosen, we obtain $\Psi = \Phi$. This gives the existence of a unique measurable map $\Phi : X \rightarrow Y$ such that $\Phi \circ J_\alpha = f_\alpha$ for each $\alpha \in \Lambda$. 

\item \label{obs; 1,4} In view of (\ref{pt;csims1}), (\ref{pt;csims2}) and (\ref{pt;csims3}) we call the measurable space $(X, \Sigma)$ as the inductive limit of the strictly inductive system of measurable spaces $\big \{(X_\alpha, \Sigma_\alpha) \big \}_{\alpha \in \Lambda}$ and we denote this by 
\begin{equation*}
(X, \Sigma) = \varinjlim\limits_{\alpha \in \Lambda} \;(X_\alpha, \Sigma_\alpha). 
\end{equation*}
\end{enumerate}

We conclude the above discussion by using the following commutative diagram.
\begin{center}
\begin{tikzcd}[sep=huge]
&  X   \arrow[rr, "\Phi"]  & & Y \\
X_\beta  \arrow[ur, "J_\beta"]  & & X_\alpha \arrow[to path={node[midway,scale=3] {$\circlearrowright$}}] \arrow[ll, "J_{\beta, \alpha}"] \arrow[ul, "J_\alpha"'] \arrow[ur, "f_\alpha"] \arrow[rr, "J_{\beta, \alpha}"']  & & \arrow[ul, "f_\beta"']  X_\beta 
\end{tikzcd}
\end{center}

\begin{remark}
    If we define $\Sigma_{0} =  \bigcup\limits_{\alpha \in \Lambda} \Sigma_\alpha $, then $\Sigma_{0}$ may not be a $\sigma$-algebra (see Example \ref{Eg: sigma0}). However, we see that $\Sigma_{0} \subseteq \Sigma.$ Suppose $E \subseteq X$ such that $E \in \Sigma_0$, meaning $E \in \Sigma_{\alpha^\prime}$ for some $\alpha^\prime \in \Lambda$. This implies $E \in \Sigma_{\beta}$ for all $\beta \in \Lambda$ such that $\alpha^\prime \leq \beta$. Let $\alpha \in \Lambda$. Then there exists $\beta^\prime \in \Lambda$ such that $\alpha \leq \beta^\prime$ and $\alpha^\prime \leq \beta^\prime$. Therefore, $E \in \Sigma_{\beta^\prime}$, and $E \cap X_\alpha \in \Sigma_\alpha$. Since $\alpha$ was chosen arbitrarily, we get $E \in \Sigma$. This proves $\Sigma_0 \subseteq \Sigma$.
\end{remark}
\noindent
In the following example, we illustrate that $\Sigma_0$, the union of $\sigma$-algebras is not necessarily a $\sigma$-algebra.

\begin{example} \label{Eg: sigma0}
Consider the family of measurable spaces $\big \{([-n, n], \Sigma_n) \big \}_{n \in \mathbb{N}}$, where $\Sigma_n$ denotes the $\sigma$-algebra of Lebesgue measurable subsets of $[-n, n]$. For each $n \in \mathbb{N}$, we have $[-n, n] \in \Sigma_n$. However, $\bigcup\limits_{n \in \mathbb{N}} [-n, n] = \mathbb{R} \notin \Sigma_0 = \bigcup\limits_{n \in \mathbb{N}} \Sigma_n$. This example demonstrates that $\Sigma_0$ is not necessarily a $\sigma$-algebra. 
\end{example}

Next, we introduce the notion of projective system of finite measures. We begin by considering a strictly inductive system $\{(X_\alpha, \Sigma_\alpha)\}_{\alpha \in \Lambda}$ of measurable spaces. Suppose for each $\alpha \in \Lambda$, $(X_\alpha, \Sigma_\alpha, \mu_\alpha)$ is a finite (positive) measure space, then we call the family $\{\mu_\alpha\}_{\alpha \in \Lambda}$ a \textit{projective system of measures}, if for each $E_\alpha \in \Sigma_\alpha$, we have
\begin{equation*}
\mu_\alpha(E_\alpha) = \mu_\beta(E_\alpha),  \; \; \text{whenever} \; \;  \alpha \leq \beta.
\end{equation*}
This implies that for every $E \in \Sigma$, we see that 
\begin{equation*}
\mu_\alpha(E \cap X_\alpha) = \mu_\beta(E \cap X_\alpha)
\leq \mu_\beta(E \cap X_\beta).
\end{equation*}

\begin{proposition} \label{prop;m}
Suppose $\{(X_\alpha, \Sigma_\alpha, \mu_\alpha)\}_{\alpha \in \Lambda}$ is a family of finite measure spaces such that $\{(X_\alpha, \Sigma_\alpha)\}_{\alpha \in \Lambda}$ is a strictly inductive system of measurable spaces and $\{\mu_\alpha\}_{\alpha \in \Lambda}$ is a projective system of measures. If $(X, \Sigma) = \varinjlim\limits_{\alpha \in \Lambda} \;(X_\alpha, \Sigma_\alpha)$, then the map $\mu : \Sigma \rightarrow [0, \infty]$ defined by 
\begin{equation*}
\mu(E) := \begin{cases}
\lim\limits_\alpha \; \mu_\alpha(E \cap X_\alpha), & \text{if} \;\; \{ \mu_\alpha(E \cap X_\alpha) \}_{\alpha \in \Lambda} \; \text{converges};\\
& \\
\infty, & \text{otherwise}
\end{cases} 
\end{equation*} 
is a measure. In this situation, we call the measure $\mu$ as the projective limit of the projective system $\{\mu_\alpha\}_{\alpha \in \Lambda}$ of measures and denote by the notation $\mu = \varprojlim\limits_{\alpha \in \Lambda} \;\mu_\alpha$. 
\end{proposition}
\begin{proof}
Let $\emptyset$ denote the empty set of $X$. Then $\mu(\emptyset) = \lim\limits_\alpha \{ \mu_\alpha(\emptyset \cap X_\alpha) \} = 0.$ Further, assume that $\{ E_n \in \Sigma ~~ : ~~  n \in \mathbb{N} \}$ is a collection of pairwise disjoint subsets of $X$ in $\Sigma$. Then, we obtain
\begin{align*} 
\mu \big (\bigcup_{n \in \mathbb{N}} E_n \big ) &= \lim\limits_{\alpha} \; \mu_\alpha \big (\big (\bigcup_{n \in \mathbb{N}} E_n \big ) \bigcap X_\alpha \big ) \\
&= \lim\limits_{\alpha} \; \mu_\alpha  \big (\bigcup_{n \in \mathbb{N}} \left (E_n \bigcap X_\alpha \right ) \big) \\
&= \lim\limits_{\alpha} \; \sum_{n=1}^{\infty} \mu_\alpha \big (E_n \bigcap X_\alpha \big ) 
\end{align*}
and 
\begin{align*} 
\sum_{n=1}^{\infty} \mu(E_n) = \sum_{n=1}^{\infty} \lim\limits_{\alpha} \; \mu_\alpha \big  (E_n \bigcap X_\alpha \big).
\end{align*}

For each $\alpha \in \Lambda$, define a function $f_\alpha : \mathbb{N} \rightarrow [0, \infty)$ by $f_\alpha(n) := \mu_\alpha(E_n \cap X_\alpha).$ Whenever $\alpha \leq \beta$, we have $f_\alpha(n) = \mu_\alpha(E_n \cap X_\alpha) \leq f_\beta(n) = \mu_\beta(E_n \cap X_\beta)$ for all $n \in \mathbb{N}$. If we define a function $f : \mathbb{N} \rightarrow [0, \infty]$ by $f(n) := \mu(E_n)$, then for all $n \in \mathbb{N}$ we get $\lim\limits_\alpha f_\alpha(n) = f(n)$. By using the monotone convergence theorem, we have 
\begin{equation*}
\lim\limits_\alpha \sum\limits_{n=1}^{\infty} f_\alpha(n) = \sum\limits_{n=1}^{\infty} f(n).
\end{equation*}
This implies that
\begin{align*}
\lim\limits_\alpha \sum_{n=1}^{\infty} \mu_\alpha \left (E_n \bigcap X_\alpha \right ) = \lim\limits_\alpha \sum\limits_{n=1}^{\infty} f_\alpha(n) = \sum\limits_{n=1}^{\infty} f(n) &= \sum_{n=1}^{\infty} \mu (E_n) \\
&= \sum_{n=1}^{\infty} \lim\limits_\alpha \mu_\alpha \big (E_n \bigcap X_\alpha \big ) \\
&= \mu \big (\bigcup_{n \in \mathbb{N}} E_n \big ).
\end{align*}
This proves $\mu \big (\bigcup\limits_{n \in \mathbb{N}} E_n \big ) = \sum\limits_{n=1}^{\infty} \mu(E_n)$, and hence the map $\mu$ is a measure.
\end{proof}

\begin{definition} (Locally standard measure space) \label{def; lsms}
We call the measure space $(X, \Sigma, \mu)$ obtained in Proposition \ref{prop;m} a \textbf{locally standard measure space}, if $(X_\alpha, \Sigma_\alpha, \mu_\alpha)$ is a  standard measure space for each $\alpha \in \Lambda$. That is, $(X, \Sigma) = \varinjlim\limits_{\alpha \in \Lambda} \;(X_\alpha, \Sigma_\alpha)$ and $\mu =\varprojlim\limits_{\alpha \in \Lambda} \;\mu_\alpha$, where $(X_\alpha, \Sigma_\alpha, \mu_\alpha)$ is a standard measure space for each $\alpha \in \Lambda$.
\end{definition}

Throughout this article, $(X, \Sigma, \mu)$ indicates a locally standard measure space along with the family $\{ (X_\alpha, \Sigma_\alpha, \mu_\alpha)\}_{\alpha \in \Lambda}$ of standard measure spaces such that $(X, \Sigma) = \varinjlim\limits_{\alpha \in \Lambda} \;(X_\alpha, \Sigma_\alpha)$ and $\mu =\varprojlim\limits_{\alpha \in \Lambda} \;\mu_\alpha$, unless otherwise stated.

\begin{example} \label{ex;lsms}
Let $\Lambda = \mathbb{N}$ and $X_n = [-n, n]$ for each $n \in \mathbb{N}$. Suppose $B(X_n)$ denotes the Borel $\sigma$-algebra of $X_n$ and $\mu_n$ denotes the Lebesgue measure on $B(X_n)$. Then $\{(X_n, B(X_n), \mu_n)\}_{n \in \mathbb{N}}$ is a family of standard measure spaces such that $\{(X_n, B(X_n)\}_{n \in \mathbb{N}}$ is a strictly inductive system of measurable spaces and $\{\mu_n\}_{n \in \mathbb{N}}$ is a projective system of finite measures. As we know that a subset $U$ of $\mathbb{R}$ is Borel if and only if $U \cap X_n$ is Borel for every $n \in \mathbb{N}$, it follows from Equation \eqref{eq;sigma algebra} and (\ref{obs; 1,4}) of Observations I that  
\begin{equation*}
    (\mathbb{R}, B(\mathbb{R})) = \varinjlim\limits_{n \in \mathbb{N}} \;(X_n, B(X_n)).
\end{equation*}  
Moreover, if $\mu$ is the Lebesgue measure of $\mathbb{R}$ and $E \in B(\mathbb{R})$, then $E = \bigcup\limits_{n \in \mathbb{N}}(E \cap X_n)$, where 
\begin{equation*}
E \cap X_m \subseteq E \cap X_n, \; \; \text{whenever} \; \; m \leq n \; \; \text{and} \; \; \mu_n(E \cap X_n) = \mu (E \cap X_n) \; \; \text{for all} \; \; n \in \mathbb{N}.
\end{equation*} 
This implies that either $\mu(E) = \infty$ or $\mu(E) = \lim\limits_{n \to \infty} \mu (E \cap X_n) = \lim\limits_{n \to \infty} \mu_n (E \cap X_n)$. Thus $\mu = \varprojlim\limits_n \mu_n$. Hence, $(\mathbb{R}, B(\mathbb{R}), \mu)$ is a locally standard measure space.
\end{example}

\begin{remark}
Let $(X, \Sigma, \mu)$ be a locally standard measure space. Then the set $X$ is equipped with the inductive limit topology, which is the strongest topology that makes the inclusion map $J_\alpha : X_\alpha \rightarrow X$ continuous for each $\alpha \in \Lambda$. In fact, $U \subseteq X$ is open in $X$ with respect to the inductive limit topology if and only if $U \cap X_\alpha$ is open in 
$X_\alpha$ for every $\alpha \in \Lambda$.
\end{remark}

\subsection{Direct integral of locally Hilbert spaces}
Now we are in a position to state the notion of direct integral of locally Hilbert spaces.

\begin{definition} \label{Defn: directint_loc}
Let $(\Lambda, \leq)$ be a directed POSET and $(X, \Sigma, \mu)$ be a locally standard measure space. For each $p \in X$,  assign a locally Hilbert space $\mathcal{D}_p = \varinjlim\limits_{\alpha \in \Lambda} \mathcal{H}_{\alpha, p}$, where $\{ \mathcal{H}_{\alpha, p} \}_{\alpha \in \Lambda}$ forms a strictly inductive system of separable Hilbert spaces. The direct integral of locally Hilbert spaces $\{ \mathcal{D}_p \}_{p \in X}$ over the locally standard measure space $(X, \Sigma, \mu)$ comprises functions of the form $u : X \rightarrow \bigcup\limits_{p \in X} \mathcal{D}_p$ with $u(p) \in \mathcal{D}_p$ for all $p \in X$ satisfying the following conditions: 
\begin{enumerate}
\item for each $u : X \rightarrow \bigcup\limits_{p \in X} \mathcal{D}_p$, there exists $\alpha_u \in \Lambda$ such that the support of $u$, 
\begin{equation*}
\text{supp}(u) := \{ p \in X \; : \; u(p) \neq 0_{\mathcal{D}_p}  \} \subseteq X_{\alpha_u}
\end{equation*} 
and $u(p) \in \mathcal{H}_{\alpha_u, p}$ for almost every $p \in X_{\alpha_u}$; 
\item for any $u$ and $v$ satisfying the property (1), the function $\zeta_{u,v} : X \rightarrow \mathbb{C}$ defined by 
\begin{equation*}
\zeta_{u,v}(p) := \langle u(p), v(p) \rangle_{\mathcal{D}_p}, \; \text{for all}\; p \in X
\end{equation*}
is in  $\text{L}^1(X, \mu)$;
\item let $\{ v_q \}_{q \in X}$ be a family of vectors such that $v_q \in \mathcal{D}_q$ for all $q \in X$ and there exists $\alpha \in \Lambda$ such that $v_q = 0_{\mathcal{D}_q}$ for all $q \in X \setminus X_{\alpha}$. For each $u : X \rightarrow \bigcup\limits_{p \in X} \mathcal{D}_p$ with $u(p) \in \mathcal{D}_p$ for all $p \in X$, define the function $\eta_{u, \{ v_q \}_{q \in X}} : X \rightarrow \mathbb{C}$ by 
\begin{equation*}
\eta_{u, \{ v_q \}_{q \in X}}(p) := \la u(p), v_{p} \ra_{\mathcal{D}_p} \; \; \text{for all} \; \; p \in X.
\end{equation*} 
If the map $\eta_{u, \{ v_q \}_{q \in X}} \in \text{L}^1(X, \mu)$ for every $u$ satisfying the property (1) and $\zeta_{u,u} \in \text{L}^1(X, \mu)$, then there exists a function  $v : X \rightarrow \bigcup\limits_{p \in X} \mathcal{D}_p$ satisfying the property (1) and $\zeta_{v,v} \in \text{L}^1(X, \mu)$ such that $v(p) = v_p$ for almost every $p \in X$.  
\end{enumerate}
In this case, we denote the direct integral of locally Hilbert spaces $\{ \mathcal{D}_p \}_{p \in X}$ over the locally standard measure space $(X, \Sigma, \mu)$ by $\displaystyle \dilX \mathcal{D}_p \, \dmu$ and the vector $u \in \displaystyle \dilX \mathcal{D}_p \, \dmu$ by $\displaystyle \dilX u(p) \, \dmu$.
\end{definition}


Next, we show that the set $\displaystyle\dilX \mathcal{D}_p \, \dmu$ is indeed a locally Hilbert space.

\begin{proposition} \label{prop;dilhs}
Let $(X, \Sigma, \mu)$ be a locally standard measure space and $\{ \mathcal{D}_p \}_{p \in X}$ be a family of locally Hilbert spaces, where $\mathcal{D}_p = \varinjlim\limits_{\alpha \in \Lambda} \mathcal{H}_{\alpha, p}$. Then $\displaystyle \dilX \mathcal{D}_p \, \dmu$ is a locally Hilbert space. 
\end{proposition}
\begin{proof}
To prove that $\displaystyle \dilX  \mathcal{D}_p \, \dmu$ is a locally Hilbert space, first we construct a strictly inductive system of Hilbert spaces. For each fixed $\alpha \in \Lambda$, we define the set $\mathcal{H}_\alpha$ by
\begin{equation}\label{eq; H alpha}
\mathcal{H}_\alpha := \left \{ u \in \dilX \mathcal{D}_p \, \dmu ~~ : ~~  \text{supp}(u) \subseteq X_\alpha, \; \;  u(p) \in \mathcal{H}_{\alpha, p} \; \; \text{for almost every} \; \; p \in X_\alpha \right \}.
\end{equation}
The set $\mathcal{H}_\alpha$ can be $\{ 0_{\mathcal{H}_\alpha} \}$ for some $\alpha \in \Lambda$. For instance, if the family $\{ \mathcal{D}_p \}_{p \in X}$ is such that for a fixed $\alpha \in \Lambda$, $\mathcal{H}_{\alpha, p} = \{ 0_{\mathcal{H}_{\alpha, p}}  \}$ for almost every $p \in X_\alpha$, then we get that $\mathcal{H}_\alpha$ to be the zero space. However, $\displaystyle \dilX  \mathcal{D}_p \, \dmu$ need not be the zero space. On the other hand, let $\alpha \in \Lambda$ be fixed, and $E \subseteq X_\alpha$ be such that $\mu (E) = \mu_\alpha(E) > 0$ with $\mathcal{H}_{\alpha, p}$ is non-trivial Hilbert space for all $p \in E$. Then consider the family $\{ v_p \}_{p \in X}$, where $v_p$ is a unit vector in $\mathcal{H}_{\alpha, p} \subseteq \mathcal{D}_p$ if $p \in E$ and $v_p = 0_{\mathcal{D}_p}$ if $p \in X \setminus E$. Suppose $u : X \rightarrow \bigcup\limits_{p \in X} \mathcal{D}_p$ is a function satisfying the property (1) of Definition \ref{Defn: directint_loc} and $\zeta_{u,u} \in \text{L}^1(X, \mu)$, \big(equivalently, the map $p \mapsto \big \| 
u(p) \big \|_{\mathcal{H}_{{\alpha_u}, p}}$ is in $\text{L}^2(X, \mu)$ \big) then 
\begin{align*}
\int_{X} \; \big | \la u(p), v_p  \ra \big| \; \dmu &\leq \int\limits_{X_{\alpha_u} \bigcap E} \; \| u(p) \|_{\mathcal{H}_{{\alpha_u}, p}}\; \| v_p \|_{\mathcal{H}_{\alpha, p}}\; \dmu \\
&= \int\limits_{X_{\alpha_u} \bigcap E} \; \| u(p) \|_{\mathcal{H}_{{\alpha_u}, p}}  \; \mathrm{d} \mu_{\alpha_u} \\
&< \infty.
\end{align*}
The last inequality holds true as the map $p \mapsto \big \| 
u(p) \big \|_{\mathcal{H}_{{\alpha_u}, p}}$ is in $\text{L}^2(X_{\alpha_u}, \mu_{\alpha_u})$, where $\mu_{\alpha_u}$ is a finite measure. So, from the property (3) of Definition \ref{Defn: directint_loc}, there exists $v \in \displaystyle \dilX  \mathcal{D}_p \, \dmu$ such that $v(p) = v_p$ if $p \in X_\alpha$ and $v(p) = 0_{\mathcal{D}_p}$ if $p \in X \setminus X_\alpha$. In particular, $0_{\mathcal{H}_\alpha} \neq v \in \mathcal{H}_\alpha$. Thus $\mathcal{H}_\alpha$ is non-trivial.

Let $u, v \in \mathcal{H}_\alpha$, and define
\begin{equation} \label{eq;ip alpha}
\langle u, v \rangle_{\mathcal{H}_\alpha} : = \int\limits_{X_{\alpha}} \big \langle u(p),\; v(p) \big \rangle_{\mathcal{H}_{\alpha, p}} \; \mathrm{d} \mu_{\alpha}.
\end{equation} 
This gives an inner product on $\mathcal{H}_\alpha$. Whenever $\alpha \leq \beta$, it is clear that $\mathcal{H}_\alpha \subseteq \mathcal{H}_\beta$ and since $X_\alpha \subseteq X_\beta$, for $u, v \in \mathcal{H}_\alpha$, we have
\begin{equation} \label{eq;ip alpha, beta}
\langle u, v \rangle_{\mathcal{H}_\beta} = \int\limits_{X_{\beta}} \big \langle u(p),\; v(p) \big \rangle_{\mathcal{H}_{\beta, p}} \mathrm{d} \mu_{\beta} 
= \int\limits_{X_{\alpha}} \big \langle u(p),\; v(p) \big \rangle_{\mathcal{H}_{\alpha, p}} \mathrm{d} \mu_{\alpha} + \int\limits_{X_\beta \setminus X_{\alpha}} \big \langle u(p),\; v(p) \big \rangle_{\mathcal{H}_{\beta, p}} \mathrm{d} \mu_{\alpha} 
= \langle u, v \rangle_{\mathcal{H}_\alpha}.
\end{equation}
Thus the inclusion map $J_{\beta, \alpha} : \mathcal{H}_\alpha \rightarrow \mathcal{H}_\beta$ is an isometry. Now for each $\alpha \in \Lambda$, we will show that the inner product space $\mathcal{H}_\alpha$ is complete. Fix $\alpha \in \Lambda$, and consider 
$\int^\oplus_{X_\alpha} \mathcal{H}_{\alpha, p} \, \mathrm{d} \mu_\alpha (p)$, the direct integral of the family $\{ \mathcal{H}_{\alpha, p} \}_{p \in X_\alpha}$ of Hilbert spaces over the standard measure space $ \big (X_\alpha, \Sigma_\alpha, \mu_\alpha \big )$. Now define $V_\alpha : \mathcal{H}_\alpha \rightarrow \int^\oplus_{X_\alpha} \mathcal{H}_{\alpha, p} \, \mathrm{d} \mu_\alpha (p)$ by 
\begin{equation} \label{eq;iso}
V_\alpha(u)(p) := u(p) \; \; \text{for all} \; \; p \in X_\alpha \; \; \text{and} \; \; u \in \mathcal{H}_\alpha.
\end{equation}
Then we get
\begin{align*}
\langle u, u \rangle_{\mathcal{H}_\alpha} &=
\int\limits_{X_{\alpha}} \big \langle u(p),\; u(p) \big \rangle_{\mathcal{H}_{\alpha, p}}\; \mathrm{d} \mu_{\alpha} \\
&= \int\limits_{X_{\alpha}} \big \langle V_\alpha(u)(p),\; V_\alpha(u)(p)\big \rangle_{\mathcal{H}_{\alpha, p}} \; \mathrm{d} \mu_{\alpha} = \big \langle V_\alpha(u), V_\alpha(u) \big \rangle.
\end{align*}
This proves that the map $V_\alpha$ is an isometry. Now, we show that the map $V_\alpha$ is surjective. If $v^\prime \in \int^\oplus_{X_\alpha} \mathcal{H}_{\alpha, p} \, \mathrm{d} \mu_\alpha (p)$, then consider a function $v \in \mathcal{H}_{\alpha}$ such that $v(p) = v^\prime(p)$ if $p \in X_\alpha$ and $v(p) = 0_{\mathcal{D}_p}$ if $p \in X \setminus X_\alpha$. Then we get $V_\alpha(v) = v^\prime$. Hence, the map $V_\alpha$ defines an isomorphism between the inner product space $\mathcal{H}_\alpha$ and the Hilbert space $\int^\oplus_{X_\alpha} \mathcal{H}_{\alpha, p} \, \mathrm{d} \mu_\alpha (p)$. Since $\alpha \in \Lambda$ is arbitrary, each $\mathcal{H}_\alpha$ is a Hilbert space. Thus, $\{ \mathcal{H}_\alpha \}_{\alpha \in \Lambda}$ forms a strictly inductive system of Hilbert spaces and from Equation \eqref{eq;lhs}, we get that 
\begin{equation} \label{eq;il=u}
\varinjlim\limits_{\alpha \in \Lambda} \mathcal{H}_\alpha = \bigcup\limits_{\alpha \in \Lambda} \mathcal{H}_\alpha.
\end{equation} 

Clearly, $\bigcup\limits_{\alpha \in \Lambda} \mathcal{H}_\alpha \subseteq \displaystyle \dilX \mathcal{D}_p \, \dmu$. If $u \in \displaystyle \dilX \mathcal{D}_p \, \dmu$, then from the property (1) of Definition \ref{Defn: directint_loc} there exists $\alpha_u \in \Lambda$ such that $\text{supp}(u) \subseteq X_{\alpha_u}$ and $u(p) \in \mathcal{H}_{\alpha_u, p}$ for almost every $p \in X_{\alpha_u}$. This implies that $u \in \mathcal{H}_{\alpha_{u}}$. Hence, we get $\displaystyle \dilX \mathcal{D}_p \, \dmu \subseteq \bigcup\limits_{\alpha \in \Lambda} \mathcal{H}_\alpha$, and this proves 
\begin{equation} \label{eq;dilhs=u}
\dilX \mathcal{D}_p \, \dmu = \bigcup\limits_{\alpha \in \Lambda} \mathcal{H}_\alpha.
\end{equation}
Then Equation \eqref{eq;il=u} and Equation \eqref{eq;dilhs=u} imply that 
\begin{equation} \label{eq;dilhs=il}
\dilX \mathcal{D}_p \, \dmu = \varinjlim\limits_{\alpha} \mathcal{H}_\alpha.
\end{equation} 
Therefore, from Definition \ref{def;lhs}, we conclude that $\displaystyle \dilX \mathcal{D}_p \, \dmu$ is a locally Hilbert space.
\end{proof}

\begin{note}
Note that, one can define an inner product on $\displaystyle \dilX \mathcal{D}_p \, \dmu$ as follows. If $u, v \in \displaystyle \dilX \mathcal{D}_p \, \dmu$, then from Proposition \ref{prop;dilhs}, there exists $\alpha \in \Lambda$ such that $u, v \in \mathcal{H}_\alpha$  $(\alpha_u, \alpha_v \leq \alpha)$ and define 
\begin{equation*}
\langle u, v \rangle := \langle u, v \rangle_{\mathcal{H}_\alpha} = \int\limits_{X_{\alpha}} \big \langle u(p),\; v(p) \big \rangle_{\mathcal{H}_{\alpha, p}} \; \mathrm{d} \mu_{\alpha}.
\end{equation*}
In view of Equation \eqref{eq;ip alpha, beta}, the map $\la \cdot, \cdot \ra$ is well defined and it gives an inner product on $\displaystyle \dilX \mathcal{D}_p \, \dmu$. In particular, if $u \in \displaystyle \dilX \mathcal{D}_p \, \dmu$, then 
\begin{equation*}
\|u\|^{2} =  \int\limits_{X_{\alpha_{u}}} \|u(p)\|^{2}_{\mathcal{H}_{\alpha_{u}, p}} \mathrm{d} \mu_{\alpha_{u}}. \qedhere
\end{equation*}
\end{note}
\begin{example} \label{ex;dilhs}
In this example, we see two distinct cases, namely the discrete and the non discrete case.
\begin{enumerate}
\item \label{ex;dilhs1}
Consider the locally standard measure space $\big ( \mathbb{N}, \Sigma, \mu \big )$, where $\mu$ is the counting measure on $\mathbb{N}$ and $\{ \mathcal{D}_n \}_{n \in \mathbb{N}}$ is a sequence of locally Hilbert spaces. The direct integral is given by 
\begin{equation*}
\int^{\oplus_\text{loc}}_\mathbb{N} \mathcal{D}_n \; \mathrm{d} \mu(n) = \bigoplus_{n \in \mathbb{N}} \mathcal{D}_n.
\end{equation*}

\item \label{ex;dilhs2}
Consider the locally standard measure space $(\mathbb{R}, B(\mathbb{R}), \mu)$ (see Example \ref{ex;lsms}, here $\Lambda = \mathbb{N}$). For each $p \in \mathbb{R}$, assign a locally Hilbert space $\mathcal{D}_p = \varinjlim\limits_{n \in \mathbb{N}} \mathcal{H}_{n,p}$, where $\mathcal{H}_{n,p} = \mathbb{C}$ for each $n \in \mathbb{N}$. Then for each $p \in \mathbb{R}$, we get $\mathcal{D}_p = \mathbb{C}$ and from Proposition \ref{prop;dilhs}, we have $\displaystyle \int^{\oplus_\text{loc}}_{\mathbb{R}} \mathcal{D}_p \, \dmu = \varinjlim\limits_{n \in \mathbb{N}} \mathcal{H}_n$, where 
\begin{equation*}
\mathcal{H}_n := \left \{ u \in \int^{\oplus_\text{loc}}_{\mathbb{R}} \mathcal{D}_p \, \dmu \; \; : \; \;  \text{supp}(u) \subseteq [-n,n], \;  \; u(p) \in \mathcal{H}_{n, p} \; \text{for almost every} \; p \in [-n,n] \right \}.
\end{equation*}
Further, we know from the proof of Proposition \ref{prop;dilhs} that the Hilbert space $\mathcal{H}_n$ is isomorphic to $\int^\oplus_{[-n,n]} \mathcal{H}_{n, p} \, \mathrm{d} \mu_n$, the direct integral of the family $\{ \mathcal{H}_{n, p} \}_{p \in [-n, n]}$ of Hilbert spaces for each fixed $n \in \mathbb{N}$. That is, $\mathcal{H}_n$ consists of all Borel measurable functions $u : \mathbb{R} \rightarrow \mathbb{C}$ such that $\text{supp}(u) \subseteq [-n, n]$ and $\int\limits_{\mathbb{R}} |u(p)|^{2} \; \mathrm{d} \mu < \infty$. Therefore, 
\begin{equation*}
\displaystyle \int^{\oplus_\text{loc}}_{\mathbb{R}} \mathcal{D}_p \, \dmu = \Big\{ u \in \text{L}^2 \big (\mathbb{R}, B(\mathbb{R}), \mu \big) \; \; : \; \; \text{supp}(u) \subseteq [-n, n] \; \; \text{for some} \; \; n \in \mathbb{N}  \Big\}.
\end{equation*}
\end{enumerate}
\end{example}

\begin{remark}
By following Definition \ref{def;dihs}, we see that $\big (\mathbb{R}, B(\mathbb{R}), \mu \big)$ is a standard measure space. Since each $\mathcal{D}_p = \mathbb{C}$  in \ref{ex;dilhs2} of Example  \ref{ex;dilhs} is indeed a Hilbert space, we can consider the direct integral of Hilbert spaces $\{ \mathcal{D}_p = \mathbb{C}\}_{p \in \mathbb{R}}$ over the standard measure space $\big (\mathbb{R}, B(\mathbb{R}), \mu \big)$, which is  
\begin{align*}
\int^{\oplus}_{\mathbb{R}} \mathcal{D}_p \, \dmu &= \Big\{ u : \mathbb{R} \rightarrow \mathbb{C} \; \; : \; \; u \; \; \text{is Borel measurable and} \; \; \int\limits_{\mathbb{R}} |u(p)|^{2} \; \mathrm{d} \mu < \infty  \Big\} \\
&= \text{L}^2 \big (\mathbb{R}, B(\mathbb{R}), \mu \big).
\end{align*}
Further, it is clear from (\ref{ex;dilhs2}) of Example \ref{ex;dilhs} that $ \displaystyle \int^{\oplus_\text{loc}}_{\mathbb{R}} \mathcal{D}_p \, \dmu $ is dense subspace of $\int^{\oplus}_{\mathbb{R}} \mathcal{D}_p \, \dmu$. 
\end{remark}

Now we turn our attention towards a sub-collection of locally bounded operators defined on $\displaystyle \dilX \mathcal{D}_p \, \dmu$ that adheres to the notion of direct integrals. We introduce this class motivated from the classical setup. For this, we use the same set of notations as used in Definition \ref{Defn: directint_loc}. 

\begin{definition} \label{def;DecDiag(lbo)}
A locally bounded operator $T : \displaystyle \dilX \mathcal{D}_p \, \dmu \rightarrow \displaystyle \dilX \mathcal{D}_p \, \dmu$ is said to be:
\begin{enumerate} 
\item \label{def;Dec(lbo)} \textbf{decomposable}, if there exists a family $ \big \{ T_p : \mathcal{D}_p \rightarrow \mathcal{D}_p \big \}_{p \in X}$ of locally bounded operators such that for any $u \in \displaystyle \dilX \mathcal{D}_p \, \dmu$, we have
\begin{align*}
(Tu)(p) = T_pu(p)
\end{align*}
for almost every $p \in X$. In this case, we denote the operator $T$ by the notation $\displaystyle \dilX T_p \, \dmu$ and so
\begin{equation*}
\left (\dilX T_p \, \dmu \right ) \left (\dilX u(p) \, \dmu \right )= \dilX T_pu(p) \, \dmu;
\end{equation*}
\item \label{def;Diag(lbo)} \textbf{diagonalizable}, if $T$ is decomposable and there exists a measurable function $f : X \rightarrow \mathbb{C}$ such that for any $u \in \displaystyle \dilX \mathcal{D}_p \, \dmu$, we have 
\begin{equation*}
(Tu)(p) = f(p)u(p)
\end{equation*}
for almost every $p \in X$. In this situation, we get $T = \displaystyle \dilX T_p \, \dmu = \displaystyle \dilX f(p) \cdot \mathrm{Id}_{\mathcal{D}_p} \, \dmu$.
\end{enumerate}  
\end{definition}

\noindent
We denote the collection of all decomposable locally bounded operators and the collection of all diagonalizable locally bounded operators on $\displaystyle \dilX \mathcal{D}_p \, \dmu$ by $\mathcal{M}_{\text{DEC}}$ and $\mathcal{M}_{\text{DIAG}}$ respectively.

\begin{example}
Consider the direct integral of locally Hilbert spaces
\begin{equation*}
\displaystyle \int^{\oplus_\text{loc}}_{\mathbb{R}} \mathcal{D}_p \, \dmu = \Big\{ u \in \text{L}^2 \big (\mathbb{R}, B(\mathbb{R}), \mu \big) \; \; : \; \; \text{supp}(u) \subseteq [-n, n] \; \; \text{for some} \; \; n \in \mathbb{N}  \Big\}
\end{equation*}
as in Example \ref{ex;dilhs}. Let $f : \mathbb{R} \rightarrow \mathbb{C}$ be a measurable function defined by $f(p) := p$ for all $p \in \mathbb{R}$. Corresponding to the function $f$, we define a locally bounded operator 
\begin{equation*}
T_f : \displaystyle \int^{\oplus_\text{loc}}_{\mathbb{R}} \mathcal{D}_p \, \dmu 
\rightarrow \displaystyle \int^{\oplus_\text{loc}}_{\mathbb{R}} \mathcal{D}_p \, \dmu \; \; \; \; \text{by} \; \; \; \; (T_fu)(p) = f(p)u(p) = pu(p)
\end{equation*}
for almost every $p \in \mathbb{R}$ and for every $u \in  \displaystyle \int^{\oplus_\text{loc}}_{\mathbb{R}} \mathcal{D}_p \, \dmu$. Then $T_f$ is a decomposable and moreover a diagonalizable operator on $\displaystyle \int^{\oplus_\text{loc}}_{\mathbb{R}} \mathcal{D}_p \, \dmu$.
\end{example}

Now we give an example of a decomposable locally bounded operator (defined on direct integral of locally Hilbert spaces) that is not diagonalizable .

\begin{example} \label{eg; Dec but not Diag 1}
Consider the locally standard measure space $\big ( \mathbb{N}, \Sigma, \mu \big )$ and direct integral of locally Hilbert spaces $\{ 
\mathcal{D}_n \}_{n \in \mathbb{N}}$ given by
\begin{equation*}
\int^{\oplus_\text{loc}}_\mathbb{N} \mathcal{D}_n \; \mathrm{d} \mu(n) = \bigoplus_{n \in \mathbb{N}} \mathcal{D}_n,
\end{equation*}
where $\mathcal{D}_1 = \bigcup\limits_{n \in \mathbb{N}} \text{span} \{ e_1, e_2, ..., e_n \} \subset \ell^2(\mathbb{N})$ \big (here $\{e_{n}: n \in \mathbb{N}\}$ is a Hilbert basis of $\ell^2(\mathbb{N})$ \big) and $\mathcal{D}_n = \{ 0 \}$ for $n \geq 2$. Now define a locally bounded operator $T : \bigoplus\limits_{n \in \mathbb{N}} \mathcal{D}_n \rightarrow \bigoplus\limits_{n \in \mathbb{N}} \mathcal{D}_n$ as 
\begin{equation*}
T\big ( \big \{ u(n) \big \}_{n \in \mathbb{N}}  \big ) = \big \{ T_nu(n) \big \}_{n \in \mathbb{N}},
\end{equation*}
where $T_1$ is as defined in Example \ref{ex;lbo} and $T_n = 0$ for $n \geq 2$. For instance, 
\begin{equation*} 
T \left ( \left \{ \sum^N_{k = 1} \lambda_k e_k, 0, 0, 0, 0, ... \right \} \right) = \left \{ \sum^N_{k = 1} k \lambda_k e_k, 0, 0, 0, 0, ... \right \}.
\end{equation*}
It shows that $T$ is a decomposable locally bounded operator on $\bigoplus\limits_{n \in \mathbb{N}} \mathcal{D}_n$.  Now we show that $T$ is not diagonalizable. By following (\ref{def;Diag(lbo)}) of Definition \ref{def;DecDiag(lbo)}, if there is a measurable function $f : \mathbb{N} \rightarrow \mathbb{C}$ satisfying $T\big ( \big \{ u(n) \big \}_{n \in \mathbb{N}}  \big ) = \big \{ f(n)u(n) \big \}_{n \in \mathbb{N}}$ for every $\big \{ u(n) \big \}_{n \in \mathbb{N}} \in \bigoplus\limits_{n \in \mathbb{N}} \mathcal{D}_n$, then for all $N \in \mathbb{N}$
\begin{equation} \label{eq; T is Dec not Diag}
\left \{ f(1) \left ( \sum^N_{k = 1}\lambda_k e_k \right ), 0, 0, 0, 0, ... \right \} = T \left( \left \{  \sum^N_{k = 1}\lambda_k e_k, 0, 0, 0, 0,... \right \} \right) = \left \{  \sum^N_{k = 1}k \lambda_k e_k, 0, 0, 0, 0, ... \right \}
\end{equation}
which is a contradiction. 

However, in view of Equation \eqref{eq; T is Dec not Diag}, we can see that the operator $T$ is diagonalizable if and only if the locally bounded operator $T_1$ (defined in Example \ref{ex;lbo}) is equal to $\lambda \cdot \mathrm{Id}_\mathcal{D}$  for some $\lambda \in \mathbb{C}$. 
\end{example}

We furnish the following example with the intention that such a decompoable, non diagonalizable locally bounded operator exists even when $\Lambda$ is an arbitrary (possibly uncountable) directed POSET.

\begin{example} \label{eg; Dec but not Diag 2}
Let $\Lambda = [0, \infty)$ and consider the locally standard measure space $\big (\mathbb{R}, B(\mathbb{R}), \mu \big )$ along with the family $\big \{ \big ( [-\alpha, \alpha], B \big ([-\alpha, \alpha] \big), \mu_\alpha \big ) \big \}_{\alpha \in [0, \infty)}$ of standard measure spaces. Corresponding to each $p \in \mathbb{R}$ and $\alpha \in [0, \infty)$ assign a Hilbert space $\mathcal{H}_{\alpha, p} = \text{L}^{2}(\mathbb{R}, \mu)$. Then for each $p \in \mathbb{R}$, we assign a locally Hilbert space $\mathcal{D}_{p} = \varinjlim\limits_{\alpha \in [0, \infty)} \mathcal{H}_{\alpha, p} = \text{L}^{2}(\mathbb{R}, \mu)$. Note that here each $\mathcal{D}_{p}$ is indeed a Hilbert space and for any $p, q \in \mathbb{R}$, we have $\mathcal{D}_{p} = \text{L}^{2}(\mathbb{R}, \mu) = \mathcal{D}_{q} $. Next, we consider the direct integral $\displaystyle \int^{\oplus_{\text{loc}}}_{\mathbb{R}}  \mathcal{D}_p \, \dmu$ which contains functions $u : \mathbb{R} \rightarrow \text{L}^{2}(\mathbb{R}, \mu)$ satisfying the properties listed in Definition \ref{Defn: directint_loc}. By using Equation \eqref{eq;dilhs=u}, we get $\displaystyle \int^{\oplus_{\text{loc}}}_{\mathbb{R}} \mathcal{D}_p \, \dmu = \bigcup\limits_{\alpha \in [0, \infty)} \mathcal{H}_\alpha$, where 
\begin{equation*}
\mathcal{H}_\alpha := \left \{ u \in \int^{\oplus_{\text{loc}}}_{\mathbb{R}} \mathcal{D}_p \, \dmu ~~ : ~~  \text{supp}(u) \subseteq [-\alpha, \alpha] \right \}.
\end{equation*}
Now define an operator $T : \displaystyle \int^{\oplus_{\text{loc}}}_{\mathbb{R}} \mathcal{D}_p \, \dmu \rightarrow \displaystyle \int^{\oplus_{\text{loc}}}_{\mathbb{R}} \mathcal{D}_p \, \dmu$ as 
\begin{equation*}
T \left ( u = \displaystyle \int^{\oplus_{\text{loc}}}_{\mathbb{R}} u(p) \, \dmu \right ) := \displaystyle \int^{\oplus_{\text{loc}}}_{\mathbb{R}} \hat{u}(p) \, \dmu,
\end{equation*}
where $\hat{u}(p) : \mathbb{R} \rightarrow \mathbb{C}$ is defined as $\hat{u}(p)(t) := u(p)(2t)$ for every $t \in \mathbb{R}$. Since $\text{supp}(u) = \text{supp}(Tu)$ for every $u \in \displaystyle \int^{\oplus_{\text{loc}}}_{\mathbb{R}} \mathcal{D}_p \, \dmu$ it follows that each $\mathcal{H}_\alpha$ is a reducing subspace for $T$. Hence $T$ is a locally bounded operator. Further, $T$ is decomposable. Because there is a family $\big \{ T_p \big \}_{p \in \mathbb{R}}$ of locally bounded operators (in fact, bounded linear operators) on $\mathcal{D}_p$ given by $T_p : \text{L}^{2}(\mathbb{R}, \mu) \rightarrow \text{L}^{2}(\mathbb{R}, \mu)$, where $T_p(f)(t) = f(2t)$ for every $f \in \text{L}^{2}(\mathbb{R}, \mu)$ and $t, p \in \mathbb{R}$ such that for each $u \in \displaystyle \int^{\oplus_{\text{loc}}}_{\mathbb{R}} \mathcal{D}_p \, \dmu$, we have
\begin{equation*}
(Tu)(p) = T_p(u(p))  \; \; \; \; \text{for almost every}  \; p \in \mathbb{R}.
\end{equation*}
Now we show that $T$ is not diagonalizable. Let $u : \mathbb{R} \rightarrow \text{L}^{2}(\mathbb{R}, \mu)$ be given by 
\begin{equation*}
u(p)(t) := \begin{cases}
1, & \text{if} \;\; t \in [-10, 10] \;,\; p \in [-1, 1];\\
0 & \text{otherwise}.
\end{cases}
\end{equation*}
Then $u \in \displaystyle \int^{\oplus_{\text{loc}}}_{\mathbb{R}} \mathcal{D}_p \, \dmu$. By following (\ref{def;Diag(lbo)}) of Definition \ref{def;DecDiag(lbo)}, suppose there exists a measurable function $f : \mathbb{R} \rightarrow \mathbb{C}$ such that 
\begin{equation} \label{Eq: Suppose T is diag}
(Tu)(p) = f(p)u(p),\; \text{for almost every}\; p \in \mathbb{R}. 
\end{equation}
By the definition of $T$, we get $Tu(p)(t) = 1$, whenever $p \in [-1, 1]$ and $t \in [-5, 5]$. From Equation \eqref{Eq: Suppose T is diag}, it follows that $f(p) = 1$ for almost every $p \in [-1, 1]$. On the other hand $Tu(p)(t) = 0$, whenever $p \in [-1, 1]$ and $t \in [-10, -5) \cup (5, 10]$. That is, $f(p) = 0$ for almost every $p \in [-1, 1]$ (from Equation \eqref{Eq: Suppose T is diag}). This is a contradiction. Therefore, $T$ is not diagonalizable.
\end{example}

Next, we present an example of a locally bounded operator which is not decomposable.

\begin{example} \label{eg; LBO but not Dec}
Let $\Lambda = \{ 1 \} \cup [2, \infty)$ and consider the locally standard measure space $\big (\mathbb{R}, B(\mathbb{R}), \mu \big )$ along with the family $\big \{ \big ( [-\alpha, \alpha], B \big ([-\alpha, \alpha] \big), \mu_\alpha \big ) \big \}_{\alpha \in \Lambda}$ of standard measure spaces. Corresponding to each $p \in \mathbb{R}$ and $\alpha \in \Lambda$ assign a Hilbert space $\mathcal{H}_{\alpha, p} = \mathbb{C}$. Then for each $p \in \mathbb{R}$, we assign a locally Hilbert space $\mathcal{D}_{p} = \varinjlim\limits_{\alpha \in \Lambda} \mathcal{H}_{\alpha, p} = \mathbb{C}$. Note that here each $\mathcal{D}_{p}$ is indeed a Hilbert space and for any $p, q \in \mathbb{R}$, we have $\mathcal{D}_{p} = \mathbb{C} = \mathcal{D}_{q} $. Then the direct integral 
\begin{equation*}
\displaystyle \int^{\oplus_{\text{loc}}}_{\mathbb{R}} \mathcal{D}_p \, \dmu = \big \{ u \in \text{L}^2 \big (\mathbb{R}, B \big (\mathbb{R} \big ), \mu \big ) \; : \; \text{supp}(u) \subseteq  [-\alpha, \alpha] \; \; \text{for some} \; \alpha \in \Lambda \big \},
\end{equation*}
which is dense in $\text{L}^2 \big (\mathbb{R}, B \big (\mathbb{R} \big ), \mu \big )$. In fact, $\displaystyle \int^{\oplus_{\text{loc}}}_{\mathbb{R}} \mathcal{D}_p \, \dmu = \bigcup\limits_{\alpha \in \Lambda} \mathcal{H}_{\alpha},$ where 
\begin{equation*}
    \mathcal{H}_{\alpha}:= \left\{u\in L^{2}(\mathbb{R}, B(\mathbb{R}), \mu);\; \text{supp}(u) \subseteq [-\alpha, \alpha] \right\}.
\end{equation*}Now we define  $T : \displaystyle \int^{\oplus_{\text{loc}}}_{\mathbb{R}} \mathcal{D}_p \, \dmu \rightarrow \displaystyle \int^{\oplus_{\text{loc}}}_{\mathbb{R}} \mathcal{D}_p \, \dmu$ by 
\begin{equation*}
T \left ( u = \displaystyle \int^{\oplus_{\text{loc}}}_{\mathbb{R}} u(p) \, \dmu \right ) := \int^{\oplus_{\text{loc}}}_{\mathbb{R}} \up{\chi}_{[\frac{-1}{2}, \frac{1}{2}]}(p)\; u(2p) \, \dmu 
\end{equation*}
It is immediate to see that $\text{supp}(Tu) \subseteq [\frac{-1}{2}, \frac{1}{2}]$ for every $u \in \displaystyle \int^{\oplus_{\text{loc}}}_{\mathbb{R}} \mathcal{D}_p \, \dmu$. Precisely, for every $\alpha \in \Lambda$,  $T(\mathcal{H}_{\alpha}) \subseteq \mathcal{H}_{1}$ and $T^{\ast}(\mathcal{H}_{\alpha}) \subseteq \mathcal{H}_{1}$ and so, $T$ is a locally bounded operator. Now we show that $T$ is not decomposable. 
Let $u, v$ be defined on $\mathbb{R}$ as follows
\begin{equation*}
u(p) := \begin{cases}
p, & \text{if} \;\; p \in [\frac{-1}{2}, \frac{1}{2}]; \\
0 & \text{otherwise}
\end{cases} \; \; \; \; \; \; 
v(p) := \begin{cases}
p, & \text{if} \;\; p \in [\frac{-1}{4}, \frac{1}{4}]; \\
0 & \text{otherwise}.
\end{cases}
\end{equation*}
Then $u, v \in \displaystyle \int^{\oplus_{\text{loc}}}_{\mathbb{R}} \mathcal{D}_p \, \dmu$. By the definition of $T$, we have 
\begin{equation*}
(Tu)(p) := \begin{cases}
2p, & \text{if} \;\; p \in [\frac{-1}{4}, \frac{1}{4}]; \\
0 & \text{otherwise}.
\end{cases} \; \; \; \; \; \; 
(Tv)(p) := \begin{cases}
2p, & \text{if} \;\; p \in [\frac{-1}{8}, \frac{1}{8}]; \\
0 & \text{otherwise}.
\end{cases}
\end{equation*}  
Suppose by (\ref{def;Dec(lbo)}) of Definition \ref{def;DecDiag(lbo)}, there exists a family $\big \{ T_p : \mathbb{C} \rightarrow \mathbb{C} \big \}_{p \in \mathbb{R}}$ of locally bounded (in fact, bounded) operators such that 
\begin{equation*}
(Tu)(p) = T_pu(p) \; \; \; \text{and} \; \; \; (Tv)(p) = T_pv(p)
\end{equation*}
for almost every $p \in \mathbb{R}$. Since the family $\big \{ T_p \big \}_{p \in \mathbb{R}}$ consists of bounded operators on $\mathbb{C}$, we can replace $T_p$ by $c_p$ for some $c_p \in \mathbb{C}$. Then for almost every $p \in [\frac{-1}{4}, \frac{-1}{8}] \cup [\frac{1}{8}, \frac{1}{4}]$, we get
\begin{equation*}
T_pu(p) = c_p p = (Tu)(p) = 2p,
\end{equation*}
that is $c_p = 2$. Whereas, for almost every $p \in [\frac{-1}{4}, \frac{-1}{8}] \cup [\frac{1}{8}, \frac{1}{4}]$, we obtain
\begin{equation*}
T_pv(p) = c_p p = (Tv)(p) = 0,
\end{equation*}
that is $c_p = 0$. This is a contradiction. Therefore, $T \notin \mathcal{M}_{\text{DEC}}$. Further, consider a measurable function $f : \mathbb{R} \rightarrow \mathbb{C}$ defined by $f(p) := p$ for all $p \in \mathbb{R}$ and define a digonalizable operator $S_f : \displaystyle \int^{\oplus_{\text{loc}}}_{\mathbb{R}} \mathcal{D}_p \, \dmu \rightarrow \displaystyle \int^{\oplus_{\text{loc}}}_{\mathbb{R}} \mathcal{D}_p \, \dmu$ corresponding to $f$. Then for almost every $p \in  [\frac{-1}{4}, \frac{1}{4}]$, we get
\begin{equation*}
TS_fu(p) = T(f(p)u(p)) = T(p^2) = 4p^2
\end{equation*}
while
\begin{equation*}
S_fTu(p) = S_fT(2p) = 2p^2.
\end{equation*}
This shows that $TS_f \neq S_fT$ and hence $T \notin \big (\mathcal{M}_{\text{DIAG}} \big )^\prime$.
\end{example}

\subsection{Observations II} The following key observations are useful to understand the notion of decomposable and diagonalizable locally bounded operators on $\displaystyle \dilX \mathcal{D}_p \, \dmu$.

\begin{enumerate} 
\item Let $T$ be a locally bounded operator on $\displaystyle \dilX \mathcal{D}_p \, \dmu$. If $T \in \mathcal{M}_\text{DEC}$, then unique up to a measure zero set there exists a family $\big \{T_p \big \}_{p \in X}$, where $T_p : \mathcal{D}_p \rightarrow \mathcal{D}_p$ is a locally bounded operator such that $T = \displaystyle \dilX T_p \, \dmu$.
\begin{proof}
Suppose $\{T_p\}_{p \in X}$ and $\{T^\prime_p\}_{p \in X}$ are two distinct families such that $\displaystyle \dilX T_p \, \dmu = T = \displaystyle \dilX T^\prime_p \, \dmu$. Then from point (3) of Definition \ref{Defn: directint_loc}, we get $T_p = T^\prime_p$ for almost every $p \in X$.
\end{proof}

\item \label{obs;M DEC and M DIAG are star algebras}
We know from Equation \eqref{eq;il=u} that $\displaystyle \dilX \mathcal{D}_p \, \dmu = \bigcup\limits_{\alpha \in \Lambda} \mathcal{H}_\alpha$. Thus the locally Hilbert space $\displaystyle \dilX \mathcal{D}_p \, \dmu$ is given by the strictly inductive system $\mathcal{E} = \{ \mathcal{H}_\alpha \}_{\alpha \in \Lambda}$ of Hilbert spaces. Recall that the collection of all locally bounded operators on $\displaystyle \dilX \mathcal{D}_p \, \dmu$ is denoted by $C^\ast_\mathcal{E} \left(\displaystyle \dilX \mathcal{D}_p \, \dmu \right)$, and from Example \ref{ex;lca}, we know that $C^\ast_\mathcal{E}\left(\displaystyle \dilX \mathcal{D}_p \, \dmu \right)$ is a locally $C^\ast$-algebra. Further, we have $\mathcal{M}_\text{DIAG} \subseteq \mathcal{M}_\text{DEC} \subseteq C^\ast_\mathcal{E}\left(\displaystyle \dilX \mathcal{D}_p \, \dmu \right)$. Also we get $\mathcal{M}_\text{DEC}$ is a locally convex $\ast$-subalgebra of $C^\ast_\mathcal{E}\left(\displaystyle \dilX \mathcal{D}_p \, \dmu \right)$ with respect to the following operations
\begin{multicols}{2}
\begin{enumerate}
\item $T + S = \displaystyle \dilX T_p + S_p \, \dmu$\\
\item $\lambda \cdot T = \displaystyle \dilX \lambda \cdot T_p\, \dmu$
\item $T \cdot S = \displaystyle \dilX T_p \cdot S_p \, \dmu$\\
\item $T^\ast = \displaystyle \dilX T^\ast_p \, \dmu$,
\end{enumerate}
\end{multicols}

\noindent
for every $T = \displaystyle \dilX T_p \, \dmu$ and $S = \displaystyle \dilX S_p \, \dmu$ in $\mathcal{M}_\text{DEC}$ and $\lambda \in \mathbb{C}$. 
Similarly, with respect to the above operations $\mathcal{M}_\text{DIAG}$ also forms a locally convex $\ast$-subalgebra of $C^\ast_\mathcal{E}\left(\displaystyle \dilX \mathcal{D}_p \, \dmu \right)$.

\item \label{obs; description of V alpha T V alpha star}
Supose $T \in \mathcal{M}_\text{DEC}$, then there is a family $\{T_p\}_{p \in X}$ such that $T = \displaystyle \dilX T_p \, \dmu$. We have $\displaystyle \dilX \mathcal{D}_p \, \dmu = \bigcup\limits_{\alpha \in \Lambda} \mathcal{H}_\alpha$ (see Equation \eqref{eq;dilhs=u}),
where 
\begin{equation*}
\mathcal{H}_\alpha := \left \{ u \in \dilX \mathcal{D}_p \, \dmu ~~ : ~~  \text{supp}(u) \subseteq X_\alpha, \; \;  u(p) \in \mathcal{H}_{\alpha, p} \; \; \text{for almost every} \; \; p \in X_\alpha \right \}.
\end{equation*}
\noindent
Now we consider for each $\alpha \in \Lambda$ the isomorphism $V_\alpha : \mathcal{H}_\alpha \rightarrow \int^\oplus_{X_\alpha} \mathcal{H}_{\alpha, p} \, \mathrm{d} \mu_\alpha (p)$ given by $V_\alpha(u)(p) := u(p)$ for all $p \in X_\alpha$ and $u \in \mathcal{H}_\alpha$ \big (refer Equation \eqref{eq;iso} \big ). Fix $\alpha \in \Lambda$ and  $V_\alpha (u) \in \int^\oplus_{X_\alpha} \mathcal{H}_{\alpha, p} \, \mathrm{d} \mu_\alpha (p)$ for some $u \in \mathcal{H}_\alpha$, then 
\begin{equation*}
V_\alpha T V^\ast_\alpha (V_\alpha(u))(p) =  V_\alpha (Tu)(p) = (Tu)(p) = T_p u(p) = T_p \big |_{\mathcal{H}_{\alpha, p}} u(p),
\end{equation*}
for almost every $p \in X_\alpha$. In other words, $V_\alpha T V_\alpha^\ast$ is a decompoable bounded operator on $\int^\oplus_{X_\alpha} \mathcal{H}_{\alpha, p} \, \mathrm{d} \mu_\alpha (p)$, that is,
\begin{equation} \label{eq;restriction of DecLBO}
V_\alpha T V_\alpha^\ast = \int^\oplus_{X_\alpha} T_{p} \big|_{\mathcal{H}_{\alpha, p}} \,  \mathrm{d} \mu_\alpha (p) \; \; \; \text{for every} \; \; \alpha \in \Lambda.
\end{equation}
Moreover, for each $\alpha \in \Lambda$, we obtain $\| V_\alpha T V_\alpha^\ast \| = \text{ess} \sup\limits_{p \in X_\alpha} \big \{ \big \| T_{p}\big|_{\mathcal{H}_{\alpha, p}} \big\| \big \} < \infty$.

In particular, if $T \in \mathcal{M}_\text{DIAG}$ with 
$T = \displaystyle \dilX f(p) \cdot \mathrm{Id}_{\mathcal{D}_{p}} \, \dmu$, for some measurable function $f : X \rightarrow \mathbb{C}$, then for each $\alpha \in \Lambda$, we have 
\begin{equation} \label{eq;restriction of DiagLBO}
V_\alpha T V_\alpha^\ast = \int^\oplus_{X_\alpha} f(p) \cdot \mathrm{Id}_{\mathcal{H}_{\alpha, p}} \,  \mathrm{d} \mu_\alpha (p),
\end{equation}
where $f \big|_{X_\alpha} \in \text{L}^\infty \big (X_\alpha, \mu_\alpha \big )$. As a result
$V_\alpha T V_\alpha^\ast$ is a diagonalizable bounded linear operator for each $\alpha \in \Lambda$.

\item If $T \in \mathcal{M}_\text{DIAG}$, then
$T = \displaystyle \dilX f(p) \cdot \mathrm{Id}_{\mathcal{D}_p}  \, \dmu$, for some measurable function $f : X \rightarrow \mathbb{C}$ and from the previous observation, we get that for each $\alpha \in \Lambda$, the bounded linear operator $V_\alpha T V_\alpha^\ast$ on $\int^\oplus_{X_\alpha} \mathcal{H}_{\alpha, p} \, \mathrm{d} \mu_\alpha$ is diagonalizable. By following (\ref{def;Diagbo}) of Definition \ref{def;Debo}, corresponding to each $V_\alpha T V_\alpha^\ast$ there is a function $f_\alpha \in \text{L}^\infty \big (X_\alpha, \mu_\alpha \big )$. As a result, to define diagonalizable locally bounded operator, one may think of considering the family of measurable measurable functions $ \big \{ f_\alpha  \in \text{L}^\infty \big (X_\alpha, \mu_\alpha \big ) \big \}_{\alpha \in \Lambda}$ such that
\begin{equation*}
(Tu)(p) = f_\alpha(p)u(p)  \; \; \text{for almost every} \; \; p \in X_\alpha
\end{equation*}
and for every $u \in \mathcal{H}_\alpha$. In that case, by using the fact that $\mathcal{H}_\alpha \subseteq \mathcal{H}_\beta$ (whenever $\alpha \leq \beta)$, we see that  
\begin{equation*}
f_\alpha(p)u(p) = (Tu)(p)  = f_\beta(p)u(p)
\end{equation*}
for almost every $p \in X_\alpha$ and for every $u \in \mathcal{H}_\alpha$. Thus, $f_\alpha(p) = f_\beta(p)$ for almost every $p \in X_\alpha$. On the other hand, if $\alpha, \beta \in \Lambda$ are not comparable, then there exists $\gamma \in \Lambda$ such that $\alpha \leq \gamma$ and $\beta \leq \gamma$. For any $u \in \mathcal{H}_\alpha \subseteq \mathcal{H}_\gamma$ and $v \in \mathcal{H}_\beta \subseteq \mathcal{H}_\gamma$, we have
\begin{equation*} 
f_\alpha(p)u(p) = (Tu)(p)  = f_\gamma(p)u(p) \; \; \; \; \text{and} \; \; \; \; f_\beta(q)v(q) = (Tv)(q)  = f_\gamma(q)v(q)
\end{equation*}
for almost every $p \in X_\alpha$ and for almost every $q \in X_\beta$. Consequently, we get $f_\alpha(p) = f_\gamma(p) = f_\beta(p)$ for almost every $p \in X_\alpha \cap X_\beta \cap X_\gamma$. 

Therefore, this shows that for any $\alpha, \beta \in \Lambda$, if $X_\alpha \cap X_\beta \neq \emptyset$, then $f_\alpha(p) = f_\beta(p)$ for almost every $p \in X_\alpha \cap X_\beta$. In view of this, by defining $f : X \rightarrow \mathbb{C}$ by $f(p) := f_\alpha(p)$, whenever $p \in X_\alpha$, we get $f$ to be measurable such that $f \big |_{X_\alpha} \in  \text{L}^\infty \big (X_\alpha, \mu_\alpha \big )$ for every $\alpha \in \Lambda$ and $(Tu)(p) = f(p) u(p)$ for almost every $p \in X$. Therefore, considering such family $ \big \{ f_\alpha  \in \text{L}^\infty \big (X_\alpha, \mu_\alpha \big ) \big \}_{\alpha \in \Lambda}$  is equivalent to saying that there is a measurable function $f : X \rightarrow \mathbb{C}$ as defined in (2) of Definition \ref{def;DecDiag(lbo)}.
\end{enumerate}

Next, we prove that $\mathcal{M}_{\text{DEC}}$ and $\mathcal{M}_{\text{DIAG}}$ are locally von Neumann algebras.

\begin{theorem} \label{thm;DEC and DIAG LvNA}
Let $\big( X, \Sigma, \mu \big) $ be a locally standard measure space and for each $p \in X$ assign a locally Hilbert space $\mathcal{D}_p$, where $\mathcal{D}_p = \varinjlim\limits_{\alpha \in \Lambda} \mathcal{H}_{\alpha, p}$. Then 
\begin{enumerate}
\item[(a)] $\mathcal{M}_{\text{DEC}}$ is a locally von Neumann algebra.
\item[(b)] $\mathcal{M}_{\text{DIAG}}$ is an abelian locally von Neumann algebra.
\end{enumerate}
\end{theorem}
\begin{proof}
Given that $\{ \mathcal{H}_{\alpha, p} \}_{\alpha \in \Lambda}$ is a strictly inductive system of Hilbert spaces for each $p \in X$. We denote the inclusion map $J_{\beta, \alpha, p} : \mathcal{H}_{\alpha, p} \rightarrow \mathcal{H}_{\beta, p}$, whenever $\alpha \leq \beta$ and $p \in X$. For a fixed $\alpha \in \Lambda$, let us consider
the Hilbert space $\int^\oplus_{X_\alpha} \mathcal{H}_{\alpha, p} \, \mathrm{d} \mu_\alpha$, the direct integral of Hilbert spaces $\{ \mathcal{H}_{\alpha, p} \}_{p \in X_\alpha}$. We denote the collection of all decomposable bounded linear operators and the collection of all diagonalizable bounded linear operators on $\int^\oplus_{X_\alpha} \mathcal{H}_{\alpha, p} \, \mathrm{d} \mu_\alpha$ by $\int^\oplus_{X_\alpha} \mathcal{B}(\mathcal{H}_{\alpha, p}) \, \mathrm{d} \mu_\alpha$ and $\int^\oplus_{X_\alpha} \mathbb{C} \cdot \mathrm{Id}_{\mathcal{H}_{\alpha, p}}   \, \mathrm{d} \mu_\alpha$ respectively. Then by Theorem 
\ref{thm;DeDibo vNA}, $\int^\oplus_{X_\alpha} \mathcal{B}(\mathcal{H}_{\alpha, p}) \, \mathrm{d} \mu_\alpha$ is a von Neumann algebra and $\int^\oplus_{X_\alpha} \mathbb{C} \cdot \mathrm{Id}_{\mathcal{H}_{\alpha ,p}}   \, \mathrm{d} \mu_\alpha$ is an abelian von Neumann algebra. We use these facts to prove our assertions.

Proof of (a): Firstly, note that for a fixed $p \in X$ and $T_{\beta, p} \in \mathcal{B}(\mathcal{H}_{\beta, p})$, the Hilbert space $\mathcal{H}_{\alpha, p}$ is not necessarily an invariant subspace for $T_{\beta, p}$  whenever $\alpha \leq \beta$. In view of this, we use the inclusion map $J_{\beta, \alpha, p}$ and define the map $\phi_{\alpha, \beta} : \int^\oplus_{X_\beta} \mathcal{B}(\mathcal{H}_{\beta, p}) \, \mathrm{d} \mu_\beta \rightarrow \int^\oplus_{X_\alpha} \mathcal{B}(\mathcal{H}_{\alpha, p}) \, \mathrm{d} \mu_\alpha$ by 
\begin{equation} \label{eq; phi alpha beta}
\phi_{\alpha, \beta} \left (\int^\oplus_{X_\beta} T_{\beta, p} \, \mathrm{d} \mu_\beta \right ) := \int^\oplus_{X_\alpha} J^\ast_{\beta, \alpha, p} T_{\beta, p} J_{\beta, \alpha, p} \, \mathrm{d} \mu_\alpha,
\end{equation}
whenever $\alpha \leq \beta$. Clearly, $\phi_{\alpha, \beta}$ is a surjective $\ast$-homomorphism. In particular $\phi_{\alpha, \alpha}$ is the identity map. Now we show that, $\phi_{\alpha, \beta}$ is a normal map between von Neumann algebras. To see this, let us consider a sequence $\left \{ \int^\oplus_{X_\beta} T^n_{\beta, p} \, \mathrm{d} \mu_\beta  \right \}_{n \in \mathbb{N}}$ in $\int^\oplus_{X_\beta} \mathcal{B}(\mathcal{H}_{\beta, p}) \, \mathrm{d} \mu_\beta$ such that $\int^\oplus_{X_\beta} T^n_{\beta, p} \, \mathrm{d} \mu_\beta \longrightarrow \int^\oplus_{X_\beta} T_{\beta, p} \, \mathrm{d} \mu_\beta$ in the weak operator topology. For any $x, y \in \int^\oplus_{X_\alpha} \mathcal{H}_{\alpha, p} \, \mathrm{d} \mu_\alpha$, define
\[ \widetilde{x}(p):= \left\{ \begin{array}{cc} x(p) & \mbox{if}\; p \in X_{\alpha};\\
& \\
0_{\mathcal{H}_{\beta, p}} & \mbox{if}\; p \in X_{\beta}\setminus X_{\alpha}\end{array}\right.
\; \; \text{and} \; \;
\widetilde{y}(p):= \left\{ \begin{array}{cc} y(p) & \text{if}\; p \in X_{\alpha};\\
&\\
0_{\mathcal{H}_{\beta, p}} & \text{if}\; p \in X_{\beta}\setminus X_{\alpha},\end{array}\right.
\]
then $\widetilde{x}, \widetilde{y} \in \int^\oplus_{X_\beta} \mathcal{H}_{\beta, p} \, \mathrm{d} \mu_\beta$ and we have 

\begin{align*}
\left\langle  x, \; \left ( \int^\oplus_{X_\alpha} J^{\ast}_{\beta, \alpha, p} \big(T^n_{\beta, p} - T_{\beta, p} \big) J_{\beta, \alpha, p} \, \mathrm{d} \mu_{\alpha} \right ) (y) \right\rangle &= \int_{X_\alpha} \la J_{\beta, \alpha, p}x(p), \; \big (T^n_{\beta, p} - T_{\beta, p} \big ) J_{\beta, \alpha, p} y(p) \ra_{\mathcal{H}_{\alpha, p}} \mathrm{d} \mu_\alpha(p) \\
&= \int_{X_\beta} \la \widetilde{x}(p), \; \big (T^n_{\beta, p} - T_{\beta, p} \big ) \widetilde{y}(p) \ra_{\mathcal{H}_{\beta, p}} \mathrm{d} \mu_\beta(p) \\
&= \left \langle \widetilde{x}, \; \left( \int^{\oplus}_{X_\beta}  \big (T^n_{\beta, p} - T_{\beta, p} \big) \; \mathrm{d} \mu_\beta(p) \right) \widetilde{y} \right \rangle\\
&\longrightarrow 0,
\end{align*}
as $n \to \infty.$ Thus $\phi_{\alpha, \beta}$ is a normal surjective $\ast$-homomorphism. Let $\alpha \leq \beta \leq \gamma$ and for any $\int^\oplus_{X_\gamma} T_{\gamma, p} \, \mathrm{d} \mu_\gamma \in \int^\oplus_{X_\gamma} \mathcal{B}(\mathcal{H}_{\gamma, p}) \, \mathrm{d} \mu_\gamma$, we get 
\begin{align*}
\phi_{\alpha, \beta} \left ( \phi_{\beta, \gamma} \left (\int^\oplus_{X_\gamma} T_{\gamma, p} \, \mathrm{d} \mu_\gamma \right ) \right ) &= \phi_{\alpha, \beta} \left ( \int^\oplus_{X_\beta} J^\ast_{\gamma, \beta, p} T_{\gamma, p} J_{\gamma, \beta, p} \, \mathrm{d} \mu_\beta \right ) \\
&= \int^\oplus_{X_\alpha} J^\ast_{\beta, \alpha, p} J^\ast_{\gamma, \beta, p} T_{\gamma, p} J_{\gamma, \beta, p} J_{\beta, \alpha, p} \, \mathrm{d} \mu_\alpha \\
&= \int^\oplus_{X_\alpha} J^\ast_{\gamma, \alpha, p} T_{\gamma, p} J_{\gamma, \alpha, p} \, \mathrm{d} \mu_\alpha \\
&= \phi_{\alpha, \gamma} \left (\int^\oplus_{X_\gamma} T_{\gamma, p} \, \mathrm{d} \mu_\gamma \right ).
\end{align*}
This implies $\left (  \left \{ \int^\oplus_{X_\alpha} \mathcal{B}(\mathcal{H}_{\alpha, p}) \, \mathrm{d} \mu_\alpha \right \}_{\alpha \in \Lambda}, \{ \phi_{\alpha, \beta} \}_{\alpha \leq \beta} \right )$ is a projective system of von Neumann algebras. We know from (\ref{obs;M DEC and M DIAG are star algebras}) of Observation II that $\mathcal{M}_{\text{DEC}}$ is a locally convex $\ast$-algebra. Since, for each $\alpha \in \Lambda$, there is a unitary operator $V_\alpha : \mathcal{H}_\alpha \rightarrow \int^\oplus_{X_\alpha} \mathcal{H}_{\alpha, p} \, \mathrm{d} \mu_\alpha$ given in Equation \eqref{eq;iso} and from Equation \eqref{eq;restriction of DecLBO}, it follows that the map $\phi_\alpha : \mathcal{M}_{\text{DEC}} \rightarrow  \int^\oplus_{X_\alpha} \mathcal{B}(\mathcal{H}_{\alpha, p}) \, \mathrm{d} \mu_\alpha$ defined by 
\begin{equation} \label{eq; phi alpha}
\phi_\alpha \left ( \displaystyle \dilX T_p \, \dmu \right ) := V_\alpha \left ( \displaystyle \dilX T_p \, \dmu \right) V_\alpha^\ast = \int^\oplus_{X_\alpha} T_{p} \big|_{\mathcal{H}_{\alpha, p}} \,  \mathrm{d} \mu_\alpha
\end{equation}
is a continuous $\ast$-homomorphism satisfying
\begin{align*}
\phi_{\alpha, \beta} \circ \phi_\beta \left ( \displaystyle \dilX T_p \, \dmu \right ) &= \phi_{\alpha, \beta} \left ( \int^\oplus_{X_\beta} T_{p} \big|_{\mathcal{H}_{\beta, p}} \,  \mathrm{d} \mu_\beta \right ) \\
&= \int^\oplus_{X_\alpha} J^\ast_{\beta, \alpha, p} T_p \big|_{\mathcal{H}_{\beta, p}} J_{\beta, \alpha, p}  \, \mathrm{d} \mu_\alpha \\
&= \int^\oplus_{X_\alpha} T_{p} \big|_{\mathcal{H}_{\alpha, p}} \,  \mathrm{d} \mu_\alpha \\
&= \phi_\alpha \left ( \displaystyle \dilX T_p \, \dmu \right ),
\end{align*}
whenever $\alpha \leq \beta$. This shows that the pair $ \big (\mathcal{M}_{\text{DEC}}, \{ \phi_\alpha \}_{\alpha \in \Lambda} \big )$ is compatible with the projective system $\left (  \left \{ \int^\oplus_{X_\alpha} \mathcal{B}(\mathcal{H}_{\alpha, p}) \, \mathrm{d} \mu_\alpha \right \}_{\alpha \in \Lambda}, \{ \phi_{\alpha, \beta} \}_{\alpha \leq \beta} \right )$  of von Neumann algebras \big ( see Subsection 1.1 of \cite{AG} \big ). Therefore, by the uniqueness of the projective limit, we get
\begin{align*}
\mathcal{M}_{\text{DEC}} = \varprojlim\limits_{\alpha \in \Lambda} \int^\oplus_{X_\alpha} \mathcal{B}(\mathcal{H}_{\alpha, p}) \, \mathrm{d} \mu_\alpha.
\end{align*}
Then by following \cite[Section 1]{MF} and the above discussion, we conclude that $\mathcal{M}_{\text{DEC}}$ is a locally von Neumann algebra.

Proof of (b): From Equation \eqref{eq; phi alpha beta}, we note that, if $T_{\beta, p} = c_{\beta, p} \cdot \mathrm{Id}_{\mathcal{H}_{\beta,p}}$, where $c_{\beta, p} \in \mathbb{C}$ for almost every $p \in X_\beta$, then 
\begin{equation*}
\phi_{\alpha, \beta} \left (\int^\oplus_{X_\beta} c_{\beta, p} \cdot \mathrm{Id}_{\mathcal{H}_{\beta,p}} \, \mathrm{d} \mu_\beta \right ) = \int^\oplus_{X_\alpha} J^\ast_{\beta, \alpha, p} c_{\beta, p} \cdot \mathrm{Id}_{\mathcal{H}_{\beta,p}} J_{\beta, \alpha, p} \, \mathrm{d} \mu_\alpha = \int^\oplus_{X_\alpha} c_{\beta, p} \cdot \mathrm{Id}_{\mathcal{H}_{\alpha, p}} \, \mathrm{d} \mu_\alpha.
\end{equation*}
This shows that $\phi_{\alpha, \beta} \left ( \int^\oplus_{X_\beta} \mathbb{C} \cdot \mathrm{Id}_{\mathcal{H}_{\beta,p}}  \, \mathrm{d} \mu_\beta \right ) = \int^\oplus_{X_\alpha} \mathbb{C} \cdot \mathrm{Id}_{\mathcal{H}_{\alpha ,p}} \, \mathrm{d} \mu_\alpha$. Thus the map ($\alpha \leq \beta$), $\psi_{\alpha, \beta} := \phi_{\alpha, \beta} \big |_{\int^\oplus_{X_\beta} \mathbb{C} \cdot \mathrm{Id}_{\mathcal{H}_{\beta,p}}  \, \mathrm{d} \mu_\beta }$ defines a surjective $\ast$-homomorphism and normal on $\int^\oplus_{X_\beta} \mathbb{C} \cdot \mathrm{Id}_{\mathcal{H}_{\beta,p}}  \, \mathrm{d} \mu_\beta$. In particular, $\psi_{\alpha, \alpha}$ is the identity map and $\psi_{\alpha, \beta} \circ \psi_{\beta, \gamma} = \psi_{\alpha, \gamma}$, whenever $\alpha \leq \beta \leq \gamma$. This means that $\left (  \left \{ \int^\oplus_{X_\alpha} \mathbb{C} \cdot \mathrm{Id}_{\mathcal{H}_{\alpha,p}}  \, \mathrm{d} \mu_\alpha \right \}_{\alpha \in \Lambda}, \{ \psi_{\alpha, \beta} \}_{\alpha \leq \beta} \right )$ is a projective system of abelian von Neumann algebras. 

On the other hand $\mathcal{M}_{\text{DIAG}}$ is a locally convex $\ast$-algebra (see (\ref{obs;M DEC and M DIAG are star algebras}) of Observations II). From Equation \eqref{eq;restriction of DiagLBO}, we know that $\phi_\alpha \big ( \mathcal{M}_{\text{DIAG}} \big ) = \int^\oplus_{X_\alpha} \mathbb{C} \cdot \mathrm{Id}_{\mathcal{H}_{\alpha ,p}} \, \mathrm{d} \mu_\alpha$. As a consequence, for each $\alpha \in \Lambda$, the map $\psi_\alpha := \phi_\alpha \big |_{\mathcal{M}_{\text{DIAG}}} : \mathcal{M}_{\text{DIAG}} \rightarrow \int^\oplus_{X_\alpha} \mathbb{C} \cdot \mathrm{Id}_{\mathcal{H}_{\alpha ,p}} \, \mathrm{d} \mu_\alpha$ is a continuous $\ast$-homomorphism. Moreover 
\begin{align*}
\psi_{\alpha, \beta} \circ \psi_\beta \left ( T =  \dilX f(p) \cdot \mathrm{Id}_{\mathcal{D}_p} \, \dmu \right )
&= \phi_{\alpha, \beta} \circ \phi_\beta \left ( \dilX f(p) \cdot \mathrm{Id}_{\mathcal{D}_p} \, \dmu \right ) \\
&= \phi_\alpha \left ( \dilX f(p) \cdot \mathrm{Id}_{\mathcal{D}_p} \, \dmu \right ) \\
&= \psi_\alpha \left ( \dilX f(p) \cdot \mathrm{Id}_{\mathcal{D}_p} \, \dmu \right ),
\end{align*}
whenever $\alpha \leq \beta$. This shows that the pair $ \big (\mathcal{M}_{\text{DIAG}}, \{ \psi_\alpha \}_{\alpha \in \Lambda} \big )$ is compatible with the projective system $\left (  \left \{ \int^\oplus_{X_\alpha} \mathbb{C} \cdot \mathrm{Id}_{\mathcal{H}_{\alpha,p}}  \, \mathrm{d} \mu_\alpha \right \}_{\alpha \in \Lambda}, \{ \psi_{\alpha, \beta} \}_{\alpha \leq \beta} \right )$ of abelian von Neumann algebras. Therefore, by the uniqueness of the projective limit, we get
\begin{align*}
\mathcal{M}_{\text{DIAG}} = \varprojlim\limits_{\alpha \in \Lambda} \int^\oplus_{X_\alpha} \mathbb{C} \cdot \mathrm{Id}_{\mathcal{H}_{\alpha,p}}  \, \mathrm{d} \mu_\alpha.
\end{align*}
Then again by following \cite[Section 1]{MF}, we conclude that $\mathcal{M}_{\text{DIAG}}$ is an abelian locally von Neumann algebra. 
\end{proof}

Motivated by the classical setup, where the abelian von Neumann algebra of diagonalizable bounded operators is the commutant of the von Neumann algebra of decomposable bounded operators (see Theorem \ref{thm;DeDibo vNA}), we explore in the remaining of this section the relationship between the locally von Neumann algebras $\mathcal{M}_{\text{DEC}}$ and $\big ( \mathcal{M}_{\text{DIAG}} \big )^ \prime$.

\begin{theorem} \label{thm; relation between M DEC and M DIAG}
Let $\big( X, \Sigma, \mu \big) $ be a locally standard measure space and for each $p \in X$ assign a locally Hilbert space $\mathcal{D}_p$, where $\mathcal{D}_p = \varinjlim\limits_{\alpha \in \Lambda} \mathcal{H}_{\alpha, p}$. Then we get
\begin{equation*}
\big ( \mathcal{M}_{\text{DIAG}} \big )^ \prime =  \varprojlim\limits_{\alpha \in \Lambda} \int^\oplus_{X_\alpha} \mathcal{B}(\mathcal{H}_{\alpha, p}) \, \mathrm{d} \mu_\alpha.
\end{equation*}
\end{theorem}
\begin{proof}
From the proof of Theorem \ref{thm;DEC and DIAG LvNA}, we recall that $\left (  \left \{ \int^\oplus_{X_\alpha} \mathcal{B}(\mathcal{H}_{\alpha, p}) \, \mathrm{d} \mu_\alpha \right \}_{\alpha \in \Lambda}, \{ \phi_{\alpha, \beta} \}_{\alpha \leq \beta} \right )$  is a projective system of von Neumann algebras given by decomposable bounded operators on $\int^\oplus_{X_\alpha} \mathcal{H}_{\alpha, p} \, \mathrm{d} \mu_\alpha$ for each $\alpha \in \Lambda$. 
Since $\mathcal{M}_{\text{DIAG}}$ is a locally von Neumann algebra in $C^\ast_\mathcal{E} \left(\displaystyle \dilX \mathcal{D}_p \, \dmu \right)$, the commutant defined by
\begin{equation*}
(\mathcal{M}_{\text{DIAG}})^\prime := \left \{ T \in C^\ast_\mathcal{E} \left(\displaystyle \dilX \mathcal{D}_p \, \dmu \right) \; \; : \; \; TS = ST \; \; \text{for all} \; \;   S  \in \mathcal{M}_{\text{DIAG}} \right \}
\end{equation*}
is also a locally convex $\ast$-subalgebra. Firstly, note from Theorem \ref{thm;DeDibo vNA} that the commutant 
\begin{equation} \label{eq; DeDibo vNA}
\left (  \int^\oplus_{X_\alpha} \mathbb{C} \cdot \mathrm{Id}_{\mathcal{H}_{\alpha,p}}  \, \mathrm{d} \mu_\alpha  \right )^\prime = \int^\oplus_{X_\alpha} \mathcal{B}(\mathcal{H}_{\alpha, p}) \, \mathrm{d} \mu_\alpha.
\end{equation} 
Next by assuming $T \in (\mathcal{M}_{\text{DIAG}})^\prime$ and considering the unitary operator $V_\alpha$ defined in Equation \eqref{eq;iso}, we show that the bounded operator $V_\alpha T V_\alpha^\ast$ on $\int^\oplus_{X_\alpha} \mathcal{H}_{\alpha, p} \, \mathrm{d} \mu_\alpha$ is decomposable for each $\alpha \in \Lambda$. Fix $\alpha \in \Lambda$ and $\int^\oplus_{X_\alpha} c_{\alpha, p} \cdot \mathrm{Id}_{\mathcal{H}_{\alpha,p}}  \, \mathrm{d} \mu_\alpha \in \int^\oplus_{X_\alpha} \mathbb{C} \cdot \mathrm{Id}_{\mathcal{H}_{\alpha,p}}  \, \mathrm{d} \mu_\alpha$ and define a function $f : X \rightarrow \mathbb{C}$ as 
\begin{equation*}
f(p) := \begin{cases}
c_{\alpha, p}, & \text{if} \;\; p \in X_\alpha; \\
o & \text{otherwise},
\end{cases} 
\end{equation*}  
then $f$ is measurable. Further, for a fixed $x = \int^\oplus_{X_\alpha} x(p)  \, \mathrm{d} \mu_\alpha \in \int^\oplus_{X_\alpha} \mathcal{H}_{\alpha,p}  \, \mathrm{d} \mu_\alpha$ consider the function $u : X \rightarrow \bigcup\limits_{p \in X} \mathcal{D}_p$ as
\begin{equation*}
u(p) := \begin{cases}
x(p), & \text{if} \;\; p \in X_\alpha; \\
o_{\mathcal{D}_p} & \text{otherwise}.
\end{cases} 
\end{equation*} 
It is immediate to see that $u = \dilX u(p) \, \mathrm{d} \mu  \in \mathcal{H}_\alpha$ (see Equation \eqref{eq; H alpha}). Since $T \in (\mathcal{M}_{\text{DIAG}})^\prime$, we get
\begin{align*}
\left ( V_\alpha T V^\ast_\alpha \right) \left (  \int^\oplus_{X_\alpha} c_{\alpha, p} \cdot \mathrm{Id}_{\mathcal{H}_{\alpha,p}}  \, \mathrm{d} \mu_\alpha \right ) \int^\oplus_{X_\alpha} x(p)  \, \mathrm{d} \mu_\alpha  &= \left ( V_\alpha T V^\ast_\alpha \right) \int^\oplus_{X_\alpha}  c_{\alpha, p} \cdot x(p)  \, \mathrm{d} \mu_\alpha \\
&= \left ( V_\alpha T V^\ast_\alpha \right) \left( V_\alpha \dilX f(p) \cdot u(p)  \, \mathrm{d} \mu \right) \\
&=  V_\alpha T \left( \dilX f(p) \cdot \mathrm{Id}_{\mathcal{D}_p}  \, \mathrm{d} \mu \right) \left( \dilX u(p) \, \mathrm{d} \mu \right) \\
&= V_\alpha \left( \dilX f(p) \cdot \mathrm{Id}_{\mathcal{D}_p}  \, \mathrm{d} \mu \right) T \left( \dilX u(p) \, \mathrm{d} \mu \right) \\
&= V_\alpha \left( \dilX f(p) \cdot \mathrm{Id}_{\mathcal{D}_p}  \, \mathrm{d} \mu \right) \left( \dilX (Tu)(p) \, \mathrm{d} \mu \right) \\
&= V_\alpha \left( \dilX f(p) \cdot (Tu)(p) \, \mathrm{d} \mu \right) \\
&= \int^\oplus_{X_\alpha} c_{\alpha, p} \cdot (Tu)(p)  \, \mathrm{d} \mu_\alpha \\
&= \left (  \int^\oplus_{X_\alpha} c_{\alpha, p} \cdot \mathrm{Id}_{\mathcal{H}_{\alpha,p}}  \, \mathrm{d} \mu_\alpha \right ) \left ( V_\alpha T V^\ast_\alpha \right) \int^\oplus_{X_\alpha} x(p)  \, \mathrm{d} \mu_\alpha.
\end{align*}
Since $\int^\oplus_{X_\alpha} x(p)  \, \mathrm{d} \mu_\alpha$, $\int^\oplus_{X_\alpha} c_{\alpha, p} \cdot \mathrm{Id}_{\mathcal{H}_{\alpha,p}}  \, \mathrm{d} \mu_\alpha$ and $\alpha \in \Lambda$ were arbitrarily chosen, by using Equation \eqref{eq; DeDibo vNA}, we get the bounded operator $V_\alpha T V_\alpha^\ast \in \int^\oplus_{X_\alpha} \mathcal{B} \big (\mathcal{H}_{\alpha, p} \big ) \, \mathrm{d} \mu_\alpha$ for every $\alpha \in \Lambda$. In particular, we denote the family by $\big \{ S_{\alpha, p} \big \}_{p \in X_\alpha}$ such that $V_\alpha T V^\ast_\alpha = \int^\oplus_{X_\alpha} S_{\alpha, p} \, \mathrm{d} \mu_\alpha$. This observation yields the following well defined map $\gamma_\alpha : (\mathcal{M}_{\text{DIAG}})^ \prime  
\rightarrow \int^\oplus_{X_\alpha} \mathcal{B}(\mathcal{H}_{\alpha, p}) \, \mathrm{d} \mu_\alpha$ given by
\begin{equation} \label{eq; gamma alpha}
\gamma_\alpha (T) = V_\alpha  T V_\alpha^*  \; \; \;  \text{for every} \; \; T \in (\mathcal{M}_{\text{DIAG}})^\prime.
\end{equation}
It follows from the definition that $\gamma_\alpha$ is a continuous map. Moreover, for $T \in (\mathcal{M}_{\text{DIAG}})^\prime$ and whenever $\alpha \leq \beta$, we have 
\begin{equation*}
\phi_{\alpha, \beta} \circ \gamma_\beta ( T )  = \phi_{\alpha, \beta} \left ( V_\beta  T V_\beta^* \right ) 
= \int^\oplus_{X_\alpha} J^\ast_{\beta, \alpha, p} S_{\beta, p} J_{\beta, \alpha, p} \, \mathrm{d} \mu_\alpha = \gamma_\alpha ( T ),
\end{equation*}
where the last equality holds true as the operator $T$ is locally bounded and the Hilbert space $\mathcal{H}_\alpha$ remains invariant under $T$.  Thus the pair $ \big ( \big (\mathcal{M}_{\text{DIAG}} \big )^ \prime, \{ \gamma_\alpha \}_{\alpha \in \Lambda} \big )$ is compatible with the projective system  $\left (  \left \{ \int^\oplus_{X_\alpha} \mathcal{B}(\mathcal{H}_{\alpha, p}) \, \mathrm{d} \mu_\alpha \right \}_{\alpha \in \Lambda}, \{ \phi_{\alpha, \beta} \}_{\alpha \leq \beta} \right )$ of von Neumann algebras. Therefore, by the uniqueness of the projective limit, we get
\begin{align*}
(\mathcal{M}_{\text{DIAG}})^ \prime = \varprojlim\limits_{\alpha \in \Lambda} \int^\oplus_{X_\alpha} \mathcal{B}(\mathcal{H}_{\alpha, p}) \, \mathrm{d} \mu_\alpha.
\end{align*}
This proves the theorem.
\end{proof}

\begin{remark} \label{rem; inverse limits of decomposable vNA}
An appeal to the Part (a) of Theorem \ref{thm;DEC and DIAG LvNA} and Theorem \ref{thm; relation between M DEC and M DIAG}, we can see that the locally von Neumann algebras $\mathcal{M}_{\text{DEC}}$ and $(\mathcal{M}_{\text{DIAG}})^\prime$ are both projective limits of projective system  $\left (  \left \{ \int^\oplus_{X_\alpha} \mathcal{B}(\mathcal{H}_{\alpha, p}) \, \mathrm{d} \mu_\alpha \right \}_{\alpha \in \Lambda}, \{ \phi_{\alpha, \beta} \}_{\alpha \leq \beta} \right )$ of von Neumann algebras. So there exists a unique continuous $\ast$-homomorphism $\Psi : \mathcal{M}_{\text{DEC}} \rightarrow (\mathcal{M}_{\text{DIAG}})^\prime$ such that
\begin{equation*}
\up{\gamma}_\alpha \circ \Psi = \phi_\alpha
\end{equation*}
for every $\alpha \in \Lambda$, where $\phi_\alpha$ is defined as in equation \eqref{eq; phi alpha} and $\gamma_\alpha$ is defined as in equation \eqref{eq; gamma alpha}.
\end{remark}

\subsection{Observations III} Let $\big( X, \Sigma, \mu \big) $ be a locally standard measure space and for each $p \in X$ assign a locally Hilbert space $\mathcal{D}_p$, where $\mathcal{D}_p = \varinjlim\limits_{\alpha \in \Lambda} \mathcal{H}_{\alpha, p}$. Then $\displaystyle \dilX \mathcal{D}_p \, \dmu$ is a locally Hilbert space with the strictly inductive system $\mathcal{E} = \big \{ \mathcal{H}_\alpha \big \}_{\alpha \in \Lambda}$. We recall that $C^\ast_\mathcal{E} \left ( \displaystyle \dilX \mathcal{D}_p \, \dmu  \right )$ denotes the collection of all locally bounded operators on $\displaystyle \dilX \mathcal{D}_p \, \dmu$. We record the following observations in order to realize the containment of the locally von Neumann algebras $\mathcal{M}_{\text{DEC}}$, $\mathcal{M}_{\text{DIAG}}$ and their commutants. 
\begin{enumerate}
\item \label{obs; M Dec and M Diag commutant}
$\mathcal{M}_{\text{DEC}} \subseteq  (\mathcal{M}_{\text{DIAG}})^\prime$ and $\mathcal{M}_{\text{DIAG}} \subseteq (\mathcal{M}_{\text{DEC}})^\prime$
\begin{proof}
Let $T \in \mathcal{M}_{\text{DEC}}$, then from \eqref{def;Dec(lbo)} of Definition \ref{def;DecDiag(lbo)} we get a family $ \big \{ T_p : \mathcal{D}_p \rightarrow \mathcal{D}_p \big \}_{p \in X}$ of locally bounded operators such that for any $u \in \displaystyle \dilX \mathcal{D}_p \, \dmu$, we have
\begin{align*}
(Tu)(p) = T_pu(p)
\end{align*}
for almost every $p \in X$. Further, if $S \in \mathcal{M}_{\text{DIAG}}$ then from \eqref{def;Diag(lbo)} of Definition \ref{def;DecDiag(lbo)} there is a measurable function $f : X \rightarrow \mathbb{C}$ such that for any $u \in \displaystyle \dilX \mathcal{D}_p \, \dmu$, we have 
\begin{equation*}
(Su)(p) = f(p)u(p)
\end{equation*}
for almost every $p \in X$. So for any $u = \displaystyle \dilX u(p) \, \dmu \in \displaystyle \dilX \mathcal{D}_p \, \dmu$, we get
\begin{align*}
\big (TS \big ) \left (\dilX u(p) \, \dmu \right ) = T \left(\dilX f(p)u(p) \, \dmu \right) &= \dilX T_p f(p) u(p) \, \dmu \\
&= \dilX  f(p) T_p u(p) \, \dmu \\
&= S \left (\dilX T_pu(p) \, \dmu \right) \\
&= \big  (ST \big ) \left(\dilX u(p) \, \dmu \right )
\end{align*}
Since $T \in \mathcal{M}_{\text{DEC}}, \; S \in \mathcal{M}_{\text{DIAG}}$ and $u \in \displaystyle \dilX \mathcal{D}_p \, \dmu$ were arbitrarily chosen, we obtain
$\mathcal{M}_{\text{DEC}} \subseteq (\mathcal{M}_{\text{DIAG}})^\prime.$ Consequently, $\mathcal{M}_{\text{DIAG}} \subseteq (\mathcal{M}_{\text{DEC}})^\prime$.
\end{proof}
\noindent In view of this, we conclude that the map $\Psi$ in Remark \ref{rem; inverse limits of decomposable vNA} is the inclusion map.


\item \label{obs; all contanments}
We have the following inclusion relations
\begin{equation*}
\mathcal{M}_{\text{DIAG}} \subseteq \mathcal{M}_{\text{DEC}} \subseteq (\mathcal{M}_{\text{DIAG}})^\prime \subseteq C^\ast_\mathcal{E} \left ( \displaystyle \dilX \mathcal{D}_p \, \dmu  \right ).
\end{equation*}
Note that, $\mathcal{M}_{\text{DIAG}} = \mathcal{M}_{\text{DEC}}$ in some cases. For instance, if $\mathcal{D}_p = \mathbb{C}$ for each $p \in X$. However, the inclusion can be strict \big (see Example \ref{eg; Dec but not Diag 1} and Example \ref{eg; Dec but not Diag 2}\big). Similarly, Example \ref{eg; LBO but not Dec} shows that $(\mathcal{M}_{\text{DIAG}})^\prime$ may not be same as $C^\ast_\mathcal{E} \left ( \displaystyle \dilX \mathcal{D}_p \, \dmu  \right )$. Also, it infers that $\mathcal{M}_{\text{DEC}}$ can be a proper subalgebra of $C^\ast_\mathcal{E} \left ( \displaystyle \dilX \mathcal{D}_p \, \dmu  \right )$.

\item In case of direct integral of Hilbert spaces, every bounded operator in the commutant of diagonalizable bounded operators is decomposable (see Theorem \ref{thm;DeDibo vNA}). While in this case, $\mathcal{M}_{\text{DEC}}$ and $(\mathcal{M}_{\text{DIAG}})^\prime$ are identified as the projective limit of the projective system $\left (  \left \{ \int^\oplus_{X_\alpha} \mathcal{B}(\mathcal{H}_{\alpha, p}) \, \mathrm{d} \mu_\alpha \right \}_{\alpha \in \Lambda}, \{ \phi_{\alpha, \beta} \}_{\alpha \leq \beta} \right )$ (see Remark \ref{rem; inverse limits of decomposable vNA}). But the set equality between $\mathcal{M}_{\text{DEC}}$ and $(\mathcal{M}_{\text{DIAG}})^\prime$ is not known yet. However, we have described a few cases where $\mathcal{M}_{\text{DEC}}$ coincides with $(\mathcal{M}_{\text{DIAG}})^\prime$ (see Theorem \ref{thm; M DEC = M DIAG Commutant}). 
\end{enumerate}


\begin{lemma}
Let $T \in C^\ast_\mathcal{E} \left ( \displaystyle \dilX \mathcal{D}_p \, \dmu  \right )$, where $\mathcal{E} = \big \{ \mathcal{H}_\alpha \big \}_{\alpha \in \Lambda}$. Then the family $\big \{ V_\alpha T V^\ast_\alpha \big \}_{\alpha \in \Lambda}$ of bounded operators satisfies the following relation,
\begin{equation*}
V_\alpha T^n V_\alpha^\ast = \big ( V_\alpha J^\ast_{\beta, \alpha} V_\beta^\ast \big ) \big ( V_\beta T^n V_\beta^\ast \big ) \big ( V_\beta J_{\beta, \alpha} V_\alpha^\ast \big ),
\end{equation*}
\text{whenever}\; $\alpha \leq \beta$, for all $n\in \mathbb{N}$. Here $J_{\beta, \alpha}\colon \mathcal{H}_{\alpha} \to \mathcal{H}_{\beta}$ is the inclusion map (for $\alpha \leq \beta).$
\end{lemma}
\begin{proof}
Recall that $\displaystyle \dilX \mathcal{D}_p \, \dmu = \varinjlim\limits_{\alpha \in \Lambda} \mathcal{H}_\alpha$, where $\mathcal{H}_
\alpha$ is defined as in Equation \eqref{eq; H alpha}. Let $x \in \int^\oplus_{X_\alpha} \mathcal{H}_{\alpha, p} \, \mathrm{d} \mu_\alpha$ and $\alpha \leq \beta$. By using the definitions of $V_\alpha$ and $V_\beta$, we have 
\begin{equation} \label{eq; V beta J beta alpha v alpha star}
V_\beta J_{\beta, \alpha} V_\alpha^\ast(x) = \begin{cases}
x(p), & \text{if} \;\; p \in X_\alpha; \\
& \\
0_{\mathcal{H}_{\beta, p}} & \text{if} \;\; p \in X_\beta \setminus X_\alpha
\end{cases}
\end{equation}
and so,
\begin{align*}
\big \| V_\beta J_{\beta, \alpha} V_\alpha^\ast(x) \big  \|^2& = \int_{X_\beta} \big \langle V_\beta J_{\beta, \alpha} V_\alpha^\ast(x)(p), \; V_\beta J_{\beta, \alpha} V_\alpha^\ast(x)(p) \big \rangle_{\mathcal{H}_{\beta, p}} \; \mathrm{d} \mu_\beta \\
&= \int_{X_\alpha} \big \langle x(p), \; x(p)  \big \rangle_{\mathcal{H}_{\alpha, p}} \; \mathrm{d} \mu_\alpha \\
&= \big \| x \big \|^2.
\end{align*}
This shows that $V_\beta J_{\beta, \alpha} V_\alpha^\ast$ is actually an isometry. Now by using the fact that $T$ is a locally bounded operator and $V_\beta$ is unitary, we obtain
\begin{align*}
\big ( V_\alpha T V_\alpha^\ast \big ) \big  (x \big ) = \big ( V_\alpha T \big ) \big ( V_\alpha^\ast x \big ) 
&= \big ( V_\alpha T J_{\beta, \alpha} \big ) \big ( V_\alpha^\ast x \big ) \\
&= \big ( V_\alpha J^\ast_{\beta, \alpha} T J_{\beta, \alpha} V_\alpha^\ast \big ) \big ( x \big ) \\
&= \big ( V_\alpha J^\ast_{\beta, \alpha} V_\beta^\ast \big ) \big ( V_\beta T V_\beta^\ast \big ) \big ( V_\beta J_{\beta, \alpha} V_\alpha^\ast \big ) \big ( x \big ).
\end{align*}
Since $x \in \int^\oplus_{X_\alpha} \mathcal{H}_{\alpha, p} \, \mathrm{d} \mu_\alpha$ was arbitrarily chosen, we get 
\begin{equation} \label{eq; dilation}
V_\alpha T V_\alpha^\ast = \big ( V_\alpha J^\ast_{\beta, \alpha} V_\beta^\ast \big ) \big ( V_\beta T V_\beta^\ast \big ) \big ( V_\beta J_{\beta, \alpha} V_\alpha^\ast \big ).
\end{equation}
Equivalently, the following diagram commutes.
\begin{center}
\begin{tikzcd}[sep=huge]
& \int^\oplus_{X_\beta} \mathcal{H}_{\beta, p} \, \mathrm{d} \mu_\beta \arrow[rr, "V_\beta T V_\beta^\ast"]  & & \int^\oplus_{X_\beta} \mathcal{H}_{\beta, p} \, \mathrm{d} \mu_\beta \\
& & \mathcal{H}_\beta  \arrow[ul, dotted, "V_\beta"] \arrow[ur, dotted, "V_\beta"']  & \\
& & \mathcal{H}_\alpha \arrow[u, dotted, "J_{\beta, \alpha}"] & \\
& \int^\oplus_{X_\alpha} \mathcal{H}_{\alpha, p} \, \mathrm{d} \mu_\alpha \arrow[ur, dotted, "V_\alpha^\ast"]  \arrow[uuu, "V_\beta J_{\beta, \alpha}V^\ast_\alpha"] \arrow[rr, "V_\alpha T V_\alpha^\ast"'] & &  \int^\oplus_{X_\alpha} \mathcal{H}_{\alpha, p} \, \mathrm{d} \mu_\alpha \arrow[ul, dotted, "V_\alpha^\ast"'] \arrow[uuu, "V_\beta J_{\beta, \alpha}V^\ast_\alpha"']
\end{tikzcd}
\end{center}
Finally, it follows from Equation \eqref{eq; dilation} that 
\begin{equation} \label{eq; power dilation}
V_\alpha T^n V_\alpha^\ast = \big ( V_\alpha J^\ast_{\beta, \alpha} V_\beta^\ast \big ) \big ( V_\beta T^n V_\beta^\ast \big ) \big ( V_\beta J_{\beta, \alpha} V_\alpha^\ast \big ),
\end{equation}
whenever $\alpha \leq \beta$ and $n \in \mathbb{N}$.
\end{proof}

\begin{theorem} \label{thm; M DEC = M DIAG Commutant}
Let $\big(\Lambda, \leq \big)$ be a directed POSET and $\big ( X, \Sigma, \mu \big )$ be a locally standard measure space. Then we get $\mathcal{M}_{\text{DEC}} = (\mathcal{M}_{\text{DIAG}})^\prime$ in the following two cases:
\begin{enumerate}
\item if $\Lambda$ is a countable set;
\item if $\mu$ is a counting measure on $X$.
\end{enumerate}
\end{theorem}
\begin{proof}
As we know from (\ref{obs; M Dec and M Diag commutant}) of Observations III, that $\mathcal{M}_{\text{DEC}} \subseteq (\mathcal{M}_{\text{DIAG}})^\prime$, in order to prove the result we show $(\mathcal{M}_{\text{DIAG}})^\prime \subseteq \mathcal{M}_{\text{DEC}}$. If $T \in (\mathcal{M}_{\text{DIAG}})^\prime$, then from the proof of Theorem \ref{thm; relation between M DEC and M DIAG} we know that $V_\alpha T V_\alpha^\ast$ on $\int^\oplus_{X_\alpha} \mathcal{H}_{\alpha, p} \, \mathrm{d} \mu_\alpha$ is decomposable for each $\alpha \in \Lambda$. This means that for each $\alpha \in \Lambda$, there is a family $\big \{ S_{\alpha, p} \in \mathcal{B} \big ( \mathcal{H}_{\alpha, p} \big ) \big \}_{p \in X_\alpha}$ of bounded operators such that 
\begin{equation*}
V_\alpha T V_\alpha^\ast = \int^\oplus_{X_\alpha} S_{\alpha, p} \, \mathrm{d} \mu_\alpha.
\end{equation*}
Further, we claim that $S_{\beta, p} \big |_{\mathcal{H}_{\alpha, p}} = S_{\alpha, p}$ for almost every $p \in X_\alpha$, whenever $\alpha \leq \beta$. Suppose there exists a measurable set $E_{\alpha, \beta} \subseteq X_\alpha$ such that $\mu_\alpha(E_{\alpha, \beta}) > 0$ and $S_{\beta, p} \big |_{\mathcal{H}_{\alpha, p}} \neq S_{\alpha, p}$ for every $p \in E_{\alpha, \beta}$. This implies that there exists a family $\big \{ \xi_{\alpha, p} \in \mathcal{H}_{\alpha, p} \big \}_{p \in E_{\alpha, \beta}}$ of vectors such that 
\begin{equation*}
S_{\beta, p}(\xi_{\alpha, p}) \neq S_{\alpha, p}(\xi_{\alpha, p})
\end{equation*}
for every $p \in E_{\alpha, \beta}$. Without loss of generality, we assume that the family $\big \{ \xi_{\alpha, p}  \in \mathcal{H}_{\alpha, p} \big \}_{p \in E_{\alpha, \beta}}$ consists of unit vectors. Now consider a family $\big \{ \hat{\xi}_{\alpha, p} \in \mathcal{H}_{\alpha, p} \big \}_{p \in X_\alpha}$ of vectors, where
\begin{equation*}
\hat{\xi}_{\alpha, p} = \begin{cases}
\xi_{\alpha, p}, & \text{if} \;\; p \in E_{\alpha, \beta}; \\
& \\
0_{\mathcal{H}_{\alpha, p}} & \text{if} \;\; p \in X_\alpha \setminus E_{\alpha, \beta}.
\end{cases}
\end{equation*}
Let $y \in \int^\oplus_{X_\alpha} \mathcal{H}_{\alpha, p} \, \mathrm{d} \mu_\alpha.$ Then the map $p \mapsto \|y(p)\|$ is in $L^{2}(X_{\alpha}, \mu_{\alpha})$ and so, it is in $L^{1}(X_{\alpha}, \mu_{\alpha})$ since $\mu_{\alpha}$ is a finite measure. It follows that
\begin{align*}
\int_{X_\alpha} \; \big | \la \hat{\xi}_{\alpha, p}, y(p)  \ra \big| \; \mathrm{d} \mu_\alpha(p) &\leq \int\limits_{E_{\alpha, \beta}} \; \big \| \xi_{\alpha, p} \big \| \; \big \| y(p) \big \| \; \mathrm{d} \mu_\alpha(p) \\
&= \int\limits_{E_{\alpha, \beta}} \; \big \| y(p) \big \|  \; \mathrm{d} \mu_\alpha(p) \\
&< \infty.
\end{align*}
As $y \in \int^\oplus_{X_\alpha} \mathcal{H}_{\alpha, p} \, \mathrm{d} \mu_\alpha$ was chosen arbitrarily, by following (2) of Definition \ref{def;dihs}, there exists $x \in \int^\oplus_{X_\alpha} \mathcal{H}_{\alpha, p} \, \mathrm{d} \mu_\alpha$ such that $x(p) = \hat{\xi}_{\alpha, p}$, for almost every $p \in X_\alpha$. Finally, by using Equation \eqref{eq; dilation} and Equation \eqref{eq; V beta J beta alpha v alpha star}, we get
\begin{align*}
\int^\oplus_{X_\alpha} S_{\alpha, p} x(p) \, \mathrm{d} \mu_\alpha &= \big ( V_\alpha T V_\alpha^\ast \big ) (x)\\
&= \big ( V_\alpha J^\ast_{\beta, \alpha} V_\beta^\ast \big ) \big ( V_\beta T V_\beta^\ast \big ) \big ( V_\beta J_{\beta, \alpha} V_\alpha^\ast \big ) (x)\\
&= \big ( V_\alpha J^\ast_{\beta, \alpha} V_\beta^\ast \big ) \big ( V_\beta T V_\beta^\ast \big ) \left ( \int^\oplus_{X_\beta} \big ( V_\beta J_{\beta, \alpha} V_\alpha^\ast \big ) (x) (p) \, \mathrm{d} \mu_\beta \right) \\
&= \big ( V_\alpha J^\ast_{\beta, \alpha} V_\beta^\ast \big ) \left ( \int^\oplus_{X_\beta} S_{\beta, p} \big ( V_\beta J_{\beta, \alpha} V_\alpha^\ast \big ) (x) (p) \, \mathrm{d} \mu_\beta \right) \\
&= \int^\oplus_{X_\alpha} S_{\beta, p} \big ( V_\beta J_{\beta, \alpha} V_\alpha^\ast \big ) (x) (p) \, \mathrm{d} \mu_\alpha \\
&= \int^\oplus_{X_\alpha} S_{\beta, p} x(p) \, \mathrm{d} \mu_\alpha.
\end{align*}
Thus $S_{\beta, p} x(p) = S_{\alpha, p} x(p)$ for almost every $p \in X_\alpha$. In particular, for almost every $p \in E_{\alpha, \beta}$, we have
\begin{equation*}
S_{\beta, p}(\hat{\xi}_{\alpha, p}) = S_{\beta, p}(\xi_{\alpha, p}) = S_{\beta, p}(x(p)) =  S_{\alpha, p}(x(p)) = S_{\alpha, p}(\xi_{\alpha, p}) = S_{\alpha, p}(\hat{\xi}_{\alpha, p})
\end{equation*}
This is a contradiction. This implies that $\mu_\alpha(E_{\alpha, \beta}) = 0$ and hence
\begin{equation*}
S_{\beta, p} \big |_{\mathcal{H}_{\alpha, p}} = S_{\alpha, p}
\end{equation*}
for almost every $p \in X_\alpha$, whenever $\alpha \leq \beta$.

\noindent
Proof of (1):Suppose $\Lambda$ is countable, then the set $E := \bigcup\limits_{\alpha, \beta \in \Lambda} E_{\alpha, \beta}$ (it is possible that $\alpha$ and $\beta$ are not comparable, in that case, consider $E_{\alpha, \beta} = \emptyset$) is a measurable set with $\mu(E) = 0$. Now consider the family $\big \{ S_{\alpha, p} \; : \; \alpha \in \Lambda, \, p \in X \setminus \bigcup\limits_{\alpha, \beta \in \Lambda} E_{ \alpha, \beta} \big\}$ of bounded operators. For each fixed $p \in X \setminus \bigcup\limits_{\alpha, \beta \in \Lambda} E_{\alpha, \beta}$, the family $\big \{  S_{\alpha, p}  \big \}_{\alpha \in \Lambda}$ is such that $S_{\beta, p} \big |_{\mathcal{H}_{\alpha, p}} = S_{\alpha, p}$, whenever $\alpha \leq \beta$. This yields a locally bounded operator $T_p : \mathcal{D}_p \rightarrow \mathcal{D}_p$ given by
\begin{align*}
T_p := \varprojlim\limits_{\alpha \in \Lambda} S_{\alpha, p}.
\end{align*}

Let $u \in \displaystyle \dilX \mathcal{D}_p \, \dmu$ be such that $u \in \mathcal{H}_\alpha$ for some $\alpha \in \Lambda$. This implies $Tu \in \mathcal{H}_\alpha$. Therefore, for almost every $p \in X \setminus X_\alpha$, we have $u(p) = Tu(p) = 0_{\mathcal{D}_p}$ and for almost every $p \in X_\alpha$, we get
\begin{equation*}
\big (Tu \big )(p) = \big (V_\alpha T u \big )(p) = \big (V_\alpha T V^\ast_\alpha \big )\big (V_\alpha u \big )(p) = S_{\alpha, p} u(p) = T_p u(p).
\end{equation*}
Since $u \in \displaystyle \dilX \mathcal{D}_p \, \dmu$ was chosen arbitrarily, we obtain 
\begin{align*}
(Tu)(p) = T_pu(p) 
\end{align*}
for every $u \in \displaystyle \dilX \mathcal{D}_p \, \dmu$ for almost every $p \in X$. This proves 
\begin{align*}
T = \dilX T_p \, \dmu.
\end{align*}
Since $T \in (\mathcal{M}_{\text{DIAG}})^\prime$ was chosen arbitrarily, we obtain $(\mathcal{M}_{\text{DIAG}})^\prime \subseteq \mathcal{M}_{\text{DEC}}$ and hence $\mathcal{M}_{\text{DEC}} =(\mathcal{M}_{\text{DIAG}})^\prime$.

\vspace{10pt} 
\noindent
Proof of (2): Now assume that $\mu$ is a counting measure on $X$. This gives, for every $p \in X_\alpha$, we have $S_{\beta, p} \big |_{\mathcal{H}_{\alpha, p}} = S_{\alpha, p}$whenever $\alpha \leq \beta$. Consider the family $\big \{ S_{\alpha, p} \; : \; \alpha \in \Lambda, \, p \in X \big\}$ of bounded operators. For each fixed $p \in X$, the family $\big \{  S_{\alpha, p}  \big \}_{\alpha \in \Lambda}$ is such that $S_{\beta, p} \big |_{\mathcal{H}_{\alpha, p}} = S_{\alpha, p}$, whenever $\alpha \leq \beta$. This yields a locally bounded operator $T_p : \mathcal{D}_p \rightarrow \mathcal{D}_p$ given by
\begin{align*}
T_p := \varprojlim\limits_{\alpha \in \Lambda} S_{\alpha, p} 
\end{align*}

Let $u \in \displaystyle \dilX \mathcal{D}_p \; \dmu$ then $u \in \mathcal{H}_\alpha$ for some $\alpha \in \Lambda$ and $Tu \in \mathcal{H}_\alpha$ since $T$ is a locally bounded operator. Therefore, for each $p \in X \setminus X_\alpha$, we have $u(p) = Tu(p) = 0_{\mathcal{D}_p}$ and for every $p \in X_\alpha$, we get
\begin{equation*}
\big (Tu \big )(p) = \big (V_\alpha T u \big )(p) = \big (V_\alpha T V^\ast_\alpha \big )\big (V_\alpha u \big )(p) = S_{\alpha, p} u(p) = T_p u(p).
\end{equation*}
Since $u \in \displaystyle \dilX \mathcal{D}_p \; \dmu$ was chosen arbitrarily, we obtain 
\begin{align*}
(Tu)(p) = T_pu(p) 
\end{align*}
for every $u \in \displaystyle \dilX \mathcal{D}_p \; \dmu$.
Hence
\begin{align*}
T = \dilX T_p \, \dmu.
\end{align*}
Since $T \in (\mathcal{M}_{\text{DIAG}})^\prime$ was chosen arbitrarily, we obtain $(\mathcal{M}_{\text{DIAG}})^\prime \subseteq \mathcal{M}_{\text{DEC}}$ and hence $\mathcal{M}_{\text{DEC}} =(\mathcal{M}_{\text{DIAG}})^\prime$.
\end{proof}

\section{Disintegration of a locally Hilbert space} \label{sec; Disintegration}

We have seen in Section \ref{sec; Direct integrals} that for a family $\big \{ \mathcal{D}_p \big \}_{p \in X}$ of locally Hilbert spaces and $\big( X, \Sigma, \mu \big)$ is a locally standard measure space (see Definition \ref{def; lsms}), the direct integral $\displaystyle \dilX \mathcal{D}_p \, \dmu$ is a locally Hilbert space such that $\mathcal{M}_{\text{DIAG}}$ \big (it is the collection of all diagonalizable locally bounded operators on $\displaystyle \dilX \mathcal{D}_p \, \dmu$ \big ) is an abelian locally von Neumann algebra (see Theorem \ref{thm;DEC and DIAG LvNA}). 

For a locally standard measure space $\big( X, \Sigma, \mu \big)$, we define a new class named ``locally essentially bounded measurable functions on $X$'' as, 
\begin{equation}\label{eq; EB loc}
\text{EB}_\text{loc} \big( X, \Sigma, \mu \big) := \Big \{ f : X \rightarrow \mathbb{C} \; \; : \; \; f \; \text{is measurable}, \; \;  f \big|_{X_\alpha} \in \text{L}^\infty \big ( X_\alpha, \mu_\alpha \big) \; \; \text{for each} \; \alpha \in \Lambda \Big \}.
\end{equation}
By following (\ref{def;Diag(lbo)}) of Definition \ref{def;DecDiag(lbo)} and using (\ref{obs; description of V alpha T V alpha star}) of Observations III, there is a one-to-one correspondence between $\text{EB}_\text{loc} \big( X, \Sigma, \mu \big)$ and $\mathcal{M}_{\text{DIAG}}$. This shows that $\text{EB}_\text{loc} \big( X, \Sigma, \mu \big)$ is an abelian locally von Neumann algebra.  

In this section, our aim is to prove the converse. Firstly, let us understand the notion of the converse in case of locally Hilbert spaces. Suppose $\mathcal{D}$ is the locally Hilbert space with a strictly inductive system $\mathcal{E} = \big \{ \mathcal{H}_\alpha \big \}_{\alpha \in \Lambda}$ is identified (through a bijective isometry) with the direct integral of some family of locally Hilbert spaces, then the associated locally standard measure space corresponds (via locally essentially bounded measurable functions) to an abelian locally von Neumann algebra in $C^\ast_\mathcal{E}(\mathcal{D})$. So, every identification corresponds to some abelian locally von Neumann algebra in $C^\ast_\mathcal{E}(\mathcal{D})$. In view of this, the converse question is framed as follows:

\textit{``Given a locally Hilbert space $\mathcal{D}$ with a strictly inductive system $\mathcal{E} = \big \{ \mathcal{H}_\alpha \big \}_{\alpha \in \Lambda}$ and an abelian locally von Neumann algebra $ \mathcal{M} \subseteq C^\ast_\mathcal{E}(\mathcal{D})$, does there exist a locally standard measure space $\big( X, \Sigma, \mu \big)$ and a family $\big \{ \mathcal{D}_p \big \}_{p \in X}$ of locally Hilbert spaces such that}
\begin{equation*}
\mathcal{D} = \displaystyle \dilX \mathcal{D}_p \, \dmu \; \; \text{and} \; \; \mathcal{M} = \mathcal{M}_\text{DIAG} ?
\end{equation*}

We answer this question in the case of Fr\'echet space with the assumption of an extra condition. We write the condition here. \\

\noindent
\textbf{Condition I:} An abelian locally von Neumann algebra $\mathcal{M} \subseteq C^\ast_\mathcal{E}(\mathcal{D})$ is of the form $\mathcal{M} = \varprojlim\limits_{n \in \mathbb{N}} \mathcal{M}_n$, where $\mathcal{M}_n$ is isomorphic to $\text{L}^\infty \big ( X_n, \Sigma_n, \mu_n \big )$ for every $n \in \mathbb{N}$, the family $\big \{ \big (X_n, \Sigma_n \big ) \big \}_{n \in \mathbb{N}}$  forms a strictly inductive system of measurable spaces along with a projective system of finite measures $\big \{\mu_n \big \}_{n \in \mathbb{N}}$.


\begin{theorem} \label{thm;dlhs}
Let $\mathcal{D}$ be a locally Hilbert space which is the inductive limit of the strictly inductive system $\mathcal{E}=\{\mathcal{K}_{n}\}_{n \in \mathbb{N}}$ of Hilbert spaces. If $\mathcal{M} = \varprojlim\limits_{n \in \mathbb{N}} \mathcal{M}_n$ is an abelian locally von Neumann algebra in $C^\ast_\mathcal{E}(\mathcal{D})$ satisfying \textbf{Condition I}, then there exists
\begin{enumerate}
    \item[(i)] a locally standard measure space $\big (X, \Sigma, \mu \big )$;
    \item[(ii)] a family  $\big \{ \mathcal{D}_p \big \}_{p \in X}$ of locally Hilbert spaces 
\end{enumerate}
such that  $\mathcal{D}$ is isomorphic to $\displaystyle \dilX \mathcal{D}_p \, \dmu$ and the locally von Neumann algebra $\mathcal{M}$ is isomorphic to the locally von Neumann algebra of all daigonalizable operators on $\displaystyle \dilX \mathcal{D}_p \, \dmu$.
\end{theorem}

We prove the following five intermediate lemmas to construct the proof of the main theorem. 

\begin{lemma} \label{lem;flhs}
Let  $\big \{ \mathcal{K}_n \big \}_{n \in \mathbb{N}}$ be a strictly inductive system of Hilbert spaces and $\mathcal{D} = \bigcup\limits_{n \in \mathbb{N}} \mathcal{K}_n$ be a locally Hilbert space. If $\mathcal{M}$ is an abelian locally von Neumann algebra satisfying the \textbf{Condition I}, then there exists a locally standard measure space $\big (X, \Sigma, \mu \big )$ and a family $\big \{ \mathcal{D}_p \big \}_{p \in X}$ of locally Hilbert spaces.
\end{lemma}
\begin{proof}
Let $\mathcal{K}$ be the Hilbert space obtained by completing the locally Hilbert space $\mathcal{D}$. For a fixed $n \in \mathbb{N}$, let $P_n : \mathcal{K} \rightarrow \mathcal{K}_n$ be the projection of the Hilbert space $\mathcal{K}$ onto $\mathcal{K}_n$. Consider $\mathcal{Z} \subseteq \mathcal{D}$, a countable, dense subset such that $\mathcal{Z}$ is closed under linear combinations with rational complex coefficients. Moreover, we get $\mathcal{Z} \cap \mathcal{K}_n$ is dense in $\mathcal{K}_n$ for all $n \in \mathbb{N}$ and $\mathcal{Z} \cap \mathcal{K}_m \subseteq \mathcal{Z} \cap \mathcal{K}_n$ whenever $m \leq n$. We fix the notation $\mathcal{Z}_n = \mathcal{Z} \cap \mathcal{K}_n$ and thus $\mathcal{Z} = \bigcup\limits_{n \in \mathbb{N}} \mathcal{Z}_n$. Next, for each $n \in \mathbb{N}$, consider the set 
\begin{align*}
\mathcal{Z}^\prime_n := \mathcal{Z}_n ~~ \bigcup ~~ \big \{ P_nh_j ~~ : ~~ h_j \in \mathcal{Z}_r, ~~\text{whenever} ~~ n \leq r \big \} 
\end{align*} 
and then define
\begin{align*}
\mathcal{Z}^\prime := \bigcup_{n \in \mathbb{N}} \mathcal{Z}^\prime_n.
\end{align*}
Since for each $n \in \mathbb{N}$, the set $\mathcal{Z}_n$ is countable, we get $\mathcal{Z}^\prime_n$ is also a countable set. This shows $\mathcal{Z}^\prime$ is countable as well. Finally, consider the set $\widetilde{\mathcal{Z}}$ consisting of the collection of all linear combinations of vectors from the set $\mathcal{Z}^\prime$ with rational complex coefficients. This implies $\widetilde{\mathcal{Z}}$ is a also a countable, dense subset of $\mathcal{D}$. We denote the collection of all elements of $\widetilde{\mathcal{Z}}$ by $ \big \{ \xi_j ~~ : ~~ j \in \mathbb{N} \big \}$. 

Given that $\mathcal{D} = \bigcup\limits_{n \in \mathbb{N}} \mathcal{K}_n$ and $\mathcal{M} = \varprojlim\limits_{n} \mathcal{M}_n$ satisfying condition I. That is, there is a measure space $\big ( X_n, \Sigma_n, \mu_n \big )$ such that $\mathcal{M}_n$ is isomorphic to $\text{L}^\infty \big ( X_n, \Sigma_n, \mu_n \big )$ for every $n \in \mathbb{N}$ with the property that the family $\big \{ \big (X_n, \Sigma_n \big ) \big \}_{n \in \mathbb{N}}$  is a strictly inductive system of measurable spaces and the family $\big \{\mu_n \big \}_{n \in \mathbb{N}}$ is a projective system of finite measures. By taking $X = \bigcup\limits_{n \in \mathbb{N}} X_n$, $\Sigma$ as defined in Equation \eqref{eq;sigma algebra} and $\mu$ as defined in Proposition \ref{prop;m}, we get a locally standard measure space $\big ( X, \Sigma, \mu \big )$ (see Definition \ref{def; lsms} ). For a fixed $n \in \mathbb{N}$, we know that $\mathcal{M}_n$ is an abelian von Neumann algebra in $\mathcal{B}(\mathcal{K}_n)$. So, let $\mathcal{A}_n$ be a dense C*-algebra contained in $\mathcal{M}_n$, and let $\Gamma_n : C(X_n) \rightarrow \mathcal{A}_n$ be the Gelfand isomorphism.  By the notion of disintegration of a separable Hilbert space (see Theorem \ref{thm;dihs}), we obtain a family $\{ \mathcal{K}_{n, p} \}_{p \in X_n}$ of separable Hilbert spaces with respect to the abelian von Neumann subalgebra 
$\mathcal{M}_n$ such that the Hilbert space $\mathcal{K}_n$ is isomorphic to the direct integral of $\{ \mathcal{K}_{n, p} \}_{p \in X_n}$. That is
\begin{align*}
\mathcal{K}_n \cong \int^\oplus_{X_n} \mathcal{K}_{n, p} \, \mathrm{d} \mu_n(p).
\end{align*}
For a fixed $m \in \mathbb{N}$ and $\xi_j, \xi_k \in \widetilde{V}$ define a linear functional $\psi_m^{\xi_j, \xi_k} : C(X_m) \rightarrow \mathbb{C}$ by
\begin{align*}
\psi_m^{\xi_j, \xi_k}(f) := \big \langle P_m\xi_j, \Gamma_m(f)P_m\xi_k \big \rangle
\end{align*} 
Since $\psi_m^{\xi_j, \xi_k}$ is a bounded linear functional, then by Riesz-representation theorem, there exists a measure, denoted by $\mu^{\xi_j, \xi_k}_{m}$ on $X_m$ satisfying, $\mu^{\xi_j, \xi_k}_{m} \ll \mu_m$ and 
\begin{equation*}
    \psi_m^{\xi_j, \xi_k}(f) = \big \langle P_m\xi_j, \Gamma_m(f)P_m\xi_k \big \rangle = \int\limits_{X_{m}} f\; d\mu_{m}^{\xi_{j}, \xi_{k}},\; \text{for every}\; f\in C(X_{m}).
\end{equation*}
Further, by using \cite[Part I,  Chapter 7, Proposition 1]{DixV} and the facts that the $C^\ast$-algebras $C(X_m)$ and $\mathcal{A}_m$ are dense in the von Neumann algebras $\text{L}^\infty \big ( X_m, \Sigma_m, \mu_m \big )$ and 
$\mathcal{M}_n$ respectively, the Gelfand isomorphism $\Gamma_m$ can be uniquely extended to a bijective homomorphism $\widehat{\Gamma}_m : \text{L}^\infty \big ( X_m, \Sigma_m, \mu_m \big ) \rightarrow \mathcal{M}_m$ such that
\begin{equation*}
\big \langle P_m\xi_j, \widehat{\Gamma}_m(f)P_m\xi_k \big \rangle = \int\limits_{X_{m}} f\; d\mu_{m}^{\xi_{j}, \xi_{k}},\; \; \text{for every}\; f\in \text{L}^\infty \big ( X_m, \Sigma_m, \mu_m \big ).
\end{equation*}
By using the fact that $\mu^{\xi_j, \xi_k}_{m} \ll \mu_m$ there is a unique function (the Radon--Nikodym derivative) $f^{\xi_j, \xi_k}_{m} \in \text{L}^1 \big (X_m, \mu_m \big )$  such that for any $E \in \Sigma_m$, we get
\begin{align*}
\mu^{\xi_j, \xi_k}_{m}(E) = \int_{E} f^{\xi_j, \xi_k}_{m} \, \mathrm{d} \mu_m. 
\end{align*}
It is clear that the function $f^{\xi_j, \xi_k}_{m}$ is defined almost everywhere. Since $\widetilde{\mathcal{Z}}$ is countable, there are countably many such functions. Consequently, we remove the measure zero set from $X_m$ corresponding to each such function $f^{\xi_j, \xi_k}_{m}$, where it is not defined. However, without loss of generality, we continue to work with the modified set $X_m$, and thus, $X$. 

Next, fix $p \in X$. We can choose the smallest $m \in \mathbb{N}$ such that $p \in X_m$ (i.e. $p \notin X_l$ for any $l < m$). Further, we define the map $\phi_p : \widetilde{\mathcal{Z}} \times \widetilde{\mathcal{Z}} \rightarrow \mathbb{C}$ by
\begin{equation} \label{eq; sesquilinear form}
\phi_p (\xi_j, \xi_k) := f^{\xi_j, \xi_k}_{m}(p).
\end{equation} 
For arbitrarily fixed $\xi_j, \xi_k, \xi_l \in \widetilde{\mathcal{Z}}$, $\lambda \in \mathbb{C}$ and $E \in \Sigma_m$, observe that
\begin{align*}
\int_{E} f^{\xi_j, \;  \lambda \xi_k + \xi_l}_{m} \, \mathrm{d} \mu_m = \mu^{\xi_j, \; \lambda \xi_k + \xi_l}_{m}(E) 
&= \int_{X_m} \up{\chi_E} \, \mathrm{d} \mu^{\xi_j, \;  \lambda \xi_k + \xi_l}_m \\
&= \big \langle P_m\xi_j, \; \widehat{\Gamma}_m(\up{\chi_E})P_m(\lambda \xi_k + \xi_l) \big \rangle \\
&= \lambda \big \langle P_m\xi_j, \; \widehat{\Gamma}_m(\up{\chi_E})P_m\xi_k \big \rangle + \big \langle P_m\xi_j, \; \widehat{\Gamma}_m(\up{\chi_E})P_m\xi_l \big \rangle \\
&= \lambda \int_{X_m} \up{\chi_E} \, \mathrm{d} \mu^{\xi_j, \;  \xi_k}_m + \int_{X_m} \up{\chi_E} \, \mathrm{d} \mu^{\xi_j, \;  \xi_l}_m \\
&= \lambda \mu^{\xi_j,  \xi_k}_{m}(E) + \mu^{\xi_j,  \xi_l}_{m}(E) \\
&= \lambda \int_{E} f^{\xi_j, \; \xi_k}_{m} \, \mathrm{d} \mu_m + \int_{E} f^{\xi_j, \; \xi_l}_{m} \, \mathrm{d} \mu_m.
\end{align*}
Since this is true for all $E \in \Sigma_m$, we get 
\begin{equation*}
f^{\xi_j, \; \lambda \xi_k + \xi_l}_{m}(p) = \lambda f^{\xi_j,  \xi_k}_{m}(p) + f^{\xi_j,  \xi_l}_{m}(p)
\end{equation*}
for almost every $p \in X_m$. Thus we have
\begin{equation} \label{eq; sesquilinear form 1}
\phi_p(\xi_j, \lambda \xi_k + \xi_l) = \lambda  \phi_p(\xi_j, \xi_k) +  \phi_p(\xi_j, \xi_l) \; \; \text{for almost every} \; \; p \in X_m.
\end{equation}
Further, for $\xi_j \in \widetilde{\mathcal{Z}}$,
\begin{equation} \label{eq; sesquilinear form 2}
\phi_p(\xi_j, \xi_j) = f^{\xi_j, \xi_j}_{m}(p) \geq 0 \; \; \text{for almost every} \; \; p \in X_m.
\end{equation}
Next, for an arbitrarily fixed $E \in \Sigma_m$, we observe
\begin{align*}
\int_{E} f^{\xi_j, \;  \xi_k}_{m} \, \mathrm{d} \mu_m = \mu^{\xi_j, \; \xi_k}_{m}(E) 
= \int_{X_m} \up{\chi_E} \, \mathrm{d} \mu^{\xi_j, \;  \xi_k}_m 
&= \big \langle P_m\xi_j, \; \widehat{\Gamma}_m(\up{\chi_E})P_m\xi_k \big \rangle \\
&= \overline{\big \langle \widehat{\Gamma}_m(\up{\chi_E})P_m\xi_k,  \; P_m\xi_j \big \rangle} \\
&= \overline{\big \langle P_m\xi_k,  \; \widehat{\Gamma}_m(\up{\chi_E}) P_m\xi_j \big \rangle} \\
&= \overline{\int_{X_m} \up{\chi_E} \, \mathrm{d} \mu^{\xi_k, \;  \xi_j}_m} \\
&= \overline{\mu^{\xi_k, \; \xi_j}_{m}(E)} \\
&= \overline{\int_{E} f^{\xi_k, \;  \xi_j}_{m} \, \mathrm{d} \mu_m}.
\end{align*}
Since this is true for all $E \in \Sigma_m$, we get 
\begin{equation*}
f^{\xi_j, \; \xi_k}_{m}(p) = \overline{f^{\xi_k,  \xi_j}_{m}(p)}
\end{equation*}
for almost every $p \in X_m$. Thus we have
\begin{equation} \label{eq; sesquilinear form 3}
\phi_p(\xi_j, \xi_k) = \overline{\phi_p(\xi_k, \xi_j)} \; \; \text{for almost every} \; \; p \in X_m.
\end{equation}
We remove the measure zero sets from $X_m$ corresponding to Equation \eqref{eq; sesquilinear form 1}, Equation \eqref{eq; sesquilinear form 2} and Equation \eqref{eq; sesquilinear form 3}, where the equalities are not attained. However, note that since $\widetilde{\mathcal{Z}}$ is a countable set, again without loss of generality, we continue to work with the modified set $X_m$, and hence $X$. 
The above discussion shows that the map $\phi_p$ defined in Equation \eqref{eq; sesquilinear form} is a positive hermitian sesquilinear form. For the same $p \in X$, let us consider the set 
\begin{align*}
\mathcal{N}_p = \Big\{ \xi_j \in \widetilde{\mathcal{Z}} \; \; : \; \; \phi_p(\xi_j, \xi_j) =  f^{\xi_j, \xi_j}_{m}(p) = 0 \Big\}.
\end{align*}
We can see that $\mathcal{N}_p$ is same as $\Big\{ \xi_j \in \widetilde{\mathcal{Z}} \; \; : \; \; \phi_p(\xi_j, \xi_k) =  0 \; \text{for all} \; \xi_k \in \widetilde{\mathcal{Z}} \Big\}$. It follows that $\mathcal{N}_p$ is a subspace of $\widetilde{\mathcal{Z}}$ and for $\xi_j, \xi_k \in \widetilde{\mathcal{Z}}$,
\begin{equation*}
\left \langle \xi_j + \mathcal{N}_p, \; \xi_k + \mathcal{N}_p \right \rangle := \phi_p(\xi_j, \xi_k) =  f^{\xi_j, \xi_k}_{m}(p)
\end{equation*}
defines an inner product on the quotient space $\widetilde{\mathcal{Z}} / \mathcal{N}_p$. Let $\mathcal{H}_p$ be the Hilbert space obtained by the completion of the inner product space $\widetilde{\mathcal{Z}} / \mathcal{N}_p$. For each $n \in \mathbb{N}$, consider the subspace $\mathcal{H}_{n, p}$ of $\mathcal{H}_p$ given by
\begin{equation} \label{eq; Hilbert space H n p}
\mathcal{H}_{n, p} := \overline{\text{span}} \big \{ \big (P_m h_j - P_{m-1} h_j \big ) + \mathcal{N}_p \; \;  : \; \; h_j \in \mathcal{Z}_n \big \}.
\end{equation}
Note that if $l < m$, then $\mathcal{H}_{l, p} = \mathcal{N}_p$ and $\mathcal{H}_{n, p} \subseteq \mathcal{H}_{r, p}$, whenever $n \leq r$. This implies that $\big\{ \mathcal{H}_{n, p} \big \}_{n \in \mathbb{N}}$ is a strictly inductive system $\big\{ \mathcal{H}_{n, p} \big \}_{n \in \mathbb{N}}$ of Hilbert spaces. Finally, we define the locally Hilbert space $\mathcal{D}_p := \bigcup\limits_{n \in \mathbb{N}} \mathcal{H}_{n, p}$. Since $p \in X$ was chosen arbitrarily, we obtain a  family  $\{ \mathcal{D}_p \}_{p \in X}$  of locally Hilbert spaces.
\end{proof}

\begin{lemma} \label{lem; condition 2}
The family $\big \{ \widehat{\Gamma}_n : \text{L}^\infty \big ( X_n, \Sigma_n, \mu_n \big ) \rightarrow \mathcal{M}_n \;  :  \; n \in \mathbb{N} \big\}$ of isomorphisms is such that $\widehat{\Gamma}_n(f) \big |_{\mathcal{H}_m} = 0$, whenever $f \big |_{X_m} = 0$ for $1 \leq m \leq n$.
\end{lemma}
\begin{proof}
It is given in \textbf{Condition I} that the abelian locally von Neumann algebra $\mathcal{M}$ is the projective limit of the projective system $\left ( \{ \mathcal{M}_n \}_{n \in \mathbb{N}},  \{\psi_{m,n} \}_{m \leq n} \right )$ of von Neumann algebras. Since $\big ( X_m, \Sigma_m, \mu_m \big )$ is a locally standard measure space, for $m \leq n$, the map $\tau_{m,n}: \text{L}^\infty \big ( X_n, \Sigma_n, \mu_n \big )  \rightarrow \text{L}^\infty \big ( X_m, \Sigma_m, \mu_m \big )$ defined by 
\begin{equation*}
\tau_{m, n}(f) := f \big |_{X_m}
\end{equation*}
is a sujective normal $\ast$-homomorphism. It follows that $\big ( \text{L}^\infty \big ( X_n, \Sigma_n, \mu_n \big ), \{\tau_{m,n} \}_{m \leq n} \big )$ is a projective system of abelian von Neumann algebras. By using the notion described in Proposition 3.14 of \cite{MJ1}, for each $m \leq n$, we have $\psi_{m, n}(L) = L \big |_{\mathcal{H}_m}$ for every $L \in \mathcal{M}_n$, and hence the following diagram commutes.
\begin{center}
\begin{tikzcd}[sep=huge]
\text{L}^\infty \big ( X_n, \Sigma_n, \mu_n \big ) \arrow[dd, "\tau_{m, n}"'] \; \; \arrow[rr, "\widehat{\Gamma}_n"] & & \; \;  \mathcal{M}_n \arrow[dd, "\psi_{m, n}"] \\
&& \\
\text{L}^\infty \big ( X_m, \Sigma_m, \mu_m \big ) \arrow[to path={node[midway,scale=3] {$\circlearrowright$}}] \; \; \arrow[rr, "\widehat{\Gamma}_m"']& & \; \;  \mathcal{M}_m 
\end{tikzcd}  
\end{center}

\noindent 
That is 
\begin{equation} \label{eq; condition 2}
\widehat{\Gamma}_m \circ \tau_{m, n} = \psi_{m, n} \circ \widehat{\Gamma}_n
\end{equation}
In particular, if $f \in \text{L}^\infty \big ( X_n, \Sigma_n, \mu_n \big )$ satisfying that $f \big |_{X_m} = 0$, then from Equation \eqref{eq; condition 2}, we get 
\begin{align*}
\widehat{\Gamma}_n(f) \big |_{\mathcal{H}_m} &= \psi_{m, n} \circ \widehat{\Gamma}_n (f) \\
&= \widehat{\Gamma}_m \circ \tau_{m, n} (f) \\
&= \widehat{\Gamma}_m \big (f \big |_{X_m} \big ) \\
&= 0.
\end{align*}
\end{proof}

\begin{note} \label{note; locally Hilbert space}
From Lemma \ref{lem;flhs}, we get a locally standard measure space $\big ( X, \Sigma, \mu \big )$ and a family $\{ \mathcal{D}_p \}_{p \in X}$ of locally Hilbert spaces. Then, consider the direct integral $\displaystyle \dilX \mathcal{D}_p \, \dmu$ of the family $\{ \mathcal{D}_p \}_{p \in X}$ of locally Hilbert spaces over the locally standard measure space $(X, \Sigma, \mu)$. For a fixed $n \in \mathbb{N}$, define 
\begin{equation} \label{eq; Hilbert space H n}
\mathcal{H}_n := \left \{ u \in \dilX \mathcal{D}_p \, \dmu\; \;  : \; \; \text{supp}(u) \subseteq X_n, \; \; u(p) \in \mathcal{H}_{n, p}, \; \; \text{almost every} \; \; p \in X_n \right \}.  
\end{equation} 
From Proposition \ref{prop;dilhs}, it is clear that $\mathcal{H}_n$ is a Hilbert space and moreover the family $\{ \mathcal{H}_n \}_{n \in \mathbb{N}}$ is a strictly inductive system of Hilbert spaces such that 
\begin{equation*}
\dilX \mathcal{D}_p \, \dmu = \bigcup\limits_{n \in \mathbb{N}} \mathcal{H}_n.
\end{equation*}
Next, we show that the given locally Hilbert space $\mathcal{D} = \bigcup\limits_{n \in \mathbb{N}} \mathcal{K}_n$ can be identified with $\displaystyle \dilX \mathcal{D}_p \, \dmu = \bigcup\limits_{n \in \mathbb{N}} \mathcal{H}_n$ through a bijective isometry. In this direction, we prove the following lemma.
\end{note}

\begin{lemma} \label{lem;i}
For every $n \in \mathbb{N}$, there exists an isometry from the Hilbert space $\mathcal{K}_n$ to $\mathcal{H}_n$.
\end{lemma}
\begin{proof}
We recall that $\mathcal{Z}_n$ is a countable dense subspace of the Hilbert space $\mathcal{K}_n$ for each $n \in \mathbb{N}$. For $i \geq 2$, we note that $(P_i(\mathrm{Id}_\mathcal{K} - P_{i-1})$ is the projection of the Hilbert space $\mathcal{K}$ onto $\mathcal{K}_i \ominus \mathcal{K}_{i -1 }$. Using this, for a fixed $n \in \mathbb{N}$, we give a map $W_n : \mathcal{Z}_n \rightarrow \mathcal{H}_n$ as; for $h_j \in \mathcal{Z}_n$, define $W_n(h_j) : X \rightarrow \bigcup\limits_{p \in X} \mathcal{D}_p$ by
\begin{equation} \label{eq; V n}
W_n(h_j)(p) := (P_i(\mathrm{Id}_\mathcal{K} - P_{i-1}) h_j) + \mathcal{N}_p \; \; \text{if} \; \; p \in X_i \setminus X_{i-1}.
\end{equation}
This map is defined on a countable dense subset of $\mathcal{K}_n$. To show that $W_n(h_j) \in \mathcal{H}_n$, we record the following observations. For $i \geq n+1$, if $h_j \in \mathcal{Z}_n$, then 
\begin{align*}
(P_i(\mathrm{Id}_\mathcal{K} - P_{i-1}) h_j) + \mathcal{N}_p = \mathcal{N}_p \in \mathcal{D}_{p}. 
\end{align*} 
Otherwise, for $i \leq n$, we have
\begin{align*}
(P_i(\mathrm{Id}_{\mathcal{K}} - P_{i-1}) h_j) + \mathcal{N}_p &= (P_i(h_j - P_{i-1}h_j )) + \mathcal{N}_p \\
&= (P_ih_j - P_iP_{i-1}h_j) + \mathcal{N}_p \\
&= (P_ih_j - P_{i-1}h_j)  + \mathcal{N}_p \in \mathcal{H}_{i, p} \subseteq \mathcal{H}_{n, p}.
\end{align*} 
These observations show that the $\text{supp}(W_n(h_j)) \subseteq X_n$, and $W_n(h_j)(p) \in \mathcal{H}_{n, p} \subseteq \mathcal{D}_p$ for almost every $p \in X_n$. 
Suppose $h_j, h_k \in \mathcal{Z}_n$ and $\lambda \in \mathbb{Q} + \iota \mathbb{Q}$, then $\lambda h_j + h_k \in \mathcal{Z}_n$, and the function $W_n(\lambda h_j + h_k) : X \rightarrow \bigcup\limits_{p \in X} \mathcal{D}_p$ is defined as; for $p \in X_i \setminus X_{i-1}$
\begin{align*}
W_n(\lambda h_j + h_k)(p) &= (P_i(\mathrm{Id}_\mathcal{K} - P_{i-1}) \lambda h_j + h_k) + \mathcal{N}_p \\
&= \lambda (P_i(\mathrm{Id}_\mathcal{K} - P_{i-1}) h_j) + \mathcal{N}_p + (P_i(\mathrm{Id}_\mathcal{K} - P_{i-1}) h_k) + \mathcal{N}_p \\
&= \lambda W_n(h_j)(p) + W_n(h_k)(p)
\end{align*}
This demonstrates that the map $W_n$ is linear on a countable, dense subspace $\mathcal{Z}_n$ of $\mathcal{K}_n$.

Next, we show that the map $W_n$ is an isometry on $\mathcal{Z}_n$. Fix $h_j \in \mathcal{Z}_n$, then we can choose the smallest $l \in \mathbb{N}$ such that $h_j \in \mathcal{K}_l$ (i.e. $h_j \notin \mathcal{K}_{l-1}$). Now we denote $I = \| W_n(h_j) \|^2$. Then we have
\begin{align*}
I =& \la W_n(h_j), W_n(h_j) \ra \\
=& \int_X \la W_n(h_j)(p), W_n(h_j)(p) \ra \, \dmu \\
=& \int_{X_n} \la W_n(h_j)(p), W_n(h_j)(p) \ra \, \mathrm{d} \mu_n(p)  \\
=& \sum_{i=1}^{n} \int_{X_i \setminus X_{i-1}} \la W_n(h_j)(p), W_n(h_j)(p) \ra \, \mathrm{d} \mu_i(p) \\
=& \sum_{i=1}^{l} \int_{X_i \setminus X_{i-1}} \la (P_i(\mathrm{Id}_\mathcal{K} - P_{i-1}) h_j) + \mathcal{N}_p, (P_i(\mathrm{Id}_\mathcal{K} - P_{i-1}) h_j) + \mathcal{N}_p \ra \, \mathrm{d} \mu_i(p) \\
=&  \sum_{i=1}^{l} \int_{X_i \setminus X_{i-1}} \la (P_ih_j - P_{i-1} h_j) + \mathcal{N}_p, (P_ih_j - P_{i-1} h_j) + \mathcal{N}_p \ra \, \mathrm{d} \mu_i(p) \\
=& I_1 - I_2 - I_3 + I_4,
\end{align*}
where
\begin{align*}
I_1 =  \sum_{i=1}^{l} \left ( \int_{X_i \setminus X_{i-1}} \la P_ih_j + \mathcal{N}_p, P_ih_j + \mathcal{N}_p \ra \, \mathrm{d} \mu_i(p) \right ) &= \sum_{i=1}^{l} \int_{X_i \setminus X_{i-1}} \phi_p(P_i h_j, P_i h_j) \, \mathrm{d} \mu_i(p) \\
&= \sum_{i=1}^{l} \int_{X_i \setminus X_{i-1}} f_i^{P_i h_j, P_i h_j}(p) \, \mathrm{d} \mu_i(p) \\
 &= \sum_{i=1}^{l} \big \langle P_i h_j, \; \widehat{\Gamma}_i\big  (\up{\chi}_{X_i \setminus X_{i-1}} \big )P_i h_j \big \rangle \\
 &=\big \langle h_j, \widehat{\Gamma}_l \big (\up{\chi}_{X_l} \big )h_j \big \rangle = \|h_j\|^2.
\end{align*}
Next, in the process of computing $I_2, I_3$ and $I_4$, we use Lemma \ref{lem; condition 2}. Thus
\begin{align*}
I_2 = \sum_{i=1}^{l} \left ( \int_{X_i \setminus X_{i-1}} \la P_ih_j + \mathcal{N}_p, P_{i-1} h_j + \mathcal{N}_p \ra \, \mathrm{d} \mu_i(p) \right) 
=& \sum_{i=1}^{l} \int_{X_i \setminus X_{i-1}} \phi_p ( P_i h_j, P_{i-1} h_j) \, \mathrm{d} \mu_i(p) \\
=& \sum_{i=1}^{l} \int_{X_i \setminus X_{i-1}} f_i^{P_i h_j, P_{i-1} h_j}(p) \, \mathrm{d} \mu_i(p) \\
=& \sum_{i=1}^{l} \la P_i h_j, \widehat{\Gamma}_i \big (\up{\chi}_{X_i \setminus X_{i-1}} \big ) P_{i-1} h_j \ra = 0,
\end{align*}    

\begin{align*}
I_3 = \sum_{i=1}^{l} \left ( \int_{X_i \setminus X_{i-1}} \la P_{i-1} h_j + \mathcal{N}_p, P_i h_j + \mathcal{N}_p \ra \, \mathrm{d} \mu_i(p) \right) 
=&  \sum_{i=1}^{l} \int_{X_i \setminus X_{i-1}} \phi_p ( P_{i-1} h_j, P_ih_j) \, \mathrm{d} \mu_i(p) \\
=&  \sum_{i=1}^{l} \int_{X_i \setminus X_{i-1}} f_i^{P_{i-1} h_j, P_i h_j} (p) \, \mathrm{d} \mu_i(p) \\
=&  \sum_{i=1}^{l} \la P_{i-1} h_j, \; \widehat{\Gamma}_i \big (\up{\chi}_{X_i \setminus X_{i-1}} \big )P_i h_j \ra = 0
\end{align*}
and
\begin{align*}
I_4 = \sum_{i=1}^{l} \left (  \int_{X_i \setminus X_{i-1}} \la P_{i-1} h_j + \mathcal{N}_p, P_{i-1} h_j + \mathcal{N}_p \ra \, \mathrm{d} \mu_i(p) \right) 
=&  \sum_{i=1}^{l} \int_{X_i \setminus X_{i-1}} \phi_p ( P_{i-1} h_j, P_{i-1} h_j) \, \mathrm{d} \mu_k(p)\\
=&  \sum_{i=1}^{l} \int_{X_i \setminus X_{i-1}} f_i^{P_{i-1} h_j, P_{i-1} h_j} (p) \, \mathrm{d} \mu_i(p) \\
=&  \sum_{i=1}^{l} \la P_{i-1} h_j, \widehat{\Gamma}_i \big (\up{\chi}_{X_i \setminus X_{i-1}} \big) P_{i-1} h_j \ra = 0.
\end{align*}
Hence, we get
\begin{equation*}
\| W_n(h_j) \|^2 = I = I_1 - I_2 - I_3 + I_4 = I_1 = \|h_j\|^2.
\end{equation*}$ $
Subsequently, we extend the map linearly to $\mathcal{K}_n$. Thus, $W_n$ defines an isometry on the Hilbert space $\mathcal{K}_n$. 
\end{proof}

Consider a measurable function $f: X \rightarrow \mathbb{C}$ with the property that $\text{supp}(f) \subseteq X_n$ and $f \big |_{X_n} \in \text{L}^\infty \big (X_n, \Sigma_n, \mu_n \big )$. Corresponding to the function $f$, we define the locally diagonalizable operator denoted by $T_f$, on the locally Hilbert space $\displaystyle \dilX \mathcal{D}_p \, \dmu$ (see (\ref{def;Diag(lbo)}) of Definition \ref{def;DecDiag(lbo)}). 

\begin{lemma} \label{lem;dls}
For each $n \in \mathbb{N}$, consider the subset 
\begin{equation*}
\mathcal{H}^0_n := \Big \{ T_fW_n(h_j) \; \; : \; \; f \in \text{EB}_\text{loc} \big( X, \Sigma, \mu \big), \; \; \text{supp}(f) \subseteq X_n, \; \;   h_j \in \mathcal{Z}_n \Big \}.
\end{equation*}  
Then  $\overline{\text{span}} \; \mathcal{H}^0_n = \mathcal{H}_n$.
\end{lemma}
\begin{proof}
Assume that $v \in \mathcal{H}_n$ and $\la v, \; T_fW_n(h_j) \ra = 0$ for all $T_fW_n(h_j) \in \mathcal{H}^0_n$. Now fix $f \in \text{EB}_\text{loc} \big( X, \Sigma, \mu \big)$ such that $\text{supp}(f) \subseteq X_n$ and also fix $h_j \in \mathcal{Z}_n$. Then we get 
\begin{equation*}
\int_X \la v(p), \; f(p)W_n(h_j)(p) \ra \, \dmu = \int_{X_n}  f(p) \la v(p), \; W_n(h_j)(p) \ra \, \mathrm{d} \mu_n(p) = 0.
\end{equation*} 
Here, note that the function $f \in \text{EB}_\text{loc} \big( X, \Sigma, \mu \big)$ with $\text{supp}(f) \subseteq X_n$ was arbitrarily chosen. Thus by using the fact that that dual of 
$\text{L}^1 \big (X_n, \Sigma_n, \mu_n \big )$  is $\text{L}^\infty \big (X_n, \Sigma_n, \mu_n \big )$, we get $\la v(p), W_n(h_j)(p) \ra = 0$ for almost every $p \in X_n$. This implies for almost every $p \in X_i \setminus X_{i-1}$, we have
\begin{align*}
\la v(p), \; W_n(h_j)(p) \ra &= \la v(p), \; (P_i(\mathrm{Id}_\mathcal{K} - P_{i-1}) h_j) + \mathcal{N}_p \ra = 0.
\end{align*} 
Here note that $h_j \in W_n$ was arbitrarily fixed. Hence we use the definition of $\mathcal{H}_{n, p}$ \big (see Equation \eqref{eq; Hilbert space H n p} with $p \in X_i \setminus X_{i-1}$ \big ) along with the definition of $\mathcal{H}_n$ \big (see Equation \eqref{eq; Hilbert space H n} \big ) to conclude that $v(p) = 0$ for almost every $p \in X_i \setminus X_{i-1}$ and thus for almost every $p \in X$. Therefore, $v = 0$. This proves that $\overline{\text{span}} \; \mathcal{H}^0_n = \mathcal{H}_n$.
\end{proof}

Next, by using Lemma \ref{lem;dls}, we show that the map $W_n : \mathcal{K}_n \rightarrow \mathcal{H}_n$ is surjective. 

\begin{lemma} \label{lem;s}
The map $W_n : \mathcal{K}_n \rightarrow \mathcal{H}_n$  is surjective.
\end{lemma}
\begin{proof}
Let $f \in \text{EB}_\text{loc} \big( X, \Sigma, \mu \big)$ such that $\text{supp}(f) \subseteq X_n$ and $h_j \in \mathcal{Z}_n$ be arbitrarily fixed. We choose the smallest $m \in \mathbb{N}$ such that $h_j \in \mathcal{K}_m$  \big (i.e. $h_j \notin \mathcal{K}_{m-1}$ \big). Observe that $\widehat{\Gamma}_{m} \big ({f \big|_{X_{m}}} \big )  h_j \in \mathcal{K}_m$. Since $\mathcal{Z}_m$ is dense in $\mathcal{K}_m$, for a fixed $\epsilon > 0$ there is $\widehat{h} \in \mathcal{Z}_m$ such that 
\begin{equation} \label{eq; epsilon}
\left \| \widehat{\Gamma}_{m} \big ({f \big|_{X_{m}}} \big )  h_j  - \widehat{h} \right \|_{\mathcal{K}_m}^2 < \epsilon,
\end{equation}
Let $l \in \mathbb{N}$ be the smallest such that $\widehat{h} \in \mathcal{K}_l$  \big (i.e. $\widehat{h} \notin \mathcal{K}_{l -1}$ \big). Using Equation \eqref{eq; V n}, we observe that $W_n \left(\widehat{h} \right ) \in \mathcal{H}_l$ and $T_fW_n(h_j) \in \mathcal{H}_m$. Here, we get $l \leq m$.  Now let us denote by
\begin{equation*}
I = \big \| T_fW_n(h_j) - W_n \big (\widehat{h} \big ) \big \|^2_{\mathcal{H}_m}.
\end{equation*}
Then we have
\begin{equation} \label{eq; integral}
I = I_1 -I_2 - I_3 + I_4,
\end{equation}
where 
\begin{flalign*}
I_1 =& \big \langle T_fW_n(h_j), T_fW_n(h_j) \big \rangle \\
=& \sum_{i=1}^{m} \int_{X_i \setminus X_{i-1}} |f(p)|^2 \la (P_ih_j - P_{i-1} h_j) + \mathcal{N}_p, (P_ih_j - P_{i-1} h_j) + \mathcal{N}_p \ra \, \mathrm{d} \mu_i(p) \\
= & \sum_{i=1}^{m} \left ( \int_{X_i \setminus X_{i-1}} |f(p)|^2 \phi_p(P_ih_j, P_ih_j) \, \mathrm{d} \mu_i(p) - \int_{X_i \setminus X_{i-1}} |f(p)|^2 \phi_p ( P_ih_j, P_{i-1} h_j) \, \mathrm{d} \mu_i(p) \right) \\
& + \sum_{i=1}^{m} \left ( - \int_{X_i \setminus X_{i-1}} |f(p)|^2 \phi_p ( P_{i-1} h_j, P_ih_j) \, \mathrm{d} \mu_i(p) + \int_{X_i \setminus X_{i-1}} |f(p)|^2 \phi_p ( P_{i-1} h_j, P_{i-1} h_j) \, \mathrm{d} \mu_i(p)\right) \\
= & \sum_{i=1}^{m} \left (  \int_{X_i \setminus X_{i-1}} |f(p)|^2 f_i^{P_i h_j, P_i h_j}(p) \, \mathrm{d} \mu_i(p) - \int_{X_i \setminus X_{i-1}} |f(p)|^2 f_i^{P_i h_j, P_{i - 1} h_j}(p) \, \mathrm{d} \mu_i(p) \right ) \\
& + \sum_{i=1}^{m} \left (  - \int_{X_i \setminus X_{i-1}} |f(p)|^2 f_i^{P_{i -1} h_j, P_i h_j}(p) \, \mathrm{d} \mu_i(p) + \int_{X_i \setminus X_{i-1}} |f(p)|^2 f_i^{P_{i - 1} h_j, P_{i - 1} h_j}(p) \, \mathrm{d} \mu_i(p) \right ) \\
= & \sum_{i=1}^{m} \big ( \la P_ih_j, \; \widehat{\Gamma}_i \big ({\up{\chi}_{X_i \setminus X_{i-1}}f^2} \big )P_ih_j \ra - \la P_ih_j, \; \widehat{\Gamma}_i \big ({\up{\chi}_{X_i \setminus X_{i-1}} f^2} \big ) P_{i-1}h_j \ra \big ) \\
& + \sum_{i=1}^{m} \big   ( - \la P_{i-1} h_j, \;  \widehat{\Gamma}_i \big ({\up{\chi}_{X_i \setminus X_{i-1}} f^2} \big )P_ih_j \ra + \la P_{i-1} h_j, \;  \widehat{\Gamma}_i \big ({\up{\chi}_{X_i \setminus X_{i-1}} f^2} \big ) P_{i-1} h_j \ra \big ) \\
= & \sum_{i=1}^{m} \left ( \la P_i h_j, \; \widehat{\Gamma}_i({\up{\chi}_{X_i \setminus X_{i-1}}f^2}) P_i h_j \ra - 0 - 0  + 0 \right ).
\end{flalign*}
Next,
\begin{flalign*}
I_2 =& \big \langle T_fW_n(h_j), W_n \big (\hat{h} \big ) \big \rangle \\
=& \sum_{i=1}^{l} \int_{X_i \setminus X_{i-1}} \overline{f(p)} \la (P_ih_j - P_{i-1} h_j) + \mathcal{N}_p, (P_i \hat{h} - P_{i-1} \hat{h}) + \mathcal{N}_p \ra \, \mathrm{d} \mu_i(p) \\
=& \sum_{i=1}^{l} \left ( \int_{X_i \setminus X_{i-1}} \overline{f(p)} \phi_p \big (P_ih_j, P_i \hat{h} \big ) \, \mathrm{d} \mu_i(p) - \int_{X_i \setminus X_{i-1}} \overline{f(p)} \phi_p \big (P_ih_j, P_{i-1}\hat{h} \big ) \, \mathrm{d} \mu_i(p) \right) \\
& +\sum_{i=1}^{l} \left ( - \int_{X_i \setminus X_{i-1}} \overline{f(p)} \phi_p \big (P_{i-1}h_j, P_i \hat{h} \big ) \, \mathrm{d} \mu_i(p) + \int_{X_i \setminus X_{i-1}} \overline{f(p)} \phi_p \big (P_{i-1}h_j, P_{i-1} \hat{h} \big ) \, \mathrm{d} \mu_i(p)\right) \\
= & \sum_{i=1}^{l} \left (  \int_{X_i \setminus X_{i-1}} \overline{f(p)} f_i^{P_i h_j, P_i \hat{h}}(p) \, \mathrm{d} \mu_i(p) - \int_{X_i \setminus X_{i-1}} \overline{f(p)} f_i^{P_i h_j, P_{i - 1} \hat{h}}(p) \, \mathrm{d} \mu_i(p) \right ) \\
& + \sum_{i=1}^{l} \left (  - \int_{X_i \setminus X_{i-1}} \overline{f(p)} f_i^{P_{i -1} h_j, P_i \hat{h}}(p) \, \mathrm{d} \mu_i(p) + \int_{X_i \setminus X_{i-1}} \overline{f(p)} f_i^{P_{i - 1} h_j, P_{i - 1} \hat{h}}(p) \, \mathrm{d} \mu_i(p) \right ) \\
= & \sum_{i=1}^{l} \left ( \la P_ih_j, \; \widehat{\Gamma}_i \big ({\up{\chi}_{X_i \setminus X_{i-1}} \overline{f}} \big )P_i \hat{h} \ra - \la P_ih_j, \; \widehat{\Gamma}_i \big ({\up{\chi}_{X_i \setminus X_{i-1}} \overline{f}} \big ) P_{i-1} \hat{h} \ra \right ) \\
& + \sum_{i=1}^{l} \left   ( - \la P_{i-1} h_j, \;  \widehat{\Gamma}_i \big ({\up{\chi}_{X_i \setminus X_{i-1}} \overline{f}} \big )P_i \hat{h} \ra + \la P_{i-1} h_j, \;  \widehat{\Gamma}_i \big ({\up{\chi}_{X_i \setminus X_{i-1}} \overline{f}} \big ) P_{i-1} \hat{h} \ra \right ) \\
= & \sum_{i=1}^{l} \left ( \la P_i h_j, \widehat{\Gamma}_i({\up{\chi}_{X_i \setminus X_{i-1}} \overline{f}}) P_i \hat{h} \ra - 0 - 0 + 0 \right).
\end{flalign*}
Next,
\begin{flalign*}
I_3 =& \big \langle W_n \big (\hat{h} \big ), T_fW_n(h_j) \big \rangle  \\
=& \sum_{i=1}^{l} \int_{X_i \setminus X_{i-1}} f(p) \la (P_i \hat{h} - P_{i-1} \hat{h}) + \mathcal{N}_p, (P_ih_j - P_{i-1} h_j) + \mathcal{N}_p \ra \, \mathrm{d} \mu_i(p) \\
=& \sum_{i=1}^{l} \left ( \int_{X_i \setminus X_{i-1}} f(p) \phi_p \big (P_i \hat{h}, P_ih_j \big ) \, \mathrm{d} \mu_i(p) - \int_{X_i \setminus X_{i-1}} f(p) \phi_p \big (P_i \hat{h}, P_{i-1}h_j \big ) \, \mathrm{d} \mu_i(p) \right) \\
& + \sum_{i=1}^{l}  \left ( - \int_{X_i \setminus X_{i-1}} f(p) \phi_p \big (P_{i-1} \hat{h}, P_i h_j \big ) \, \mathrm{d} \mu_i(p) + \int_{X_i \setminus X_{i-1}} f(p) \phi_p \big (P_{i-1} \hat{h}, P_{i-1}h_j \big ) \, \mathrm{d} \mu_i(p) \right) \\
= & \sum_{i=1}^{l} \left (  \int_{X_i \setminus X_{i-1}} f(p) f_i^{P_i \hat{h}, P_i h_j}(p) \, \mathrm{d} \mu_i(p) - \int_{X_i \setminus X_{i-1}} f(p) f_i^{P_{i} \hat{h}, P_{i - 1} h_j}(p) \, \mathrm{d} \mu_i(p) \right ) \\
& + \sum_{i=1}^{l} \left (  - \int_{X_i \setminus X_{i-1}} f(p) f_i^{P_{i -1} \hat{h}, P_{i} h_j}(p) \, \mathrm{d} \mu_i(p) + \int_{X_i \setminus X_{i-1}} f(p) f_i^{P_{i - 1} \hat{h}, P_{i - 1} h_j}(p) \, \mathrm{d} \mu_i(p) \right ) \\
= & \sum_{i=1}^{l} \left ( \la P_i \hat{h}, \; \widehat{\Gamma}_i \big ({\up{\chi}_{X_i \setminus X_{i-1}} f} \big )P_i h_j \ra - \la P_i \hat{h}, \; \widehat{\Gamma}_i \big ({\up{\chi}_{X_i \setminus X_{i-1}} f} \big ) P_{i-1}h_j \ra \right ) \\
& + \sum_{i=1}^{l} \left   ( - \la P_{i-1} \hat{h}, \;  \widehat{\Gamma}_i \big ({\up{\chi}_{X_i \setminus X_{i-1}} f} \big )P_i h_j \ra + \la P_{i-1} \hat{h}, \;  \widehat{\Gamma}_i \big ({\up{\chi}_{X_i \setminus X_{i-1}} f} \big ) P_{i-1} h_j  \ra \right ) \\
= & \sum_{i=1}^{l} \left ( \la P_i \hat{h}, \; \widehat{\Gamma}_i({\up{\chi}_{X_i \setminus X_{i-1}} f}) P_i h_j \ra - 0 - 0 + 0 \right).
\end{flalign*}
Finally,
\begin{flalign*}
I_4 =& \big \langle W_n \big (\hat{h} \big ), W_n \big (\hat{h} \big ) \big \rangle \\
=& \sum_{i=1}^{l} \int_{X_i \setminus X_{i-1}} \la \big (P_i \hat{h} - P_{i-1} \hat{h} \big ) + \mathcal{N}_p, \big (P_i \hat{h} - P_{i-1} \hat{h} \big) + \mathcal{N}_p \ra \, \mathrm{d} \mu_i(p) \\
=& \sum_{i=1}^{l} \left ( \int_{X_i \setminus X_{i-1}} \phi_p \big (P_i \hat{h}, P_i \hat{h} \big) \, \mathrm{d} \mu_i(p) - \int_{X_i \setminus X_{i-1}} \phi_p \big (P_i \hat{h}, P_{i-1} \hat{h} \big) \, \mathrm{d} \mu_i(p) \right) \\
& - \sum_{i=1}^{l} \left ( \int_{X_i \setminus X_{i-1}} \phi_p \big(P_{i-1} \hat{h}, P_i \hat{h} \big) \, \mathrm{d} \mu_i(p) + \int_{X_i \setminus X_{i-1}} \phi_p \big (P_{i-1} \hat{h}, P_{i-1} \hat{h} \big) \, \mathrm{d} \mu_i(p)  \right) \\
= & \sum_{i=1}^{l} \left (  \int_{X_i \setminus X_{i-1}}  f_i^{P_i \hat{h}, P_i \hat{h}}(p) \, \mathrm{d} \mu_i(p) - \int_{X_i \setminus X_{i-1}}  f_i^{P_i \hat{h}, P_{i - 1} \hat{h}}(p) \, \mathrm{d} \mu_i(p) \right ) \\
& + \sum_{i=1}^{l} \left (  - \int_{X_i \setminus X_{i-1}}  f_i^{P_{i -1} \hat{h}, P_i \hat{h}}(p) \, \mathrm{d} \mu_i(p) + \int_{X_i \setminus X_{i-1}} f_i^{P_{i - 1} \hat{h}, P_{i - 1} \hat{h}}(p) \, \mathrm{d} \mu_i(p) \right ) \\
= & \sum_{i=1}^{l} \left ( \la P_i\hat{h}, \; \widehat{\Gamma}_i \big (\up{\chi}_{X_i \setminus X_{i-1}}  \big )P_i \hat{h} \ra - \la P_i\hat{h}, \; \widehat{\Gamma}_i \big (\up{\chi}_{X_i \setminus X_{i-1}}  \big ) P_{i-1} \hat{h} \ra \right ) \\
& + \sum_{i=1}^{l} \left   ( - \la P_{i-1} \hat{h}, \;  \widehat{\Gamma}_i \big (\up{\chi}_{X_i \setminus X_{i-1}}  \big )P_i \hat{h} \ra + \la P_{i-1} \hat{h}, \;  \widehat{\Gamma}_i \big (\up{\chi}_{X_i \setminus X_{i-1}} \big ) P_{i-1} \hat{h} \ra \right ) \\
= & \sum_{i=1}^{l} \left (  \la  P_i \hat{h}, \;  \widehat{\Gamma}_i({\up{\chi}_{X_i \setminus X_{i-1}}}) P_i \hat{h} \ra - 0 - 0 + 0 \right). 
\end{flalign*}
Note that in the last equality while computing $I_t$ for $t = 1, 2, 3, 4$ we obtain zeros using Lemma \ref{lem; condition 2}.
Thus, by using Equation \eqref{eq; integral}, Lemma \ref{lem; condition 2} and the above computations, we get
\begin{align*}
I =&  \sum_{i=1}^{m} \la P_i h_j, \; \widehat{\Gamma}_i({\up{\chi}_{X_i \setminus X_{i-1}}f^2}) P_i h_j \ra - \sum_{i=1}^{l}  \la P_i h_j, \widehat{\Gamma}_i({\up{\chi}_{X_i \setminus X_{i-1}} \overline{f}}) P_i \hat{h} \ra   \\ 
& - \sum_{i=1}^{l} \la P_i \hat{h}, \; \widehat{\Gamma}_i({\up{\chi}_{X_i \setminus X_{i-1}} f}) P_i h_j \ra + \sum_{i=1}^{l}  \la  P_i \hat{h}, \;  \widehat{\Gamma}_i({\up{\chi}_{X_i \setminus X_{i-1}}}) P_i \hat{h} \ra  \\
=&  \sum_{i=1}^{m} \la P_i h_j, \; \widehat{\Gamma}_i({\up{\chi}_{X_i \setminus X_{i-1}}f^2}) P_i h_j \ra - \sum_{i=1}^{l}  \la P_i h_j, \widehat{\Gamma}_i({\up{\chi}_{X_i \setminus X_{i-1}} \overline{f}}) P_i \hat{h} \ra   \\ 
& - \sum_{i=1}^{l} \la P_i \hat{h}, \; \widehat{\Gamma}_i({\up{\chi}_{X_i \setminus X_{i-1}} f}) P_i h_j \ra + \sum_{i=1}^{l}  \la  P_i \hat{h}, \;  \widehat{\Gamma}_i({\up{\chi}_{X_i \setminus X_{i-1}}}) P_i \hat{h} \ra  \\
& - \sum_{i =l + 1}^{m} \la P_ih_j, \;  \widehat{\Gamma}_i({\up{\chi}_{X_i \setminus X_{i-1}} \overline{f}}) P_i\hat{h} \ra  - \sum_{i= l + 1}^{m} \la P_i \hat{h}, \; \widehat{\Gamma}_i({\up{\chi}_{X_i \setminus X_{i-1}} f}) P_i h_j \ra \\
= & \la h_j, \; \widehat{\Gamma}_{m}({f^2 \big |_{X_{m}}}) h_j \ra - \la h_j, \; \widehat{\Gamma}_{m}({ \overline{f} \big |_{X_{m}}}) \hat{h} \ra - \la \hat{h}, \; \widehat{\Gamma}_{m}({f \big |_{X_{m}}})  h_j \ra +  \la \hat{h}, \hat{h} \ra \\
= & \left \| \widehat{\Gamma}_{m}({f \big |_{X_{m}}}) h_j - \hat{h} \right \|^2.
\end{align*}

\noindent
Then Equation \eqref{eq; epsilon} implies that
\begin{align*}
I = \big \| T_fW_n(h_j) - W_n \big (\hat{h} \big ) \big \|^2_{\mathcal{H}_m}  < \epsilon.
\end{align*}
It follows that $W_n (\mathcal{Z}_n)$ is dense in $\mathcal{H}^0_n$ and so in span of $\mathcal{H}^0_n$. We know from Lemma \ref{lem;dls} that the span of $\mathcal{H}^0_n$ is dense in $\mathcal{H}_n$. Hence $W_n (\mathcal{Z}_n)$ is dense in $\mathcal{H}_n$. Finally by using the fact that $\mathcal{Z}_n$ is dense in $\mathcal{K}_n$, we conclude that the isometry  
$W_n$ is surjective for each $n \in \mathbb{N}$.
\end{proof}

Now we are ready to prove the main theorem. \\

\noindent \textbf{Proof of Theorem \ref{thm;dlhs}:} As discussed in Note \ref{note; locally Hilbert space}, $\displaystyle \dilX \mathcal{D}_p \, \dmu$ is a locally Hilbert space with the strictly inductive system $\big \{ \mathcal{H}_n \big \}_{n \in \mathbb{N}}$, where $\mathcal{H}_n$ is defined in Equation \eqref{eq; Hilbert space H n}. An appeal to Lemma \ref{lem;i} and Lemma \ref{lem;s}, for every $n \in \mathbb{N}$, the map $W_n : \mathcal{K}_n \rightarrow \mathcal{H}_n$ is unitary. Recall that if $h_j \in \mathcal{Z}_n$, then 
\begin{equation*}
W_n(h_j)(p) = (P_i(\mathrm{Id}_\mathcal{K} - P_{i-1}) h_j) + \mathcal{N}_p \; \; \text{if} \; \; p \in X_i \setminus X_{i-1}.
\end{equation*}
For the same $h_j \in \mathcal{Z}_n$, we see that 
\begin{equation*}
W_{n + 1}(h_j)(p) = (P_i(\mathrm{Id}_\mathcal{K} - P_{i-1}) h_j) + \mathcal{N}_p \; \; \text{if} \; \; p \in X_i \setminus X_{i-1}.
\end{equation*}
This shows that $W_{n + 1} \big |_{\mathcal{K}_n} = W_n$. Similarly, we see that $W_{r} \big |_{\mathcal{K}_n} = W_n$ for all $r \leq n$. In other words, $\big \{ W_n \big \}_{n \in \mathbb{N}}$ is a projective system of unitary operators. Now we define a locally bounded operator (see Equation \eqref{eq; inverese limit of bounded operators}) $W : \mathcal{D} \rightarrow \displaystyle \dilX \mathcal{D}_p \, \dmu$ by 
\begin{equation*}
W := \varprojlim_{n \in \mathbb{N}} W_n 
\end{equation*}
Moreover, $W$ is a bijective and local isometry. Thus, $\mathcal{D}$ is isomorphic to $\displaystyle \dilX \mathcal{D}_p \, \dmu$. On the other hand, for each $n \in \mathbb{N}$, the map $\tau_n : \text{EB}_\text{loc} \big( X, \Sigma, \mu \big) \rightarrow \mathcal{M}_n$ defined by 
\begin{equation*}
\tau_n(f) =  \widehat{\Gamma}_n \big ( f\big |_{X_n} \big )
\end{equation*}
is a normal homomorphism. Further, by using Equation \eqref{eq; condition 2}, for $m \leq n$, we have 
\begin{align*}
\psi_{m, n} \circ \tau_n(f) &= \psi_{m, n} \circ \widehat{\Gamma}_n \big ( f\big |_{X_n} \big ) \\
&= \widehat{\Gamma}_n \big ( f\big |_{X_n} \big ) \big |_{\mathcal{H}_m} \\
&= \widehat{\Gamma}_m \big ( f\big |_{X_m} \big ) \\
&= \tau_m(f)
\end{align*}
for every $f \in \text{EB}_\text{loc} \big( X, \Sigma, \mu \big)$. This shows that $\big ( \text{EB}_\text{loc} \big( X, \Sigma, \mu \big), \{ \tau_n \}_{n \in \mathbb{N}} \big )$ is compatibe with the projective system $\left ( \{ \mathcal{M}_n \}_{n \in \mathbb{N}},  \{\psi_{m,n} \}_{m \leq n} \right )$ of von Neumann algebras. So, by the uniquness of projective limit there exists a unique normal map $\tau : \text{EB}_\text{loc} \big( X, \Sigma, \mu \big) \rightarrow \mathcal{M}$ such that
\begin{equation*}
\psi_n \circ \tau = \tau_n
\end{equation*}
for every $n \in \mathbb{N}$. In fact, $\tau$ is can be defined as 
\begin{equation*}
\tau(f) = \varprojlim_{n \in \mathbb{N}} \widehat{\Gamma}_n \big (f \big |_{X_n} \big ).
\end{equation*}
Since $\widehat{\Gamma}_n$ is isomorphic for each $n \in \mathbb{N}$, we conclude that $\tau$ is bijective. Finally by using the fact that the abelian locally von Neumann algebra of all locally bounded diagonalizable operators on $\displaystyle \dilX \mathcal{D}_p \, \dmu$ is in one-to-one correspondence with $\text{EB}_\text{loc} \big( X, \Sigma, \mu \big)$ (see Equation \eqref{eq; EB loc}) via bijective map $\tau$ the result follows.

\subsection*{Declaration} 
The authors declare that there are no conflicts of interest.

\section*{Acknowledgement} 
The first named author kindly acknowledges the financial support received as an Institute postdoctoral fellowship from the Indian Institute of Science Education and Research Mohali. The first named author also sincerely thanks his Ph.D. thesis supervisor for introducing him to the topic of direct integral and disintegration of Hilbert spaces. The second named author would like to thank SERB (India) for a financial support in the form of Startup Research Grant (File No. SRG/2022/001795). The authors express their sincere thanks to DST for a financial support in the form of the FIST grant (File No. SR/FST/MS-I/2019/46(C)) and the Department of Mathematical Sciences, IISER Mohali for providing necessary facilities to carryout this work. 

\bibliographystyle{plain}

\end{document}